\documentclass [11pt]{amsart}
\usepackage {amsmath, amssymb, amscd, graphicx, color, url,epsf,enumitem}
\usepackage[all,knot,arc]{xy}
\usepackage[text={6.5in,9in},centering,letterpaper,dvips]{geometry}
\usepackage{hyphenat}
\usepackage{tikz}
\usepackage{comment}
\usepackage{color}

\newcommand{\RRing}{{\mathcal {R}}}
\newcommand{\CFm}{\mathrm{CF}^-}
\newcommand\CFa{\widehat{\mathrm{CF}}}

\newenvironment{prooff}{\par \noindent {\bf Proof}\ }{\hfill
$\square$\medskip \par}
        \def\sqr#1#2{{\vcenter{\hrule height.#2pt
                \hbox{\vrule width.#2pt height#1pt \kern#1pt
                \vrule width.#2pt}\hrule height.#2pt}}}
        \def\square{\mathchoice\sqr67\sqr67\sqr{2.1}6\sqr{1.5}6}

\newcommand{\vv}{v}
\newcommand\Mor{\mathrm{Mor}}

\newcommand{\defin}[1]{\textbf{#1}}
\newtheorem {theorem}{Theorem}[section]
\newtheorem {lemma}[theorem]{Lemma}
\newtheorem {prop}[theorem]{Proposition}
\newtheorem {corollary}[theorem]{Corollary}
\newtheorem {conjecture}[theorem]{Conjecture}
\newtheorem {definition}[theorem]{Definition}

\theoremstyle{remark}
\newtheorem {remark}[theorem]{Remark}
\newtheorem {example}[theorem]{Example}

\def\mathcenter#1{\vcenter{\hbox{$#1$}}}

{\makeatletter\gdef\reallynopagebreak{\nopagebreak\@nobreaktrue}}

\newcommand\goesupto{\nearrow}
\newcommand\goesdownto{\searrow}
\newcommand\Gen{\mathfrak{S}}
\newcommand\Sym{\mathrm{Sym}}
\newcommand\HFa{\widehat{HF}}
\newcommand\Z{\mathbb{Z}}

\newcommand\Field{\mathbb F}

\newcommand\Dual{\mathcal D}
\newcommand\Duality\Dual

\newcommand\HFKm{{\mathrm {HFK}}^-}
\newcommand\HFKa{\widehat{\mathrm{HFK}}}

\newcommand\ws{\mathbf w}
\newcommand\zs{\mathbf z}

\newcommand\relspinc{\underline{\spinc}}

\newcommand\x{\mathbf x}
\newcommand\w{\mathbf w}
\newcommand\z{\mathbf z}

\newcommand\y{\mathbf y}

\newcommand\ModSphere{\ModFlow\left({\mathbb S}\longrightarrow 
\Sym^{g-1}(\Sigma_{1})\times \Sym^2(\Sigma_{2})\right)}
\newcommand\ModSpheres\ModSphere

\newcommand\HFp{\HF^+}

\newcommand\HFm{\HF^-}
\newcommand\CFinf{{\rm {CF}}^\infty}

\newcommand\gr{\mathrm{gr}}
\newcommand\Mas{\mu}
\newcommand\UnparModSp{\widehat \ModSp}
\newcommand\UnparModFlow\UnparModSp
\newcommand\Mod\ModSp

\newcommand{\spinc}{\mathfrak s}

\newcommand\ModMaps{\mathcal M}
\newcommand\ModSp\ModMaps

\newcommand\Ta{{\mathbb T}_{\alpha}}
\newcommand\Tb{{\mathbb T}_{\beta}}

\newcommand\Torus{\mathbb{T}}

\newcommand\alphas{\mbox{\boldmath$\alpha$}}

\newcommand\betas{\mbox{\boldmath$\beta$}}

\newcommand\Ring{\mathcal R}

\newcommand\CFLm{\mathrm{CFL}^-}

\newcommand\spincrel\relspinc

\newcommand\CFK{\mathrm{CFK}}

\newcommand\HFK{\mathrm{HFK}}

\newcommand\CFKa{\widehat\CFK}
\newcommand\CFKm{\CFK^-}
\newcommand\fCFKm{{\mathcal{CFK}}^-}
\newcommand\fCFKinf{{\mathcal{CFK}}^\infty}
\newcommand\CFKinf{\CFK^{\infty}}

\newcommand\HFKinf{\HFK^{\infty}}

\newcommand{\Tors}{\mathrm{Tors}}

\newcommand{\HF}{\mathrm{HF}}

\def\endproof{\relax\ifmmode\expandafter\endproofmath\else
  \unskip\nobreak\hfil\penalty50\hskip.75em\hbox{}\nobreak\hfil\bull
  {\parfillskip=0pt \finalhyphendemerits=0 \bigbreak}\fi}
\def\endproofmath$${\eqno\bull$$\bigbreak}
\def\bull{\vbox{\hrule\hbox{\vrule\kern3pt\vbox{\kern6pt}\kern3pt\vrule}\hrule}}

\newcommand{\OneHalf}{\frac{1}{2}}

\newcommand{\ModSWfour}{\mathcal{M}}
\newcommand{\ModFlow}{\ModSWfour}

\newcommand\abuts\Rightarrow

\newcommand\Alg{\mathbb{A}}

\newcommand\Concord{\mathcal {C}}
\newcommand\Cont{\mathit {Cont}}
\newcommand\CFKt{\mathrm{tCFK}}
\newcommand\HFKt{\mathrm{tHFK}}
\newcommand\SpinC{\mathrm{Spin}^c}
\renewcommand{\HFa}{\widehat {\mathrm{HF}}}

\begin{document}

\title{Concordance homomorphisms from knot Floer homology}

\author[Peter S. Ozsv\'ath]{Peter S. Ozsv\'ath}
\thanks {P. Ozsv{\'a}th was partially supported by NSF DMS-1258274}
\address {Department of Mathematics, Princeton Unversity\\ Princeton,
  New Jersey 08544}
\email {petero@math.princeton.edu}

\author[Andr{\'a}s I. Stipsicz]{Andr{\'a}s I. Stipsicz} 
\thanks {A. Stipsicz was partially supported by the
  Lend{\"u}let program, ERC Advanced Grant LDTBud and OTKA NK81203}
\address {MTA R\'enyi Institute of Mathematics, Budapest, Hungary}
\email {stipsicz@renyi.hu}

\author[Zolt{\'a}n Szab{\'o}]{Zolt{\'a}n Szab{\'o}}
\thanks{Z. Szab{\'o} was partially supported by NSF DMS-1006006,
NSF DMS-1309152 and NSF DMS-1606571}
\address{Department of Mathematics, Princeton University\\ Princeton,
  New Jersey 08544} 
\email {szabo@math.princeton.edu}

\begin {abstract} 
  We modify the construction of knot Floer homology to produce a
  one-parameter family of homologies $\HFKt$  for knots in
  $S^3$. These invariants can be used to give homomorphisms from the
  smooth concordance group $\Concord$ to $\Z$, giving bounds on the
  four-ball genus and the concordance genus of knots. 
  We give some applications of these homomorphisms.
\end {abstract}


\maketitle
\newcommand\Q{\mathbb Q}
\newcommand\Conc{\mathit{Conc}}
\section{Introduction}
\label{sec:intro}

The signature of the symmetrized Seifert matrix gives a knot invariant
$\sigma(K)$ satisfying a number of basic properties~\cite{Murasugi}:
it is additive under connected sums, it changes in a controlled manner
under crossing changes, and it gives a lower bound on the genus of a
slice surface. Levine and Tristram~\cite{Tristram} extend this
  invariant to a one-parameter family of knot invariants
  $\sigma_{\omega}$ indexed by points $\omega$ on the unit circle.
  More recently, knot invariants whose properties are similar to those
  of $\sigma$ have been constructed using techniques such as knot
  Floer homology, resulting in the
  invariant~$\tau(K)$~\cite{FourBall,RasmussenThesis}; and Khovanov
  homology, resulting in Rasmussen's $s$ invariant~\cite{Jake}.  While
  $\sigma$ and $\sigma_\omega$ bound the topological slice genus, the
  newer invariants often give better bounds for the smooth
  slice genus.

The goal of the present paper is to use methods of knot Floer
  homology to construct a one-parameter family of knot invariants
  $\{\Upsilon_K(t)\}_{t\in [0,2]}$, \emph{upsilon of $K$ at $t$},
  which fit together to give a real-valued function $\Upsilon _K\colon
  [0,2]\to {\mathbb {R}}$. These invariants are additive under
  connected sums, they behave in a controlled manner under crossing
  changes, and they give lower bounds on the smooth slice genus.  This invariant
is extracted from the filtered knot Floer complex, and it is similar
to, and indeed inspired by, the work of Jen
Hom~\cite{JenHomInfinitelyGenerated}.  (For a comparison of $\Upsilon
$ to~\cite{JenHomInfinitelyGenerated}, see
Section~\ref{sec:furtherformal}.)

The invariants $\{\Upsilon_K(t)\}_{t\in [0,2]}$ are extracted from a
suitably modified variant of knot Floer homology~\cite{OSKnots,
  RasmussenThesis}. Recall that knot Floer homology is defined as the
homology of a bigraded chain complex over the base ring $\Field [U]$,
the ring of polynomials over the field $\Field $ of two elements.  (In
the following we will use coefficients in $\Field [U]$, although, with
the appropriate use of signs, the constructions and results admit
extensions to give invariants over $\Z [U]$.)  This chain complex is
associated to a doubly pointed Heegaard diagram representing the knot $K$ 
(equipped with some orientation). Denote the two basepoints by $w$ and
$z$. The generators of the knot Floer complex over $\Field [U]$ are
given combinatorially from the Heegaard diagram, the differential
of the complex counts pseudo-holomorphic disks that do not
cross $z$, while the exponent of $U$ records the multiplicity
with which the pseudo-holomorphic disk crosses $w$. The complex is also
equipped with a pair of gradings, the {\em Maslov grading} $M$
and the {\em Alexander grading} $A$, which descends to homology, endowing
$\HFKm (K)$ with the structure of a bigraded $\Field
       [U]$-module.

The above construction of $\HFKm (K)$ admits the following
variation. Fix a rational number $t\in [0,2]\cap \Q$, and let
$t=\frac{m}{n}$ where $m$ and $n$ are relatively prime integers.
The modified complex
is defined over the polynomial algebra in 
$\vv^{1/n}$.  The generators of the modified theory are the same as
those in the traditional knot Floer complex; and there
is a single grading, now by a rational number, given by $M-t\cdot A$.  The exponent of $\vv$ records $(2-t)$ times the
multiplicity with which the disk crosses $w$ and $t$ times the
multiplicity with which it crosses $z$.  
Multiplication by $\vv$ drops grading by $1$, as does the
differential.  The modified theory provides a family $\HFKt (K) $
($t\in [0,2]\cap \Q$) of \emph{$t$-modified knot Floer homologies},
which is a $\Q$-graded module over the polynomial algebra in $\vv^{1/n}$.
\begin{theorem}
  \label{thm:Invariance}
  For all rational $t=\frac{m}{n}\in [0,2]$ the $t$-modified knot
  Floer homology $\HFKt(K)$, thought of as a graded $\Field
  [\vv^{1/n}]$-module, is an invariant of the  knot $K$. 
\end{theorem}

A homology class $\xi$ is said to be {\em homogeneous} if it is
represented by a cycle in a fixed grading. It is called {\em
  non-torsion} if $\vv^d \xi\neq 0$ for all $d\in
  \frac{1}{n}\Z$.  We define the invariant $\Upsilon_K(t)$ to be the
  maximal degree of any homogeneous, non-torsion homology class in
  $\HFKt(K)$.  It follows immediately from
  Theorem~\ref{thm:Invariance} that $\Upsilon_K(t)$ is also a knot
  invariant (see Corollary~\ref{cor:InvariantUps} below).

\subsection{The behaviour of $\Upsilon_K(t)$ as a function of $t$}
\label{sec:UpsilonOfT}

The function $\Upsilon_K$ satisfies the following symmetry:

\begin{prop}
  \label{prop:SymmUps}
  For any knot $K$, $\Upsilon_K(t)=\Upsilon_K(2-t)$.
 \end{prop}

$\Upsilon$ also satisfies the following integrality properties
(compare also Proposition~\ref{prop:DeltaDerivative} below).

\begin{prop}
  \label{prop:Integrality}
  The quantity $\Upsilon_K(\frac{m}{n})$ lies in $\frac{1}{n} \Z$.
\end{prop}

Indeed, the above definition of $\HFKt$ and 
$\Upsilon _K(t)$ for rational $t$ can be extended to any real
$t\in [0,2]$, giving a knot invariant
$\Upsilon _K\colon [0,2]\to {\mathbb {R}}$, with the following properties:

\begin{prop}
  \label{prop:FinitelyManySlopes}
  For any knot $K$, the function $\Upsilon_K$ (defined on $[0,2]\cap
  {\mathbb {Q}}$) has a continuous extension to a real-valued
  function on $[0,2]$, which is a piecewise linear function of $t$,
  and whose derivative has finitely many discontinuities.  Each slope
  is equal to some Alexander grading $s$ for which
  $\HFKa_*(K,s)\neq 0$; hence, in particular, each slope is an
  integer.
\end{prop}
The following two propositions determine the behaviour of
$\Upsilon _K$ near $0$ (and so near $2$, in view of Proposition~\ref{prop:SymmUps}).

\begin{prop}
  \label{prop:BoundaryCases}
  $\Upsilon_K(0)=0$.
\end{prop}

\begin{prop}
  \label{prop:SlopeUpsilon}
  The slope of $\Upsilon_K(t)$ at $t=0$ is given by $-\tau(K)$, where
  $\tau(K)$ denotes the concordance invariant of the knot
  $K$ defined from the knot Floer homology module $\HFKm (K)$.
\end{prop}

Sometimes it is convenient to consider discontinuities of the
derivative of $\Upsilon_K(t)$. To this end, let
\[ 
\Delta \Upsilon_K'(t_0) = 
\lim_{t\goesdownto t_0} \Upsilon'_K(t)
-
\lim_{t\goesupto t_0} \Upsilon'_K(t).
\] 
Note that by Proposition~\ref{prop:SlopeUpsilon} the quantity
$\Delta\Upsilon'_K$ and $\tau$ together determine $\Upsilon_K$:
\[ \Upsilon_K(t) = -\tau(K) \cdot t + \sum_{0<s<t} \Delta
\Upsilon_K'(s) \cdot (t-s).\]

For the function $\Upsilon _K \colon [0,2]\to {\mathbb {R}}$ we have
the following extension of Proposition~\ref{prop:Integrality}:

\begin{prop}
  \label{prop:DeltaDerivative}
  For any $t\in  [0,2]$, 
  $t \cdot \Delta\Upsilon_K'(t)$ is an even integer.
\end{prop} 

Propositions~\ref{prop:SymmUps} and~\ref{prop:Integrality} are proved
in Section~\ref{sec:Formal}; 
Propositions~\ref{prop:FinitelyManySlopes},~\ref{prop:BoundaryCases},
and~\ref{prop:SlopeUpsilon} are proved in Section~\ref{sec:varying}.

\subsection{Topological properties of $\Upsilon_K (t)$}
\label{intro:TopologicalProperties}

Topological properties of $\Upsilon_K(t)$ follow from corresponding properties
of knot Floer homology:

\begin{prop}
  \label{prop:AddUpsilon}
  $\Upsilon_K$ is additive under connected sum of knots; i.e.
  \[ \Upsilon_{K_1\# K_2}(t)=\Upsilon_{K_1}(t)+\Upsilon_{K_2}(t).\]
\end{prop} 

\begin{prop}
\label{prop:mirror}
Let $m(K)$ denote the mirror of the knot $K$. Then 
\[
\Upsilon _{m(K)}(t)=-\Upsilon _K (t).
\]
\end{prop}

The invariant $\Upsilon_K$ changes in a controlled manner under crossing
changes:

\begin{prop}
  \label{prop:CrossChangeUpsilon}
  Let $K_+$ and $K_-$ be two knots which differ in a crossing change.
  Then, for $0\leq t \leq 1$ we have that
   \[ \Upsilon_{K_+}(t)\leq \Upsilon_{K_-}(t)\leq \Upsilon_{K_+}(t)+t.\]
\end{prop}

For $1\leq t \leq 2$, symmetry and the above inequality implies that
  \[ \Upsilon_{K_+}(t)\leq \Upsilon_{K_-}(t)\leq \Upsilon_{K_+}(t)+(2-t).\]

The invariant $\Upsilon _K (t)$ also provides a lower bound
for the (smooth)  \emph{slice genus} $g_s(K)$ of the
knot $K$ as follows:

\begin{theorem}
  \label{thm:BoundSliceGenus}
  The invariants $\Upsilon_{K}(t)$ bound the slice genus of $K$; more
  precisely, for $0\leq t \leq 1$, 
  \[ |\Upsilon_K(t)|\leq t\cdot g_s(K).\]
\end{theorem}

The bounds on the slice genus are no stronger than the bounds coming
from Rasmussen's ``local $h$ invariants''~\cite{RasmussenThesis}, see
also~\cite{HomWu, JakeOtherPaper};  in fact, the slice bounds are
proven by bounding $\Upsilon _K$ in terms of $h$ invariants (see
Proposition~\ref{prop:BoundByNu} below). 
The bounds based on  $\Upsilon_K(t)$ are
convenient, though, as they come from homomorphisms:

\begin{corollary}
  \label{cor:UpsConcHomom}
  For each fixed $t$, the map $K\mapsto \Upsilon_K(t)$ gives a
  homomorphism from the (smooth) knot concordance group
  $\Concord$ to ${\mathbb R}$; indeed, $\Upsilon_K(t)$ induces a
  homomorphism $\Upsilon \colon \Concord \to \Cont([0,2])$ from the
  concordance group $\Concord$ to the vector space of continuous
  functions on $[0,2]$. \qed
\end{corollary}

\begin{proof}
  It follows from Theorem~\ref{thm:BoundSliceGenus} and
  Proposition~\ref{prop:AddUpsilon} that if $K_1$ and $K_2$ are
  concordant, then $\Upsilon_{K_1}=\Upsilon_{K_2}$; i.e. $\Upsilon$ is
  a well-defined function on the concordance
  group. Proposition~\ref{prop:AddUpsilon} now implies that it is a
  homomorphism.
\end{proof}

In a different direction, recall that the {\em{concordance genus}} of
$K$, written $g_c(K)$, is the minimal Seifert genus of any knot $K'$
which is concordant to $K$.  The invariant $\Upsilon_K$ can
be used to bound this quantity, according to the following:

\begin{theorem}
  \label{thm:BoundConcordanceGenus}
  Let $s$ denote the maximum of the finitely many slopes appearing in the graph of $\Upsilon_K(t)$
  (c.f. Proposition~\ref{prop:FinitelyManySlopes}).
  Then, 
  \[ s\leq g_c(K).\]
\end{theorem}

It is interesting to compare this result to~\cite{HomConcordanceGenus}. 

Propositions~\ref{prop:AddUpsilon},~\ref{prop:mirror},~\ref{prop:CrossChangeUpsilon}
and Theorem~\ref{thm:BoundSliceGenus} are all proved in
Section~\ref{sec:Formal}.
\subsection{Calculations}

The invariant $\Upsilon _K$ can be explicitly computed for some classes of
knots. For alternating knots we have

\begin{theorem}
  \label{thm:AltKnots}
  Let $K$ be an alternating knot (or, more generally, a
  quasi-alternating one) with signature $\sigma$. Then,
  \[\Upsilon_K(t)=(1-|t-1|)\frac{\sigma}{2}.\] 
  In particular, the derivative of $\Upsilon_K(t)$ has at most one
  discontinuity, which can occur at $t=1$.
\end{theorem}

The knot Floer homology of torus knots was determined
in~\cite{NoteLens}. These computations lead to the following
computation of their $\Upsilon_K$ invariant, which can be phrased
purely in terms of their Alexander polynomial.  If $K=T_{p,q}$ is the
$(p,q)$ torus knot (where $p$ and $q$ are positive, relatively prime
integers), then the nonzero coefficients in the Alexander polynomial
$\Delta _K(t)$ are all $\pm 1$, and they alternate in sign.  Write the
Alexander polynomial of $K$ as
\begin{align*}
  \Delta_{K}(t)&= t^{-\frac{pq -p
      -q+1}{2}}\frac{(t^{pq}-1)(t-1)}{(t^p-1)(t^q-1)}=
  \\ &=\sum_{k=0}^n (-1)^k t^{\alpha_k},
\end{align*}
where $\{\alpha_k\} _{k=0}^n$ is a decreasing sequence of integers.
Consider a corresponding sequence $\{ m_k\}_{k=0}^n$ of integers,
defined inductively by the formulae
  \begin{align*}
    m_{0}&=0 \\
    m_{2k} &= m_{2k-1}-1 \\
    m_{2k+1} &= m_{2k}-2(\alpha_{2k}-\alpha_{2k+1})+1.
  \end{align*}

From these integers the invariant $\Upsilon _K(t)$ is computed
by the following formula:

\begin{theorem}
  \label{thm:TorusKnots}
  Let $K$ be a positive torus knot, and let
  $\{m_k,\alpha_k\}_{k=0}^n$
  be the above sequences extracted from its Alexander polynomial.
  Then, 
  \[
  \Upsilon _K (t)=\max _{\{i\big|0\leq 2i\leq n\}} \{ m_{2i}-t\alpha _{2i}\} .
  \]
\end{theorem}

In fact, we will prove a more general analogue of the above theorem
(Theorem~\ref{thm:LSpaceKnotsUpsilon}),
which holds for any knot on which some positive surgery gives an
$L$-space, in the sense of~\cite{NoteLens}.

\begin{example}
  \label{ex:T34}
  Let $K=T_{3,4}$ be the $(3,4)$ torus knot. Since
  $\Delta_K(t)=t^3-t^2+1-t^{-2}+t^{-3}$, the function
 $\Upsilon_{K}(t)$ is given by 
  \[ \Upsilon_K(t)= \left\{\begin{array}{ll}
      -3t & {t\in [0,\frac{2}{3}]} \\
        -2 & {t\in [\frac{2}{3},\frac{4}{3}]} \\
          -6+3 t & {t\in[\frac{4}{3},2]}.
      \end{array}
      \right.\]
\end{example}

Theorems~\ref{thm:AltKnots} and~\ref{thm:TorusKnots} 
are proved in Section~\ref{sec:Calcs}.
For an inductive formula computing $\Upsilon _{T_{p,q}}$ in terms
of the functions $\Upsilon _{T_{n,n+1}}$ see \cite{FellKrat}.

\subsection{Applications of $\Upsilon$ to the concordance group}
\label{subsec:Consequences}

Partially computing $\Upsilon _K$ for an infinite family of torus knots, 
we get

\begin{theorem}
  \label{thm:Independence}
  The function 
  \[ K\mapsto  
  \left(\frac{1}{n}
  \Delta\Upsilon'_K(\frac{2}{n})\right)_{n=2}^{\infty}\] from the
  concordance group $\Concord$ to $\Z
  ^{\infty}=\bigoplus_{n=2}^{\infty} \Z$ is surjective.
\end{theorem}

\begin{remark}
  Implicit in the above theorem is the statement that (a) for any knot
  $K$, the invariant $\Delta \Upsilon'_K(\frac{2}{n})$ is divisible by $n$ (as
  a consequence of Proposition~\ref{prop:DeltaDerivative}), and that
  (b) for a knot $K$ the value
$\Delta\Upsilon'_K(\frac{2}{n})$ is non-zero for
  only finitely many $n$ (which follows from
  Proposition~\ref{prop:FinitelyManySlopes}).
\end{remark}

Theorem~\ref{thm:Independence} then easily implies the (well-known)
existence of a direct summand of $\Concord$ isomorphic to $\Z
^{\infty}$~\cite{Litherland}.

By examining discontinuities of $\Upsilon _K'$, 
Theorem~\ref{thm:AltKnots} and Example~\ref{ex:T34} have the following
immediate corollary (which indeed can be seen by other means, as
well):

\begin{corollary}
In the smooth concordance group $\Concord$ the torus knot $T_{3,4}$
is linearly independent from all alternating knots. \qed
\end{corollary}

The Levine-Tristram signature function is a powerful tool for studying the
concordance group; see for instance~\cite{Litherland,Tristram}. However,
$\Upsilon_K(t)$ can also be used to study knots for which such
topological methods yield no information: using $\Upsilon $ we can
prove results 
for the subgroup $\Concord _{TS}\subset \Concord$ given by
\emph{topologically slice} knots, while the Levine-Tristram signature function
vanishes on this subgroup. We illustrate this phenomenon by reproving
a recent result of J. Hom~\cite{JenHomInfinitelyGenerated}, which states
that $\Concord _{TS}$ admits a direct summand isomorphic to
$\Z^{\infty}$.

For a given knot $K$, let $W_0^+(K)$ denote its untwisted positive
Whitehead double; and let $C_{p,q}(K)$ denote its $(p,q)$ cable (for
$p$ and $q$ relatively prime). Consider the family of knots
\begin{equation}\label{eq:TheKnots}
K_n=C_{n,2n-1}( W_0^+(K))\# (-T_{n,2n-1}).
\end{equation}
Observe that $K_n$ are topologically slice: by a theorem of
  Freedman~\cite{Freedman} the knot $W_0^+(K)$ is topologically slice, hence
  the cable $C_{n,2n-1}( W_0^+(K))$ is topologically concordant to the
  same cable of the unknot, consequently $K_n$ is topologically slice.  The
partial computation of $\Upsilon _{K_n}$, and the same map as used in
Theorem~\ref{thm:Independence}, now yields the following:

\begin{theorem}
  \label{thm:IndependenceTopSlice}
  Consider the topologically slice knots $\{K_n\}_{n=2}^{\infty}$
  given in \eqref{eq:TheKnots}. These form a basis for a free direct
  summand of the subgroup $\Concord _{TS}$ of the concordance group
  given by topologically slice knots.  In fact, the map $\Concord\to
  \bigoplus_{n=2}^{\infty} \Z$ defined by
  \[ K\mapsto  \left(\frac{1}{2n-1}
    \Delta\Upsilon'_K(\frac{2}{2n-1})\right)_{n=2}^{\infty}\] maps the
    span of $\{K_n\}_{n=2}^{\infty}$ isomorphically onto $\Z
    ^{\infty}=\bigoplus_{n=2}^{\infty}\Z$.
\end{theorem}

\begin{remark}
The fact that the group $\Concord _{TS}$
 of topologically slice knots contains a
$\Z^{\infty}$ direct summand was first proved by Jen Hom
in~\cite{JenHomInfinitelyGenerated}. Her examples are very similar to
the ones we have given here: only the cabling parameters are
different. (We chose our parameters out of convenience for our
computations.)  Her homomorphisms also use the knot Floer complex, but
they appear to use it differently from ours; see especially
Proposition~\ref{prop:NotHom} below.
\end{remark}

\begin{remark}
  Corollary~\ref{cor:ConcordanceGenusFamily} provides a refinement of
  Theorem~\ref{thm:IndependenceTopSlice}, giving lower bounds
  on the concordance genera of linear combinations of the $K_n$.
  Compare also~\cite{Chen} for a generalization of the above result.
\end{remark}

\subsection{Outline of the paper}
In Section~\ref{sec:HFK}, we review some notation from knot Floer
homology, as well as some of its key results.  In
Section~\ref{sec:DefUpsilon}, we spell out the definition of
$\Upsilon_K(t)$ in more detail, extracted from $t$-modified knot Floer
homology.  Invariance of the $t$-modified theory is seen as a special
case of a formal construction described in Section~\ref{sec:Formal}.
The behaviour of $\Upsilon_K(t)$ as a function of $t$ is studied in
Section~\ref{sec:varying}, where we also verify
Proposition~\ref{prop:FinitelyManySlopes}.  In
Section~\ref{sec:Calcs}, we give some computations, verifying the
computations for alternating and torus knots.  In
Section~\ref{sec:Bordered}, we recall the essentials of bordered Floer
homology, which will be used in the computations from
Section~\ref{sec:LinIndep}, where we prove
Theorem~\ref{thm:IndependenceTopSlice}.  In
Section~\ref{sec:furtherformal}, we compare the homomorphism
$\Upsilon_K(t)$ with those arising from the work of
Hom~\cite{JenHomInfinitelyGenerated}.  Finally, in
Section~\ref{sec:Links}, we give a generalization to the case of links.

\subsection{Further remarks and questions}

Note that $t$-modified knot Floer homology has a special
behaviour when we specialize to $t=1$.  In that case, one can
associate moves to {\em unoriented} saddles. This will be further
pursued in~\cite{Unorient}.

The results from Section~\ref{sec:UpsilonOfT} can be thought of as
giving linear relations between the values of $\Upsilon_K$ at
various values of $t$: $\Upsilon_K(t)=\Upsilon_K(2-t)$ and
$\Upsilon_K(0)=0$. It is natural to wonder if there are any further
linear relations between the various values of
$\Upsilon_K(t)$ for $t\in [0,1]$. We conjecture that there are none.

More explicitly, for each rational number $t$,  consider the 
homomorphism
\[ \phi_{t}\colon \Concord \to \Z \]
defined in terms of the expression $t=\frac{m}{n}$ where $m$ and $n$ are relatively prime integers by
\[ \phi_{\frac{m}{n}}(K)=\left\{\begin{array}{ll}
    \frac{1}{2n} \Delta\Upsilon'_K(\frac{m}{n}) & {\text{if $m$ is odd}} \\ \\
    \frac{1}{n} \Delta\Upsilon'_K(\frac{m}{n}) & {\text{if $m$ is even.}}
  \end{array}\right.\]
\begin{conjecture}
  The map $K\mapsto (\phi_t(K))_ {\{t\in\Q\big| 0<t<1\}}$, where
  $t=\frac{m}{n}$ with  $(m,n)$  relatively prime
  integers, induces a surjection onto $\bigoplus_{\{t\in\Q\big|
    0<t<1\}} \Z$.
\end{conjecture}

A more challenging variant of the above conjecture is the following:

\begin{conjecture}
  The map $K\mapsto (\phi_t(K))_{\{t\in\Q\big| 0<t<1\}}$, where
  $t=\frac{m}{n}$ with $(m,n)$ relatively prime integers,
  induces a surjection from the subgroup $\Concord _{TS}$ of
  topologically slice knots onto $\bigoplus_{\{t\in\Q\big| 0<t<1\}}
  \Z$.
\end{conjecture}

It is natural to wonder what the image of the above map is, when
further restricted to knots with $\epsilon=0$, in the sense
of~\cite{JenHomInfinitelyGenerated}; in particular, for those knots
which are in the kernel of Hom's
homomorphisms~\cite{JenHomInfinitelyGenerated}. (For a brief
discussion about $\epsilon$, see Section~\ref{sec:furtherformal}.)


The limitations of $\Upsilon_K(t)$ become apparent when we consider
alternating knots: Theorem~\ref{thm:AltKnots} can be interpreted as
saying that the span of all alternating knots has a one-dimensional
image under $\Upsilon_K$. By contrast, alternating torus knots
$T_{2,2n+1}$ are linearly independent in $\Concord$; more generally, a
theorem of Litherland~\cite{Litherland} states that all torus knots
are linearly indepedent in the concordance group. These limitations
notwithstanding, it seems likely that one can get more information by
pushing the present techniques further. For instance, in the spirit
of~\cite{CassonGordon}, one can consider branched covers to construct
further invariants. The simplest of these branched covers is the
double branched cover $\Sigma(K)$ of a knot $K\subset S^3$, which is a
rational homology sphere. The branch locus forms a null-homologous
knot in $\Sigma(K)$. It would be natural to consider an analogue of
$\Upsilon$ in that double branched cover to try to get further
concordance information.

According to a recent result of Jen Hom~\cite{HomEpsilonUpsilon},
there are knots for which $\Upsilon\equiv 0$, 
but her invariant $\epsilon$ is non-zero.

\bigskip

{\bf Acknowledgements:} The first author wishes to thank Matt Hedden,
Robert Lipshitz, and Dylan Thurston for many memorable hours,
computing knot Floer homology groups of various satellite knots.  We
would like to thank Tim Cochran and Dan Dore for their feedback, and
Linh Truong for detailed corrections to an early draft of this paper.
We also wish to thank Jen Hom and Chuck Livingston for very
interesting conversations and suggestions.  We did not take the parity
of the numerator of $t\in \Q$ into account properly in an earlier
draft of this paper; we are particularly indebted to Chuck for having
pointed out the mistake, which led to a correction in the statement of
the above conjectures. We also would like to thank the referees for
  helpful comments and suggestions.

\newcommand\CFp{{\rm {CF}}^+}
\section{Notions from knot Floer homology}
\label{sec:HFK}

The knot Floer complex from \cite{OSKnots} (which will be briefly
recalled below) fits into the following formal framework:

\begin{definition}
  \label{def:AlexMas}
  An \defin{Alexander filtered, Maslov graded chain complex} $C$
  is a
  chain complex over $\Field[U]$ with the following additional structure:
  \begin{itemize}
  \item The chain complex is a finitely generated free module over $\Field[U]$.
  \item The complex is generated over $\Field$ by a generating set equipped
    with
    two integer-valued functions, called the \defin{Maslov} and
    the \defin{Alexander functions}.
  \item Multiplication by $U$ drops the Maslov function by two.
  \item Multiplication by $U$ drops the Alexander function by
    one.
    \end{itemize}
    The Maslov and Alexander functions induce a grading and a filtration, the 
\defin{Maslov grading} and 
    the \defin{Alexander filtration} respectively.
    We require the following further properties:
    \begin{itemize}
    \item The differential drops Maslov grading by one.
    \item The Alexander function induces a filtration and the
      differential respects this Alexander filtration (i.e. elements
      with Alexander filtration $\leq t$ are mapped to elements with
      Alexander filtration $\leq t$ for all $t$).
  \end{itemize}
\end{definition} 

A chain complex with the above properties can be associated to a
doubly pointed Heegaard diagram representing a knot $K\subset
S^3$~\cite{OSKnots, RasmussenThesis}.  Let ${\mathcal
  {H}}=(\Sigma,\alphas,\betas,w,z)$ be such a Heegaard diagram for
$K$.  The chain complex $\fCFKm ({\mathcal {H}})$ is generated over
$\Field [U]$ by the same generating set $\Gen=\Ta\cap\Tb\subset {\rm
  {Sym}}^g (\Sigma )$ as the Heegaard Floer chain complex $\CFm$ of
the ambient 3-manifold $S^3$.  The generators of the complex come
equipped with two integer-valued functions, the Maslov function
and the Alexander function.  Up to an additive constant, these 
functions are characterized
as follows (the additive indeterminacy will be removed later).  If
$\x$ and $\y$ are two generators, there is a space of homotopy classes
of maps which connect them, written $\pi_2(\x,\y)$. These homotopy
classes are equipped with two additive functions (i.e. additive under
juxtaposition): the {\em Maslov index}, written $\mu(\phi)$, and, for
a point $p\in \Sigma -\alphas-\betas$ in the Heegaard surface, the
{\em multiplicity at $p$}, written $n_p(\phi)$, which measures the
algebraic intersection number of $\phi$ with the divisor $\{ p\}
\times \Sym^{g-1}(\Sigma)$.

The Maslov and Alexander functions are characterized up to overall
additive shifts
by the equations:
  \begin{align}
    M(\x)-M(\y)&=\Mas(\phi)-2n_w(\phi), \label{eq:MaslovFormula}\\
    A(\x)-A(\y)&=n_z(\phi)-n_w(\phi), \label{eq:AlexanderFormula} 
  \end{align}
for any $\phi\in\pi_2(\x,\y)$.

With the generating set and grading defined as above, the differential
$\partial ^-$ counts pseudo-holomorphic representatives of some
$\phi\in\pi_2(\x,\y)$ with Maslov index one, and it records the
multiplicity at $w$ in the exponent of a formal variable $U$.
Explicitly, the differential on $\fCFKm({\mathcal {H}})$ is defined by
\begin{equation}\label{eq:boundary}
\partial^-\x = \sum_{\{\y\in\Gen\}}
\sum_{\{\phi\in\pi_2(\x,\y)\big|\Mas(\phi)=1\}} \#\Big(
\frac{\ModFlow(\phi)}{\mathbb R} \Big) U^{n_w(\phi)} \y.
\end{equation}
Here $\ModFlow (\phi )$ is the space of holomorphic representatives of
$\phi \in \pi _2(\x , \y )$, and once its dimension (which is equal to
$\mu (\phi )$) is one, the symbol $\# \big(
\frac{\ModFlow(\phi)}{\mathbb R}\big)$ denotes the $\mod{2}$ count of
the elements in the factor space $\frac{\ModFlow(\phi)}{\mathbb R}$.
The Alexander function gives a filtration, ultimately resulting in the
Alexander filtered, Maslov graded chain complex $\fCFKm ({\mathcal
  {H}})$. (Note that $\fCFKm({\mathcal H})$ is a filtered complex. It
is the object which was denoted $CFK^{-,*}(S^3,K)$
in~\cite{OSKnots}. The homology of the associated graded object
of $\fCFKm ({\mathcal {H}})$ gives the knot Floer homology
$\HFKm(K)$.)

The total homology of $\fCFKm({\mathcal {H}})$ can be shown to
be isomorphic to $\Field[U]$,~cf.~\cite{HolDisk}; see
Proposition~\ref{prop:HFmRankOne}
below. Equation~\eqref{eq:MaslovFormula} determines $M$ only up
  to an overall additive constant. That indeterminacy is removed by
  requiring that the the generator of $\fCFKm({\mathcal {H}})$ has
  Maslov grading equal to $0$.  (Note that this convention differs
from the grading convention from~\cite{HolDisk}, where this generator
had grading $-2$; hopefully, no confusion will arise.)

Likewise, the Alexander function is specified by Equation~\eqref{eq:AlexanderFormula} up to an overall additive constant. We remove that indeterminacy as follows. First specialize the chain complex
$\fCFKm ({\mathcal {H}})$ to $U=0$, and take the homology of the
associated (Alexander) graded object, to get the knot Floer homology
group $\HFKa(K)$. The Maslov and Alexander gradings now descend to a
bigrading on $\HFKa(K)=\bigoplus_{d,s\in\Z}\HFKa_d(K,s)$, which is a
finite dimensional vector space over $\Field$. (Here, $d$ is induced
from the Maslov grading and $s$ from the Alexander grading.)
Equivalently, we consider the $\Field$-vector space $\CFKa ({\mathcal
  {H}})$ generated over $\Field$ by $\Gen=\Ta\cap\Tb$, endowed with
the differential $\widehat{\partial}$ counting only those holomorphic
disks which satisfy that $n_z=n_w=0$ (dropping the formal variable $U$
from the formula of Equation~\eqref{eq:boundary}).  The normalization
of $A$ is chosen so that the graded Euler characteristic $\chi =\sum
_{d,s}(-1)^d\dim _{\Field} \HFKa _d (K, s)\cdot t^s$ is a symmetric
polynomial in $t$; in fact, it is the symmetrized Alexander polynomial
of $K$.

We can tensor a Maslov graded, Alexander filtered chain complex $C$
(as in Definition~\ref{def:AlexMas}) with $\Field[U,U^{-1}]$ to obtain
a complex $C^{\infty}$ with a second $\Z$-filtration, given by the
powers of $U$.  More precisely, we say that 
for a generator $\x$ of $C$ over $\Field [U]$
an element $U^i\cdot \x$
has {\em algebraic filtration level} $-i$.  There is no loss of
information in doing this: $C$ can be recovered from $C^{\infty}$ by
taking the $\Field [U]$-subcomplex of $C^{\infty}$ consisting of elements of
algebraic filtration level $\leq 0$.

In the case of knot Floer homology, by applying the above procedure
to $\fCFKm ({\mathcal {H}})$ we get the chain complex
$\fCFKinf({\mathcal {H}})$. 
\begin{theorem}(\cite{OSKnots})
The $\Z \oplus \Z$-filtered 
chain homotopy type of the resulting
Alexander and algebraically filtered, Maslov graded complex $\fCFKinf
({\mathcal {H}})$ is a knot invariant. (In particular, the $\Z\oplus\Z$ filtered chain homotopy type is independent of the orientation.)
\qed
\end{theorem}

Using the Alexander filtration on $\fCFKm ({\mathcal {H}})$ we
can consider the homology of the associated graded object. Explicitly,
we endow the same $\Field[U]$-module freely generated by the set $\Gen$ equipped with
a differential $\partial ^-_K$ that counts only those holomorphic
disks which satisfy $n_z=0$. The homology of this
associated graded object is called the
\emph{knot Floer homology group} $\HFKm (K)$, which is a bigraded
$\Field [U]$-module.

\subsection{Results from knot Floer homology}
\label{subsec:Results}
We will be using several theorems about knot Floer homology in the
present paper.  The invariance of knot Floer homology in this setting
means that for two Heegaard diagrams ${\mathcal {H}}_1$ and ${\mathcal
  {H}}_2$ representing the same knot $K$ the corresponding Maslov
graded, Alexander filtered chain complexes $\fCFKm ({\mathcal {H}}_1)$
and $\fCFKm ({\mathcal {H}}_2)$ are Maslov graded, Alexander filtered
chain homotopy equivalent. From now on, this filtered chain
homotopy type will be denoted by $\fCFKm (K)$; similarly, we write
$\fCFKinf(K)$ for the filtered chain homotopy type of the corresponding complex over $\Field[U,U^{-1}]$.
We have the following K{\"u}nneth principle for connected sums:

\begin{theorem}
  \label{thm:Kunneth} (\cite[Theorem~{7.1}]{OSKnots})
  Suppose that $K_1$ and $K_2$ are two knots in $S^3$ and $K_1\# K_2$ is
  their connected sum. Then, there is a Maslov graded, Alexander
  filtered chain homotopy equivalence
  \[ \fCFKm(K_1\#K_2)\sim \fCFKm(K_1)\otimes_{\Field[U]} \fCFKm(K_2).\]
\qed
\end{theorem}

Another key result is the following:

\begin{prop}
  \label{prop:HFmRankOne}
  The total homology of $\fCFKm(K)$ 
  (i.e. taking the homology after forgetting the filtration)
  is isomorphic to $\Field[U]$, and the
  Maslov grading of the generator is $0$.
\end{prop}

\begin{proof}
  This complex computes the Heegaard Floer homology of $S^3$,
  cf.~\cite{HolDisk}. The Maslov grading has been normalized so that
the homology has its generator in Maslov grading $0$.
\end{proof}

The above statement has the following restatement in terms of the
structure of $\HFKm(K)$ (thought of as a module over $\Field[U]$).
Consider the submodule of torsion elements 
\begin{equation}\label{eq:torsion}
\Tors=\{ x\in \HFKm (K)\mid p\cdot x=0\ \text{for some polynomial } p\in
\Field [U]-\{ 0\}\};
\end{equation}
and consider the quotient $\HFKm (K)/\Tors$, which is a free $\Field
[U]$-module. Then, the free quotient is isomorphic to $\Field[U]$ or,
more succinctly, the module $\HFKm(K)$ has rank one over $\Field[U]$.

It can be shown that any non-torsion element
in the $\Field [U]$-module $\HFKm (K)$ has even Maslov grading.

The symmetry in $\Upsilon_K$ (given in Proposition~\ref{prop:SymmUps})
is a consequence of
 the following symmetry in knot Floer homology:

\begin{prop}
  \label{prop:HFKsymmetry} (\cite[Proposition~3.9]{OSKnots})
  There is a symmetry of $\fCFKinf(K)$ which switches the role
  of the algebraic and Alexander filtrations; i.e. if
    $\fCFKinf(K)'$ denotes the $\Z\oplus\Z$-filtered chain complex
    obtained from $\fCFKinf$ by exchanging the two factors in the
    filtration, then there if a $\Z\oplus\Z$-filtered chain homotopy
    equivalence betwee $\fCFKinf(K)'$ and $\fCFKinf(K)$.
\end{prop}

For a Maslov graded, Alexander filtered chain complex $(C,\partial)$ over $\Field
[U]$  we define the \emph{dual
  complex} $(C^*, d)$ as follows.  As a module,
$C^*=\Mor_{\Field [U]}(C, \Field [U])$;
that is, $C^*$ is the set of $\Field [U]$-module homomorphisms from 
$C$ to the ground ring $\Field [U]$. It is naturally an $\Field [U]$-module,
by considering the action of $p\in \Field [U]$ on $\phi \in C^*$, defined by
$(p\cdot \phi )(x)=\phi (p\cdot x)$, for all $x\in C$.
The differential $d$ is uniquely characterized by the property that 
$(d \phi ) (x)= \phi (\partial x)$ for all $x\in C$.

The grading on $C^*$ is defined as follows. First, equip $\Field [U]$
with the Maslov grading and Alexander filtration so that $M(U^d)=-2d$
and $A(U^d)=-d$. Now a morphism $\phi \in C^*$ is said to be
{\em homogeneous} if there are integers $m$ and $a$ so that
  $\phi$ takes any element of $C$ with degree $n$ and
  filtration level $b$ to an element of $\Field [U]$ with grading
  $m+n$ and filtration level $a+b$.  
  This induces a grading and a filtration on $C^*$, 
  where the homogeneous element $\phi\in C^*$
  has grading $m$ and 
  filtration level  $a$.  

The following is essentially a restatement of~\cite[Proposition~3.7]{OSKnots}:
\begin{prop}
  \label{prop:Mirror}
  Let $K$ be a knot, and $m(K)$ denote its mirror.
  Then, there is an identification $\fCFKm (m(K))\cong (\fCFKm (K))^*$
  of Alexander filtered, Maslov graded chain complexes over $\Field [U]$.
\end{prop}

\begin{proof}
  If $(\Sigma,\alphas,\betas,w,z)$ represents $K$, then
  $(-\Sigma,\alphas,\betas,w,z)$ represents $m(K)$. Reflection
  identifies moduli spaces of pseudo-holomorphic disks from $\x$ to
  $\y$ in $\Sigma$ with corresponding moduli spaces of
  pseudo-holomorphic disks from $\y$ to $\x$ in $-\Sigma$.
\end{proof}


Recall that the $\tau$-invariant $\tau (K)$ of a knot $K$ is
defined as 
\[
\tau (K)=-\max \{ A(x)\mid x\in \HFKm (K) \quad \text{is homogeneous and
  non-torsion}\} . 
\]
(Note that this is not the definition of $\tau$ from~\cite{FourBall};
but the equivalence of the two definitions was established
in~\cite[Lemma~A.2]{Trans}.) This invariant provides a non-trivial 
lower bound for the slice genus of $K$:
\[
\vert \tau (K)\vert \leq g_s(K).
\]
The identification of 
$\fCFKm (m(K))$ as the dual of $\fCFKm (K)$ (together with the 
grading conventions applied) then easily implies that 
\[
\tau (m(K))=-\tau (K).
\]

By considering a Heegaard diagram for a knot $K$ adapted to a Seifert
surface, a strong relation between the Seifert genus and the knot
Floer homology of a knot $K$ can be proved:
\begin{prop}
  \label{prop:GenusBounds}
  Let $K$ be a knot with Seifert genus $g(K)$.  Then,
  \[\max\{s\big|\HFKa_*(K,s)\neq 0\} \leq g(K). \]
\qed
\end{prop}

In fact, the above inequality is sharp~\cite{GenusBounds}, but that is
not of importance to the present applications.

\subsection{Computations}
Knot Floer homology groups can be easily computed for certain special
classes of knots.  We will use the following computation of knot Floer
homology for alternating knots~\cite{AltKnots}:

\begin{theorem}(\cite{AltKnots})
  \label{thm:AltKnotsHFK}
  Let $K$ be an alternating knot. Then,
  $\HFKa_d(K,s)\neq 0$ only if $d-s=\frac{\sigma(K)}{2}$, where
  $\sigma(K)$ denotes the knot signature. In particular,
  $\tau (K)=-\frac{\sigma (K)}{2}$. \qed 
\end{theorem}

\begin{remark}
  The normalization of the signature in the above theorem is such that
  $\sigma$ of the right-handed trefoil knot is $-2$. Since the graded
  Euler characteristic of knot Floer homology is the Alexander
  polynomial, the above theorem determines $\HFKa (K)$ for an
  alternating knot $K$ in terms of its signature and Alexander
  polynomial. As explained in
    \cite[Corollary~10.3.2]{Gridbook}, $\HFKm (K)$ of an alternating
    knot $K$ is also determined by its signature and Alexander
    polynomial.
\end{remark}

Finally, we will use the computation of knot Floer homology for torus
knots. We state a slightly more general version, as follows.
An \emph{$L$-space} is a three-manifold $Y$ that is a rational homology
sphere (i.e. $b_1(Y)=0$), with the additional property that the total rank
of its Floer homology $\HFa(Y)$ coincides with the number of elements
in $H_1(Y;\Z)$. All lens spaces are $L$-spaces.

The knot $K$ is called an \emph{$L$-space knot} if some positive
surgery on $K$ gives a 3-manifold that is an $L$-space. 
(Since $(pq-1)$-surgery on the torus knot
  $T_{p,q}$ is the lens space $L(pq-1,p^2)$, any positive torus knot
  is an $L$-space knot.)   Let $K$ be an $L$-space knot. The invariants of $K$ are heavily
constrained~\cite{NoteLens}.  Specifically, the non-zero coefficients
of the Alexander polynomial are all $\pm 1$, and they alternate in
sign, hence there is a decreasing sequence of integers
$\{\alpha_k\}_{k=0}^n$ with the property that the symmetrized
Alexander polynomial of $K$ can be written
\begin{equation}
  \label{eq:DefAs}
  \Delta_K(t)=\sum_{k=0}^n (-1)^k t^{\alpha_k}.
\end{equation}
Define another sequence of integers $\{m_k\}_{k=0}^n$ by
\begin{align}
  m_{0}&=0 \nonumber \\
  m_{2k} &= m_{2k-1}-1 \label{eq:DefMs} \\
  m_{2k+1} &= m_{2k}-2(\alpha_{2k}-\alpha_{2k+1})+1. \nonumber
\end{align}

The polynomial $\Delta_K(t)$ 
determines a Maslov graded, Alexander filtered
chain complex $C^{\infty}(\Delta_K)$ as follows.
The complex is generated over $\Field[U,U^{-1}]$ by generators
$\{x_k\}_{k=0}^n$  with grading and filtration level
specified by
\[ M(x_k)=m_k~\qquad \text{and}~\qquad A(x_k)=\alpha_k,\]
and differential (for all $0\leq 2k-1\leq n$):
\begin{equation}
  \label{eq:LSpaceDifferential}
  \partial x_{2k-1} = U^{\alpha_{2k-2}-\alpha_{2k-1}} x_{2k-2} + x_{2k}, \qquad
\partial x_{2k}=0.
\end{equation}

\begin{theorem}
  \label{thm:LSpaceKnots}
  Suppose that $K$ is an $L$-space knot.  Then, the chain complex
  $C^{\infty}(\Delta_K(t))$ represents the $\Z\oplus\Z$-filtered chain
  homotopy type $\fCFKinf(K)$.
\end{theorem}
\begin{proof}
  This is equivalent to the main result from~\cite{NoteLens}.
\end{proof}

\subsection{Slice genus bounds}
The slice genus bounds for $\Upsilon_K(t)$ will come from certain slice genus
bounds from knot Floer homology. First, we make a formal definition:

\begin{definition}
  \label{def:Delta}
  For a chain complex $C$ over $\Field[U]$ equipped with a (Maslov) grading,
  let $\delta(C)$ denote the maximal grading of any homogenous non-torsion
  class in the homology $H_*(C)$ of $C$.
\end{definition}

Starting from the knot Floer complex $\fCFKm(K)$, we can consider a new
subcomplex ${\mathcal A}(K,s)$, generated by all elements $c\in \fCFKm
(K)$ with $A(c)\leq s$.

It is perhaps easiest to think of ${\mathcal A}(K,s)$ as generated
over $\Field$ by elements $U^i \x$, where $\x$ is a preferred generator
of $\fCFKm (K)$ over $\Field [U]$ with $i\geq \max(A(\x)-s,0)$.
The complexes ${\mathcal A}(K,s)$ govern the behaviour of the Heegaard
Floer homologies $\HFm(S^3_n(K))$ of the 3-manifolds $S^3_n(K)$
obtained by sufficiently large surgeries on $K$. Functorial properties
of the cobordism map then allow one to extract slice genus bounds from
these subcomplexes; see
especially~\cite[Corollary~7.4]{RasmussenThesis}. Here we use a
formulation akin to that of Hom and Wu~\cite{HomWu}.

\begin{definition}
  Let $\nu^-(K)$ be the minimal $s$ so that 
  $\delta({\mathcal A}(K,s))=0$.
\end{definition}

Strictly speaking, Hom and Wu formulate their invariant $\nu^+(K)$ in
terms of $\HFp$, rather than $\HFm$. The two definitions give the same
result:

\begin{prop}
  The invariant $\nu^-(K)$ agrees with the invariant $\nu^+(K)$
  defined by Hom and Wu in~\cite{HomWu}.
\end{prop} 

\begin{proof}
  For a chain complex $C$ over $\Field[U]$, let $C^+$ denote the
  cokernel of the localization map $C\to C\otimes_{\Field}
  \Field[U,U^{-1}]$.  Write $\CFm(S^3)$, $\CFinf(S^3)$
  and $\CFp(S^3)$ for $\fCFKm(K)$, $\fCFKm(K)\otimes\Field[U,U^{-1}]$, 
  and $(\fCFKm(K))^+$ respectively.  
Let ${\mathcal A}^+(K,s)$ denote the cokernel of the
 natural inclusion ${\mathcal
    A}(K,s)$ to $\CFinf(S^3)$. The definition also induces a map 
$v_s^+\colon {\mathcal A}^+(K,s)\to \CFp(S^3)$.
  In fact, there is a map of short exact sequences:
  \[
  \begin{CD} 
    0 @>>> {\mathcal A}(K,s) @>>> \CFinf(S^3) @>>> {\mathcal A}^+(K,s) @>>> 0 \\
    & & @V{v^-_s}VV  @VVV @VV{v^+_s}V \\
    0 @>>> \CFm(S^3) @>>> \CFinf(S^3) @>>> \CFp(S^3) @>>> 0 \\
  \end{CD}
  \]
  In~\cite{HomWu}, the invariant $\nu^+(K)$ is defined to be the
  minimal $s$ for which $v^+_s$ takes the image of $\CFinf(S^3)$ in
  $H({\mathcal A}^+(K,s))$ isomorphically onto $H(\CFp(S^3))$.  Now
  this condition on $s$ is equivalent to the condition that $v^-_s$ is
  surjective, which in turn is equivalent to the condition that $v^-_s$
  contains the generator $1\in \Field[U]\cong H(\CFm(S^3))$.  But
  $v^-_s$ is a Maslov graded map; so this latter condition in turn is
  equivalent to the condition that $\delta({\mathcal A}(K,s))=0$.
  This establishes the desired equality.
\end{proof}

\begin{theorem}
  \label{thm:NuBounds}
  Let $K\subset S^3$ be a knot. Then, 
  $\nu^-(K)\leq g_s(K)$.
\end{theorem}

\begin{proof}
  This is~\cite[Proposition~2.4]{HomWu}; see also~\cite[Corollary~7.4]{RasmussenThesis}.
\end{proof}

\section{Definitions of $t$-modified knot Floer homology and $\Upsilon_K(t)$}
\label{sec:DefUpsilon}

The aim of this section is to describe the definition of $t$-modified
knot Floer homology of a knot $K\subset S^3$ and its corresponding
numerical invariant $\Upsilon _K (t)$.  (See~\cite{Living} for 
an alternative description of these constructions.)

We describe the definition first for rational $t$, and then extend the
definition for the general case. For rational $t$ the base ring can be
chosen to be a polynomial ring, while for the general case we need to
work with a slightly larger ring ${\mathcal {R}}$, which will be
described below.

Fix a rational number $0\leq t=\frac{m}{n} \leq 2$, and  consider the
chain complex over $\Field[v^{1/n}]$, generated by the same generators
that were 
used to generate $\CFKm $ over $\Field [U]$.
Equip this module with the grading
\[ 
\gr_t(\x)=M(\x)-tA(\x)
\]
on the generators and take $\gr _t (v^{\alpha }\x ) = \gr _t (\x )-\alpha$
for $\alpha \in {\frac{1}{n}\Z}$,
that is,  multiplication by $v$ drops the grading by one.
Define the differential
\begin{equation}
  \label{eq:tModDiff}
  \partial_t \x = \sum_{\y\in\Gen}\sum_{\{\phi\in\pi_2(\x,\y)\big|\Mas(\phi)=1\}} \# \left(\frac{\ModFlow(\phi)}{\mathbb R}\right)
  v^{t n_z(\phi)+ (2-t) n_w(\phi)} \y.
\end{equation}
This construction makes sense even when $t\in [0,2]$ is real, once we
choose a little more complicated base ring.
The ring described below was chosen so that the definition of $\partial _t$
makes sense, while the ring retains a convenient property of $\Field [U]$:
finitely generated modules decompose as direct sums of cyclic modules.
\begin{definition}\label{def:ring}
Let ${\mathbb {R}}_{\geq 0}$ denote the set of nonnegative real numbers.
The {\em ring of long power series} ${\mathcal R}$
defined as follows. As an abelian group,
${\mathcal{R}}$ is the group of
formal sums
\[
\{ \sum _{\alpha \in A} v^{\alpha}\mid A\subset {\mathbb {R}}_{\geq 0}, \
\text{$A$ well-ordered}\} ,
\]
where we use the order on $A$ induced from ${\mathbb {R}}$. 
Note that if $A$ and $B$ are well ordered subsets of ${\mathbb{R}}$,
then so is their sum 
\[ A+B=\{\gamma\in{\mathbb{R}}\big| \gamma=\alpha+\beta~\text{for some $\alpha\in A$ and $\beta\in B$}\}.\]
The product in ${\mathcal{R}}$ is given
by the formula
\[
(\sum_{\alpha\in A} v^\alpha)\cdot
(\sum_{\beta\in B} v^{\beta}) = 
\sum_{\gamma\in A + B} \#\{(\alpha,\beta)\in A\times B\big| \alpha+\beta=\gamma\}
\cdot v^{\gamma},
\]
where the count appearing as the coefficient of $v^{\gamma}$ is of course to be interpreted as a number modulo $2$.
\end{definition}

It is straightforward to verify that the above defined 
product is well-defined.
The field of fractions
${\mathcal {R}}^*$ of the ring ${\mathcal {R}}$ above can be
identified with
\[
\{ \sum _{\alpha \in A} v^{\alpha}\mid A\subset {\mathbb {R}}, 
\text{$A$ well-ordered}\} .
\]
Define the \emph{rank} of a module $M$ over ${\RRing}$ as the dimension of
the ${\RRing}^*$-vector space $M\otimes _{\RRing}\RRing ^*$.

In the interest of uniformity, we will henceforth always consider the
$t$-modified knot complex over ${\mathcal {R}}$, bearing in mind that
$\Field[v^{1/n}]$ (used in the definition for rational $t$) is a
subring of ${\mathcal {R}}$, so we can naturally extend the base ring
in the rational case. This does not affect what we mean by $\Upsilon$; see Proposition~\ref{prop:SameUpsilon}.

\begin{remark}
  The ring ${\mathcal{R}}$  is the unique valuation ring with valuation
  group ${\mathbb{R}}$ and quotient field $\Z/2\Z$. For more information on
  this ring, see~\cite[Section~11]{brandal}  and~\cite{FuchsSalce}.
\end{remark}

\begin{lemma}
  \label{lem:tModIsCx}
  The endomorphism defined in Equation~\eqref{eq:tModDiff} is a differential.
  The differential drops the grading $\gr _t$ by one.
\end{lemma}

\begin{proof}
  The fact that the endomorphism is a differential can be seen
  directly from the fact that $n_z$ and $n_w$ are additive under
  juxtaposition of flows and that $\partial $ in 
$\CFKm$ is a differential.

  The grading properties follow quickly from
  Equations~\eqref{eq:MaslovFormula}
  and~\eqref{eq:AlexanderFormula}. Specifically, if $v^{\alpha}\y$
  appears with non-zero multiplicity in $\partial_t \x$, then there is
  a homotopy class $\phi\in\pi_2(\x,\y)$ with 
  \begin{align*}
    \Mas(\phi)&=1, \\
    t n_z(\phi)+(2-t)   n_w(\phi)&=\alpha.
  \end{align*}
In this case, 
  \begin{align*}
    \gr_t(\x) - \gr_t(v^{tn_z(\phi)+(2-t)n_w(\phi)} \y) &=
    M(\x)-M(\y) - t A(\x) + t A(\y) + t (n_z(\phi)-n_w(\phi)) + 2n_w(\phi) \\
    &=\mu (\phi ) =1.
  \end{align*}
\end{proof}

\begin{definition}
  \label{def:tModifiedCFK}
  We call the resulting $\gr _t$-graded chain complex the
  \defin{$t$-modified knot Floer complex}, and denote it by
  $\CFKt(K)$.  Its homology, denoted by $\HFKt(K)$, is called the
  \defin{$t$-modified knot Floer homology}; it is a finitely generated
  $\gr _t$-graded module over $\RRing$.
\end{definition}

The construction of $\CFKt(K)$ can be thought of as coming
from a formal construction associated to  Alexander filtered, Maslov
graded complexes, as it will be explained in
Section~\ref{sec:Formal}. 

\begin{theorem}
  \label{thm:InvarianceOfHFKt}
  $\HFKt(K)$, thought of as an isomorphism class of 
    $\gr_t$-graded module over ${\mathcal {R}}$, is a
  knot invariant.
\end{theorem}

One could repeat the invariance proof for knot Floer homology (relying
on handle slide and stabilization invariances) to prove
Theorem~\ref{thm:InvarianceOfHFKt}. We prefer instead to appeal
directly to the invariance of $\CFKinf(K)$, combined with
functoriality properties of the formal construction. This proof will
be given in Section~\ref{sec:Applications}.

Next we give the definition of $\Upsilon_K(t)$:

\begin{definition}
  \label{def:DefUps}
  $\Upsilon_K(t)\in {\mathbb {R}}$ is the maximal $\gr _t$-grading of
  any homogeneous non-torsion element of $\HFKt(K)$.
\end{definition}

Theorem~\ref{thm:InvarianceOfHFKt} has the following immediate consequence:

\begin{corollary}
  \label{cor:InvariantUps}
  $\Upsilon_K(t)$ is a knot invariant.
\end{corollary}

\section{$t$-modified knot Floer homology as a formal construction}
\label{sec:Formal}

In this section we describe a way to associate new chain complexes to
a given Maslov graded, Alexander filtered chain complex $C$ over
$\Field [U]$, in the sense of Definition~\ref{def:AlexMas}.  The
$t$-modified knot complexes can be thought of as associated to
$\CFKm(K)$ in this manner. Since the association is functorial under
filtered chain homotopy equivalences (of $C$), the invariance of the
$t$-modified homology groups are quickly seen to follow from the
invariance of $\CFKm(K)$.

Suppose that $C$ is a finitely generated, Maslov graded, Alexander
filtered chain complex over $\Field[U]$. Let $\x$ be a generator
of $C$ over $\Field [U]$, with Maslov grading $M(\x)$.
Since multiplication by $U$ decreases the Maslov grading by $2$,
elements of Maslov grading $M(\x )-1$ are linear combinations of
elements of the form $U^{\frac{M(\y)-M(\x)+1}{2}} \y$, where $\y $ is
a generator.  In particular, $M(U^{\frac{M(\y)-M(\x)+1}{2}} \y)=M(\x)
-1$ implies that $M(\y)\geq M(\x)-1$, and $M(\x)$ and $M(\y)$ have
opposite parity. The differential on $C$ can be written as
\begin{equation}
  \label{eq:GeneralComplex}
  \partial \x = \sum_{\y} c_{\x,\y}\cdot U^{\frac{M(\y)-M(\x)+1}{2}} \y,
\end{equation}
where $c_{\x,\y}\in \Field$.



\begin{definition}
  \label{def:tModifyFormal}
Suppose that $C$ is a finitely generated, Maslov graded, Alexander
filtered chain complex over $\Field [U]$, and let 
$\RRing$ be the ring of Definition~\ref{def:ring} (containing 
$\Field [U]$ by $U=v^2$). For $t\in [0,2]$
the $t$-modified complex $C^t$ of $C$ is
defined as follows:
\begin{itemize}
\item
As an ${\mathcal {R}}$-module, $C^t$ is equal to $C_{\mathcal {R}}=C\otimes
_{\Field [U]}\RRing$.
\item 
For each generator $\x$ of $C$ over $\Field[U]$, define 
$\gr_t(\x)=M(\x)-t A(\x)$, and extend this to $C^t$ by the convention that
$\gr _t (v^{\alpha}\x)=\gr _t (\x )-\alpha$. Thus,
$\gr_t$ induces a real-valued grading on $C^t$ with the property that 
multiplication by $v$ drops grading by 1.
\item 
Endow the graded module $C^t$ with a differential 
\[ 
\partial_t \x = \sum_{\y} c_{\x,\y}\cdot  v^{\gr_t(\y)-\gr_t(\x)+1}\y,
\]
where the coefficients $c_{\x, \y}\in \Field$ are taken from 
the differential of $C$ through Equation~\eqref{eq:GeneralComplex}.
\end{itemize}
\end{definition}

The exponent of $v$ is chosen so that the differential drops $\gr_t$
by exactly one. The relevance of the construction is the following:

\begin{prop}
  \label{prop:tModifyAsFormal}
  Starting from the Maslov graded, Alexander filtered chain complex
  $(\CFKm(K), \partial ^-)$ of a knot $K$ over $\Field[U]$, the
  associated $t$-modified complex $(\CFKm(K))^t$ (following
  Definition~\ref{def:tModifyFormal}) agrees with the $t$-modified
  knot Floer complex $\CFKt (K)$ (in the sense of
  Definition~\ref{def:tModifiedCFK}).
\end{prop}

\begin{proof}
After identifying the generators and their gradings, 
we only need to check that if $c_{\x,\y}\neq 0$, then
\[ 
t n_z(\phi)+(2-t)n_w(\phi)=\gr_t(\y)-\gr_t(\x)+1;
\]
but this was verified in the proof of Lemma~\ref{lem:tModIsCx}.
\end{proof}

We give the $t$-modified complex $C^t$ the following second, more
transparently functorial description.  As before, let $C$ be a
finitely generated, Maslov graded, Alexander filtered chain complex
over $\Field [U]$, and think of $\Field[U]$ as a subring of $\RRing$
(with variable $v$) where $U=v^2$.  Consider the tensor product of $C$
now with the field ${\mathcal {R}}^*$ of fractions:
\[
C_{{\mathcal {R}}^*}=C\otimes _{\Field [U]}{{\mathcal {R}}^*}.
\]
The Maslov grading and Alexander filtration on $C$ induce
real-valued Maslov gradings and Alexander filtrations on $C_{{\mathcal{R}}^*}$ by 
the convention that 
\[
A(v^{\alpha}\x)=A(\x)-\frac{\alpha}{2} ~\text{and}~ M(v^{\alpha}\x)=M(\x)-\alpha,\]
where $\x$ is a homogeneous generator of $C$ as a $\Field[U]$-module
Just as in the discussion
preceding Subsection~\ref{subsec:Results}, $C_{\RRing ^*}$ admits and
algebraic filtration (given by $-\frac{\alpha}{2}$ for
$v^{\alpha}\cdot \x$), and $C_{\RRing}= C\otimes _{\Field [U]}\RRing$
can be recovered from $C_{\RRing ^*}$ by taking the elements with
algebraic filtration level $\leq 0$.

Rewrite the boundary operator from Equation~\eqref{eq:GeneralComplex}
as
\[ 
\partial \x = \sum_{\y} c_{\x,\y} \cdot v^{M(\y)-M(\x)+1} \y.
\]
For each $t\in [0,2]$, there is a new filtration $F^t$ on
$C_{{\mathcal {R}}^*}$ defined by $\frac{t}{2}$ times the Alexander 
filtration plus $(1-\frac{t}{2})$ times the algebraic filtration.
Clearly, this filtration depends on $t$.  Observe that multiplication
by $v$ drops the algebraic filtration by $\OneHalf$ and the Alexander
filtration by $\OneHalf$, and hence it drops the $F^t$ filtration level by
$\OneHalf$. Consider the subcomplex $E^t$ of $C_{{\mathcal {R}}^*}$
(as an $\RRing$-module) with filtration level $F^t\leq 0$.  This
subcomplex retains a Maslov grading (and multiplication by $v$ drops the
Maslov grading by one).

\begin{lemma}
  \label{lem:SecondConstruction}
  The chain complex $E^t$ with its induced Maslov grading is isomorphic
  to the chain complex $C^t$ from Definition~\ref{def:tModifyFormal}.
\end{lemma}

\begin{proof}
  Consider the ${\mathcal{R}}$-module map $\phi\colon C^t \to
    C_{{\mathcal {R}}^*}$ defined by $\x \mapsto v^{t A(\x)} \x$. It
    is straightforward to check the image of $\phi$ lies in $E^t$, and
    indeed $\phi\colon C^t\to E^t$ induces an ${\mathcal {R}}$-module
    isomorphism.  This isomorphism respects grading, since
  \[ M(\phi(\x))=M(v^{tA(\x)}\x)=M(\x)-tA(\x)=\gr_t(\x).\]

  Write $\partial'$ for the differential on $E^t$ and $\partial$ for the differential on $C^t$. 
  We verify that  $\phi$ is a chain map:
  \begin{align*}
    \partial' \phi(\x)&=\partial' (v^{tA(\x)}\x)= v^{tA(\x)}\partial' \x=
     v^{tA(\x)} \sum_{\y} c_{\x,\y} v^{M(\y)-M(\x)+1}\cdot \y \\
    &= \sum_{\y} c_{\x,\y} v^{M(\y)-M(\x)+1 + tA(\x)-tA(\y)}\cdot v^{tA(\y)}\y 
  =\sum_{\y} c_{\x,\y} v^{\gr_t(\y)-\gr_t(\x)+1} \phi(\y)=\phi(\partial \x),
  \end{align*}
  since $t A(\x)-tA(\y) + M(\y)-M(\x)+1=\gr_t(\y)-\gr_t(\x)+1$.
\end{proof}

We state functoriality in terms of maps
between Alexander filtered, Maslov graded chain complexes.  A morphism
$\phi\colon C \to C'$ of degree $m\in \Z$ between two Alexander filtered, Maslov graded
chain complexes (in the sense of Definition~\ref{def:AlexMas}) is an
$\Field[U]$-module map from $C$ to $C'$ that respects filtration
levels (i.e. if $\xi\in C$ has filtration level $\leq t$, then
$\phi(\xi)\in C'$ has filtration level $\leq t$, as well) and that sends
elements in $C_d$ to elements in $C'_{d+m}$. A {\em homomorphism} $f\colon C\to C'$ between two Alexander 
filtered, Maslov graded chain complexes is a morphism of degree $0$ that also satisfies
$\partial' \circ f + f \circ \partial=0$. For instance, the identity map from $C$ to itself is a homomorphism.
Two homomorphisms $f, g \colon C\to C'$ are said to be {\em homotopic}
if there is a morphism $h\colon C\to C'$ of degree $1$ with 
$f+g=\partial'\circ h + h\circ \partial$. As usual, $C$ and $C'$ are called {\em filtered chain homotopy equivalent}
if there are homomorphisms $f\colon C\to C'$ and $g\colon C'\to C$ so that $f\circ g$ and $g\circ f$ are homotopic to the respective
identity maps. 

With the above definitions in place, functoriality follows immediately from the second version of the
$t$-modified construction (given in
Lemma~\ref{lem:SecondConstruction}):
\begin{prop}
  \label{prop:tModifyFunctor}
  Let $f\colon C \to C'$ be a Maslov graded, Alexander filtered chain
  map between chain complexes over $\Field [U]$.  There is a
  corresponding graded chain map $f^t\colon C^t \to (C')^t$, with the
  following properties:
\begin{itemize}
\item 
    If $f\colon C \to C'$ and $g \colon C'\to C''$ are two Maslov
    graded, Alexander filtered chain maps, then
\[
(g\circ f)^t = g^t \circ f^t.
\]
\item If $f,g\colon C \to C'$ Maslov
    graded, Alexander filtered chain maps are chain homotopic to each
    other, then $f^t$ and $g^t$ are chain homotopic to one another. In
    particular, filtered chain homotopy equivalent
    complexes are transformed by the construction $C\mapsto C^t$ into
 homotopy equivalent complexes.
\item For  $C$ and $C'$  Maslov graded, Alexander filtered 
chain complexes over $\Field[U]$ we have 
\[
(C\otimes_{\Field[U]} C')^t \cong (C^t)\otimes_{{\mathcal {R}}} (C')^t.
\]
\end{itemize}
\qed
\end{prop}

The dual complex of a chain complex $C$ over $\RRing$ can be defined
by a simple adaptation of the definitions we had earlier for chain
complexes over $\Field [U]$. In particular, if $C$ is a finitely
generated chain complex over $\RRing$, we can consider its dual
complex $C^*=\Mor_{\RRing}(C,\RRing )$, as the module of maps
\[
\phi\colon C \to \RRing
\]
which commute with the $\RRing$-action, i.e. for $x\in C$ and $r\in
\RRing$ we have
\[
\phi(r\cdot x)=r \cdot \phi(x).
\]
There is a natural Kronecker
pairing
\[
C \otimes _{\RRing} \Mor_{\RRing}(C,\RRing )\to \RRing ,
\]
denoted $\langle\cdot,\cdot\rangle$ and defined as $\langle x, \phi
\rangle =\phi (x)$.  The dual complex $\Mor_{\RRing}(C,\RRing )$ is
equipped with the differential $d\colon \Mor_{\RRing}(C,\RRing )\to
\Mor_{\RRing}(C,\RRing )$ determined by
\[
\langle x, d\phi\rangle = \langle \partial x,\phi\rangle.
\]

Equipping the ring $\RRing$ with the grading $\gr (v^{\alpha
})=-\alpha$, we define the degree of a morphism in $C^*$ to be $m$ if
it takes elements in $C$ of degree $n$ to algebra elements of degree
$m+n$. 

As the results of the above construction, for a graded chain complex
$C$ over $\RRing$ we get the dual chain complex $C^*$, which is also
graded.  (Note that in this way we get the usual cochain complex, only
equipped with $(-1)$-times its usual grading.)  With this notion in
place, we get
\begin{prop}\label{prop:MirrorChain}
For a Maslov graded, Alexander filtered chain complex $C$ over $\Field
[U]$ and for its dual complex $C^*={\rm {Hom}}(C, \Field [U])$ we have
that
\[
(C^*)^t\cong (C^t)^*.
\]
\end{prop}

\begin{proof}
  The proof follows quickly from the definitions.
\end{proof}

\subsection{Consequences for $\Upsilon_K(t)$}
\label{sec:Applications}

Some basic properties of $\Upsilon_K(t)$ enumerated in
Section~\ref{sec:intro} are consequences of corresponding properties
of knot Floer homology, and the formal properties of $t$-modification.
Before turning to the proofs, however, we complete the discussion of
Section~\ref{sec:DefUpsilon} by verifying invariance of $t$-modified
knot Floer homology.

\bigskip

\begin{prooff} {\bf of Theorem~\ref{thm:InvarianceOfHFKt}.}
  As shown in~\cite{OSKnots}, the Maslov graded,
  Alexander filtered chain complexes over $\Field[U]$ associated to
  two Heegaard diagrams representing the same knot $K$ are filtered
  homotopy equivalent.~(Independence of the Heegaard diagram
    and the knot orientation are verified in~\cite[Theorem~3.1 and Proposition~3.9]{OSKnots} respectively.)
  According to
  Proposition~\ref{prop:tModifyFunctor},  the filtered homotopy
  equivalence induces homotopy equivalence of $t$-modified complexes,
  concluding the proof.
\end{prooff}
Notice that this result then proves Theorem~\ref{thm:Invariance}.

\bigskip

\begin{prooff} {\bf of Corollary~\ref{cor:InvariantUps}.}
  According to Theorem~\ref{thm:InvarianceOfHFKt}, the graded
  ${\mathcal {R}}$-module $\HFKt(K)$ is a knot invariant. Since
  $\Upsilon_K(t)$ is extracted from the graded ${\mathcal {R}}$-module
  structure of $\HFKt (K)$, the claim of the corollary follows.
\end{prooff}

Having established the invariance of $\HFKt (K)$ and $\Upsilon_K$, we turn to the 
basic
properties of $\Upsilon_K$ stated in Section~\ref{sec:intro}.

\bigskip

\begin{prooff} {\bf of Proposition~\ref{prop:SymmUps}.} Suppose that $C$ is a
  chain complex for knot Floer homology derived from a Heegaard
  diagram representing the knot $K$ with two basepoints $w$ and
  $z$. Let $C'$ be the chain complex with the roles of the two
  basepoints switched. As stated in
  Proposition~\ref{prop:HFKsymmetry}, there is a
  filtered quasi-isomorphism  between $C$ and $C'$. The image of a
  generator $\x$ of $C$ is mapped to a generator $\x'$ of $C'$ with
  \begin{align*}
    M'(\x')&=M(\x)-2A(\x) \\
    A'(\x')&=-A(\x).
  \end{align*}
  Thus, $\gr_t(\x)=\gr_{2-t}'(\x')$, and since the algebraic structure
  of $C$ and $C'$ is the same, the equality 
\begin{equation}\label{eq:symm}
\Upsilon _K(t)=\Upsilon _K(2-t)
\end{equation}
follows.
\end{prooff}

\bigskip

\begin{prooff} {\bf  of Proposition~\ref{prop:BoundaryCases}.}
  Observe that when $t=0$, then $\CFKt (K)$ agrees with the usual
  differential on $\CFm(S^3)$ (the Heegaard Floer chain complex of
  $S^3$), with its standard Maslov grading.  In turn, $\CFm(S^3)$ is
  graded so that its generator has grading $0$, so $\Upsilon_K(0)=0$.
\end{prooff}

\bigskip

\begin{prooff} {\bf of Proposition~\ref{prop:AddUpsilon}.}
  This follows from the K{\"u}nneth formula for connected sums
  (c.f.~\cite[Theorem~7.1]{OSKnots}, restated here as
  Theorem~\ref{thm:Kunneth}), together with
  Propositions~\ref{prop:tModifyAsFormal} and \ref{prop:tModifyFunctor}.
\end{prooff}

\bigskip

\begin{prooff} {\bf of Proposition~\ref{prop:mirror}.}
Combining
Propositions~\ref{prop:tModifyAsFormal},~\ref{prop:MirrorChain},
and~\ref{prop:Mirror}, we have 
  \begin{align*}
    \CFKt(m(K))&=(\CFKm(m(K)))^t \\
    &= (\CFKm(K)^*)t \\
    &= (\CFKm(K)^t)^* \\
    &= (\CFKt(K))^*. 
  \end{align*}
  It follows from the universal coefficient theorem, together with our grading conventions
  on dual complexes, that
  \[ 
  \Upsilon _{m(K)}(t)=-\Upsilon _K (t),
  \]
concluding the proof.
\end{prooff}

\subsection{Slice genus bounds}
The slice genus bound of Theorem~\ref{thm:BoundSliceGenus} (and of
Proposition~\ref{prop:CrossChangeUpsilon}) will be seen as
consequences of the slice genus bounds coming from 
Theorem~\ref{thm:NuBounds}, and a simple algebraic principle.

Recall that if $C$ is a finitely generated, graded chain complex over
$\Field [U]$, then $\delta(C)$ is by definition the maximal
grading of any non-torsion element in the homology $H_*(C)$.

\begin{lemma}
  \label{lem:InclusionInequality}
  Let $C \to C'$ be a grading-preserving map of finitely generated,
  graded chain complexes over $\Field[U]$ 
  with the property that the induced map
    $H(C)\otimes_{\Field[U]} \Field[U,U^{-1}]\to H(C')\otimes_{\Field[U]} \Field[U,U^{-1}]$
  is an isomorphism.
  Then, $\delta(C)\leq \delta(C')$.
\end{lemma}
\begin{proof}
  If $c\in C$ represents a non-torsion homology class, then so
  does its image in $H(C')$ (by the hypothesis). Thus $\delta(C)$,
  which coincides with the grading of some $c\in C$, is less than or equal to
  $\delta(C')$, as needed.
\end{proof}

Obviously, a similar inequality holds for complexes over the ring
$\RRing$ (after we replace $\Field [U, U^{-1}]$ with $\RRing ^*$ in
the hypothesis).

\begin{prop}
  \label{prop:BoundByNu}
  For $0 \leq t \leq 1$, there is an inequality
  \[ -t \nu^-(K)\leq \Upsilon_K(t).\]
\end{prop}

\begin{proof}
  This claim follows from the second construction of the $t$-modified
  complex, from Lemma~\ref{lem:SecondConstruction}.  
Adapting the corresponding notion for $\Field [U]$-modules, let
${\mathcal A}_{\RRing}(K,s)$ denote the subcomplex 
of $C_{\RRing }$ generated by the elements 
of $C_{\RRing}$
satisfying $A\leq s$.
There is an inclusion (of subcomplexes over
  $\RRing$)
  \[ v^{ts}\cdot {\mathcal A}_{\RRing}(K,s)\subset \CFKt(K)\]
  which induces isomorphisms after we tensor with ${\mathcal {R}}^*$.
  Then 
  \[\delta({\mathcal A}_{\RRing}(K,s))-ts\leq \delta(\CFKt(K));\]
  so if $\delta({\mathcal A}_{\RRing}(K,s))=0$, then 
  \[ -t s \leq \delta(\CFKt(K)).\]
  Minimizing over all $s$ with $\delta({\mathcal A}(K,s))=0$, we
  obtain the claimed inequality.
\end{proof}

\bigskip

\begin{prooff} {\bf of Theorem~\ref{thm:BoundSliceGenus}.}
  By taking also the mirror $m(K)$ of $K$, and using the fact that
  $\Upsilon _{m(K)}(t)=-\Upsilon _K (t)$ from
  Proposition~\ref{prop:mirror}, we conclude that
  \[ |\Upsilon_K(t)|\leq t \max(\nu^-(K),\nu^-(m(K)).\]
  The theorem now follows from Theorem~\ref{thm:NuBounds}.
\end{prooff}

\bigskip

\begin{prooff} {\bf of Proposition~\ref{prop:CrossChangeUpsilon}.}
 Since $K_{-}\# (m(K_{+}))$ has slice genus less than or
  equal to one, Theorem~\ref{thm:BoundSliceGenus} gives
  $\Upsilon_{K_{-}\# (m(K_{+}))}(t) \leq t$, so
  \[\Upsilon_{K_-}(t)\leq \Upsilon_{K_+}(t)+t\]
  follows from Proposition~\ref{prop:AddUpsilon}.
  
  To see that $\Upsilon_{K_+}(t)\leq \Upsilon_{K_-}(t)$, we proceed as
  follows.  The triangle counting map used in the proof of the skein
  exact sequence \cite[Theorem~10.2]{OSKnots} induces a filtered map
  $\CFKinf(K_{+})\to\CFKinf(K_{-})$. This is a sum over $\SpinC$
  structures (on the cobordism $W$) of maps; but restricting to either
  $\SpinC$ structure with minimal $|c_1(\spinc)|$ (evaluated on the
  generator of $H_2(W)$), we get an isomorphism on $\HFm$. Apply
  $t$-modification to this map, as in
  Proposition~\ref{prop:tModifyFunctor}, and notice that
  Lemma~\ref{lem:InclusionInequality} applies.
\end{prooff}

\begin{remark}
The above proposition could alternatively be seen as a consequence of
the skein inequality for $\nu^-(K)$. This follows quickly from the
behaviour of the maps associated to negative definite cobordisms,
see~\cite{AbsGraded}.
\end{remark}

\subsection{Special behaviour for $t\in [0,2]\cap {\mathbb {Q}}$}
The following proposition indicates that we obtain the same
  $\Upsilon$-invariant, regardless of the base ring used in the
  definition.  Indeed, for rational $t=\frac{m}{n}$, the complex
$\CFKt (K)$ can be defined over the subring $\Field[\vv^{1/n}]$ of
$\RRing$ (this is how we defined $\Upsilon$ in the introduction).

\begin{prop}
  \label{prop:SameUpsilon}
  Let $C$ be a finitely generated, free, graded chain complex over
  $\Field[v^{1/n}]$, and consider the induced chain complex $C\otimes
  _{\Field[v^{1/n}]}\RRing$.  Then, the maximal grading of any
  homogeneous non-torsion element of $H(C)$ agrees with the maximal
  grading of any homogeneous non-torsion element of
  $H(C\otimes_{\Field[v^{1/n}]}\RRing)$. In particular, for rational
  $t$, the invariant $\Upsilon _K$, defined using $\CFKt(K)$ with
  coefficients in $\Field[v^{1/n}]$, coincides with $\Upsilon _K$,
  defined using $\CFKt(K)$ with coefficients in $\RRing$.
\end{prop}

\begin{proof}
  For a graded module $M$ over $\Field[v^{1/n}]$, let $\Tors(M)$
  denote its torsion submodule. 
  We argue first that
  \begin{equation}
    \label{eq:TorsionModule}
    (M\otimes \Ring)/\Tors(M\otimes \Ring) = (M/\Tors(M))\otimes \Ring.
  \end{equation}
  To see this, consider  the short exact sequnece 
  \[ 0\rightarrow \Tors(M)\rightarrow M\rightarrow M/\Tors(M)\rightarrow 0, \]
  where $M/\Tors(M)$ is a free module. Since $\Ring$ is torsion-free
  as a module over $\Field[v^{1/n}]$, we can tensor the above short
  exact sequence with $\Ring$ to get
  \[ 0\rightarrow \Tors(M)\otimes \Ring\rightarrow M\otimes \Ring
  \rightarrow (M/\Tors(M))\otimes \Ring \rightarrow 0. \] Since the
  image of $\Tors(M)\otimes \Ring$ is contained in $\Tors(M\otimes
  \Ring)$, and $(M/\Tors(M))\otimes \Ring$ is torsion-free, we 
  conclude that $\Tors(M)\otimes\Ring=\Tors(M\otimes \Ring)$ and hence
  Equation~\eqref{eq:TorsionModule} holds.

  By the universal coefficients theorem (with coefficient ring
  $\Field[v^{1/n}]$), there is an isomorphism
\[
H(C\otimes_{\Field[v^{1/n}]}\Ring)\cong H(C)\otimes_{\Field[v^{1/n}]}\Ring,
\]
  since $\Ring$ is a torsion-free module over $\Field[v^{1/n}]$.
  Applying Equation~\eqref{eq:TorsionModule}, we conclude that
\[ 
H(C\otimes_{\Field[v^{1/n}]}\Ring)/\Tors(H(C\otimes_{\Field[v^{1/n}]}\Ring))
  \cong (H(C)/\Tors(H(C)))\otimes \Ring.
\]
  
  The maximal grading of any non-torsion element in $H(C)$ is, in
  fact, the maximal grading of a generator of the free module
  $H(C)/\Tors(H(C))$, which of course coincides with the maximal
  grading of $H(C\otimes\Ring)/\Tors(H(C\otimes \Ring))$.
\end{proof}


\bigskip

\begin{prooff} {\bf of Proposition~\ref{prop:Integrality}.}
The grading on $\CFKt (K)$, when considered over $\Field [v^{1/n}]$, lies in 
$\frac{1}{n}\Z$. Therefore $\Upsilon _K (\frac{m}{n})\in \frac{1}{n}\Z$ 
follows from Proposition~\ref{prop:SameUpsilon}.
\end{prooff}

\section{$\Upsilon_K(t)$ as a function of $t$}
\label{sec:varying}

\subsection{Continuously varying homologies}
\label{sec:ContinuouslyVaryingHomologies}
Proposition~\ref{prop:FinitelyManySlopes} will be seen as the special
case of a general construction. As before, let ${\mathcal {R}}$ denote the
ring of long power series, defined in Definition~\ref{def:ring}.
We grade this ring (by real numbers) so that $\vv$ has grading $-1$;
i.e. $\vv^{\alpha}$ has grading $-\alpha$.

Let $C$ be a finitely generated complex over
${\mathcal {R}}$. Define $\Upsilon(C)$ to be the maximal grading of any
non-torsion element in $H_*(C)$.

Note that ${\mathcal{R}}$ has a unique maximal
  ideal denoted
  $\vv^{>0} {\mathcal{R}}$, which is the union $\cup_{\alpha>0}v^{\alpha} \cdot {\mathcal{R}}$. If $C$ is a finitely generated complex over ${\mathcal{R}}$, let
  $C/\vv^{>0}C$ denote the induced complex
  \[ C\otimes_{\mathcal R}({\mathcal R}/\vv^{>0}{\mathcal R})
  = C/\bigcup_{\alpha>0} v^{\alpha}\cdot C.\]

\begin{definition}
  \label{def:ContVaryCx}
  A continuously varying family of finitely generated chain
  complexeses $\{C^t\}$ over ${\mathcal {R}}$, indexed by $t\in [0,2]$,
  is the following data:
\begin{itemize}
  \item Generators $\{x_i\}_{i=1}^n$ (which generate each $C^t$ as
    $\RRing$-modules), so that the grading $\gr_t(x_i)\in {\mathbb {R}}$ is a
    continuous function of $t$.
  \item Differentials $D^t\colon C^t \to C^t$ which drop the grading
    $\gr_t$ by one, and which vary continuously in $t$; i.e.
    \[ D^t x_i = \sum_{j} a_{i,j}(t) x_j,\] where $a_{i,j}(t)$ is
    either zero for all $t$, or it is of the form
    $a_{i,j}(t)=\vv^{g_{i,j}(t)}$ for some continuous function
    $g_{i,j}$ of $t$.  In fact, grading considerations ensure
    $g_{i,j}(t)=\gr_t(x_j)-\gr_t(x_i)+1$.
\end{itemize}
\end{definition}

\begin{prop}
  \label{prop:Varying}
  Let $\{C^t\}_{t\in[0,2]}$ be a continuously varying family of
  finitely generated chain complexes over $\RRing$. Suppose moreover that
  the rank of $H_*(C^t)$ is one. Then, $\Upsilon(C^t)$ is a continuous
  function of $t$. Moreover, for each $t$, there is a corresponding
  generator $x(t)$ in the finite generating set with the property that
  \[ 
\Upsilon(C^t)=\gr_t(x(t)).
\] 
In fact, there is some non-zero homology class in $H_*(C^t/\vv^{>0}C^t)$ whose
grading agrees with $\Upsilon(C^t)$.
\end{prop}

Before proving the statement, recall the following:

\begin{lemma}\label{lem:HomAlgLem}
  Let $C$ be a graded, finitely generated module over ${\mathcal
    {R}}$. Then, the homology of $C$ splits as a direct sum of graded
  cyclic modules; i.e. modules of the form ${\mathcal {R}}$ or
  ${\mathcal {R}}/\vv^{\alpha} {\mathcal {R}}$ for some
  $\alpha\in{\mathbb R}_{\geq 0}$ (with a possible shift in degree).
\end{lemma}

\begin{proof}
Although the ring ${\mathcal {R}}$ is not a principal ideal domain, it
is a valuation ring, so every finitely generated ideal in ${\mathcal
  {R}}$ is principal, see~\cite[Section~11]{brandal}. (In fact, the
proof of this fact for ${\mathcal {R}}$
is so simple that we sketch it here. Suppose that $f_1,\dots, f_n$
generate the ideal $I$, and write $f_i= \vv^{\alpha_i} q_i$ where
$q_i\in {\mathcal {R}}$ is a unit.  Choosing
$\alpha=\min \{ \alpha_1,\dots,\alpha_n\}$, it is easy to see that the
element $\vv^\alpha$ generates the ideal $I$.)  Adapting the proof of
the usual classification of modules over a principal ideal domain, it
follows immediately that any finitely generated module is a sum of
cyclic modules.
\end{proof}

Recall the definition of the dual complex $C^*=\Mor _{\RRing}(C, \RRing)$
for a complex $C$ over $\RRing$, with the Kronecker pairing
\[
\langle \cdot , \cdot \rangle \colon C\otimes _{\RRing} \Mor _{\RRing
}(C, \RRing) \to \RRing , \]
defined by the formula $\langle c, \phi \rangle = \phi (c)$.
The module $C^*$ is equipped with a differential $d$ that is dual to the differential on $C$.

\begin{lemma}
  \label{lem:ReformulateUps}
  Let $C$ be a finitely generated chain complex over $\RRing$,
  generated by the elements $\{ x_1, \ldots , x_k\}$, and suppose that
  the rank of $H_*(C)$ is one.  Then there is a morphism $\phi\colon C
  \to \RRing$ with $d\phi=0$  and an element $x\in C$ with $\partial x=0$,
  so that $\langle x, \phi\rangle=1$.  In fact, for any such pair $(x,
  \phi)$, the degree of $x$ is $\Upsilon(C)$ and the degree of $\phi$
  is $-\Upsilon(C)$, and there is some generator $x_i$ of $C$ with the
  property that
  \[ \gr \, x_i=\gr \, x.\]
\end{lemma}

\begin{proof}
By the universal coefficients theorem, $H_*(C)$ contains a direct
summand which is isomorphic to $\RRing$. The grading of the generator
$x$ of this $\RRing$-summand is $\Upsilon(C)$. Consider the splitting
of $H_*(C)$ as the sum of cyclic modules, and take the map to $\RRing$
which takes $x$ to $1$. By the universal coefficients theorem in
cohomology, there is a cohomology class $[\phi]$ with the property that
Kronecker pairing with $\phi$ realizes this map; i.e. there is a cocycle
so that
\[ 
\langle x, \phi \rangle = 1.
\]
It follows that $\phi$ cannot be realized as $v$ times any other
cocycle, therefore $\gr(\phi)=\Upsilon(\Mor_{\RRing}(C,\RRing))$.
Since the grading of $1$ is zero, it also follows that
$\gr(\phi)+\gr(x)=0$, implying the statement.

To show that $\gr(x)=\gr(x_i)$ for some $i\in\{1,\dots,k\}$, we express $x$ in terms of the basis for $C$:
\[ x =\sum_{i\in I} \vv^{\alpha_i}\cdot x_i,\] where $I\subset
\{1,\dots k\}$, and the $\alpha_i$ are real numbers with $\alpha_i\geq
0$, and $\gr(x)=\gr(x_i)-\alpha_i$. Let $\alpha=\min_{i\in I} \alpha_i$. Clearly, 
$x= t^{\alpha}\cdot x'$, where
$x'$ is a cycle representing a non-torsion homology class, 
with $\gr_t(x)=\gr(x')-\alpha$. It follows that $\alpha=0$, as desired.
\end{proof}

\bigskip

\begin{prooff} {\bf of Proposition~\ref{prop:Varying}.}
Fix some $s\in [0,2]$, and let $t_i$ be any sequence with
$\lim_{i\to\infty} t_i = s$.  Lemma~\ref{lem:ReformulateUps} gives
sequences of cycles $x^{t_i}$ and $\phi^{t_i}$ with
\[ 
\langle x^{t_i}, \phi^{t_i}\rangle = 1
\]
and $\gr_{t_i}(x^{t_i})\in
  \{\gr_{t_i}(x_1),\dots,\gr_{t_i}(x_k)\}$.  It follows that there is
  a uniform bound on $\gr_{t_i}(x^{t_i})$ or, equivalently, on the
  exponents of $\vv$ in the expression of $x^{t_i}$ in terms of the
  basis. Thus, we can find a subsequence
  $\{n_i\}_{i=1}^{\infty}\subset {\mathbb {N}}$ so that the
  $x^{t_{n_i}}$ converge to $x^s\in C^s$.  Passing to a further
  subequence if needed, we can assume that the $\phi^{t_{n_i}}$
  converge to some $\phi ^s\in \Mor_{\RRing}(C^s,\RRing)$.  Since
$\partial_t(x^{t_{n_i}})=0$, by continuity we conclude that
$\partial_s(x^s)=0$. Similarly, $d_{t_{n_i}} \phi^{t_{n_i}}=0$ imply
$d_s \phi^s=0$.  Now, by continuity, $\lim_{i\to \infty}\gr_{t_{n_i}}
(x^{t_{n_i}})=\gr_s(x^s)$, and we conclude from
Lemma~\ref{lem:ReformulateUps} that
$\Upsilon(C^{t_{n_i}})\to\Upsilon(C^s)$.  Since this holds for any
sequence of $\{t_i\}$ which converges to $s$, we conclude that the
function $\Upsilon(C^t)$ is continuous at $s$. Since $s$ is arbitrary,
we conclude that $\Upsilon(C^t)$ is a continuous function.

Now, there are $n$ continuous functions $\gr(x^t_i)$, and for any $t$,
the value
$\Upsilon(C^t)$ agrees with at least one of them (again, according to
Lemma~\ref{lem:ReformulateUps}).

Finally, observe that if $x^{t_i}$ represents a boundary in
$C^{t_{n_i}}/\vv^{>0}C^{t_{n_i}}$, 
then in fact $x^{t_i}$ would be homologous (in
$C^{t_{n_i}}$) to $\vv^{\alpha}$ times a different cycle in
$C^{t_{n_i}}$. But this would contradict the statement that $x^{t_i}$
is a maximal grading, non-torsion homogeneous element.
\end{prooff}

\subsection{Applications to $\Upsilon_K(t)$}

Proposition~\ref{prop:FinitelyManySlopes} is an immediate consequence
of Proposition~\ref{prop:Varying}:

\bigskip

\begin{prooff} {\bf of Proposition~\ref{prop:FinitelyManySlopes}.}
  Consider the complexes $\CFKt (K)$ over $\RRing$. These have a fixed
  generating set, and the differential is specified by
  \[ \partial _t\x = \sum_{\y\in\Gen}
  \sum_{\{\phi\in\pi_2(\x,\y)\big|\Mas(\phi)=1\}}
  \#\left({\frac{\ModFlow(\phi)}{\mathbb R}}\right) \vv^{t
    n_{z}(\phi)+(2-t)n_{w}(\phi)} \y.\] We clearly obtain a
  continously varying family of chain complexes $C^t$ (over $\RRing$)
  in the sense of Definition~\ref{def:ContVaryCx}.
  
We can now apply Proposition~\ref{prop:Varying} (whose hypotheses are
satisfied, thanks to Proposition~\ref{prop:HFmRankOne}) to conclude
that $\Upsilon _K(t)$ is a continuous function of $t$, which agrees,
at any $t$, with one of the finitely many linear functions
$\{\gr_t(\x)\}_{\x\in\Gen}$ (recall that $\Gen =\Ta \cap \Tb$ is the
set of generators and $\gr_t(\x)=M(\x)-tA(\x)$), as stated in
Proposition~\ref{prop:Varying}.  Note that since the various slopes of
the functions $\{\gr_t(\x)\}_{\x\in\Gen}$ are $(-1)$-times the
Alexander gradings of those elements, it follows at once that the
finitely many slopes of $\Upsilon_K(t)$ are all integers.
\end{prooff}

\bigskip

\begin{prooff} {\bf  of Theorem~\ref{thm:BoundConcordanceGenus}.}
  This follows from Proposition~\ref{prop:FinitelyManySlopes},
  together with Proposition~\ref{prop:GenusBounds} 
and the concordance invariance of $\Upsilon$.
\end{prooff}

\bigskip

\begin{prooff} {\bf  of Proposition~\ref{prop:SlopeUpsilon}.}
  Consider a sequence of cycles $x^t$ indexed by $t>0$
  satisfying the following properties:
  \begin{enumerate}
  \item {\em (homogeneity)} $x^t$ is homogeneous with grading $\gr_t$;
  \item {\em (non-torsion)} $x^t$ is non-torsion;
  \item \label{item:Minimality} {\em (maximality)} $x^t$ maximizes
    $\gr_t$ among all $\gr_t$-homogeneous, non-torsion elements.
  \end{enumerate}
  Write $x^t$ in terms of a basis of generators
  \[x^t = \sum a_i(t) v^{\epsilon_i(t)}\x_i,\]
  (with $a_i(t)\in \Field$ and $\epsilon _i (t)\in {\mathbb
      {R}}$). By passing to a subsequence in $t$, we can make the
  following further assumptions:
  \begin{enumerate}
    \setcounter{enumi}{3}
  \item The $a_i(t)=a_i$ are
    constant; i.e. the $\x^t$ converge as $t\to \infty$.
    Equivalently, there is some fixed set $I$ with the property that
    \[ \x^t=\sum_{i\in I} v^{\epsilon_i(t)}\x_i.\]
  \item 
    \label{item:VanishEpsilon}
    There is some $i_0\in I$ with the property that $\epsilon_{i_0}(t)\equiv 0$.
    (This final property follows from maximality.)
  \end{enumerate}
  Maximality further ensures that 
  \[ \Upsilon_K(t)=\gr_t(\x^t)=M(\x_i)-t A(\x_i)-\epsilon_i(t)\] for
  $t>0$.  Since $\Upsilon_K(0)=0$
  (Proposition~\ref{prop:BoundaryCases}), continuity of $\Upsilon$
  (Proposition~\ref{prop:Varying}) ensures that for those $j$ with
  $\epsilon_j(0)=0$ (and those exist by
  Property~\eqref{item:VanishEpsilon}), we have
  \[ \gr_t(\x^t)=-tA(\x_j).\]
  This ensures the limiting cycle $\x^0=\lim_{t\to 0} \x^t$ is
  a sum of chains with fixed
  Alexander grading, which in fact is $A(\x_j)$ (for any $j$ with
  $\epsilon_j(0)=0$).  
    Observe that $\CFKt(K)/v^{>0}\cdot \CFKt(K)=\CFa(S^3)$,
a chain complex whose homology is $\HFa(S^3)\cong\Field$.
     The image of $\x^0$
  in the homology of this quotient complex generates
  $\HFa(S^3)$. We conclude at once that $A(\x_j)=A(\x^0)\geq
  \tau(K)$.  Thus,
  \begin{equation}
    \label{eq:UpsilonAbove}
    \Upsilon_K(t)\leq -t \cdot\tau(K)
  \end{equation}
  for all sufficiently small $t$.

For the converse, we find it convenient to work in the model for
$C^t$ considered in Lemma~\ref{lem:SecondConstruction}.
Take a chain $y_0\in \CFa(S^3)$ with the following properties:
\begin{itemize}
  \item The chain $y_0$ is a cycle, which represents the non-trivial
    homology class in $\HFa(S^3)$.
  \item The chain $y_0$ is homogeneous in Maslov and Alexander gradings;
    in particular, $M(y_0)=0$.
  \item The Alexander grading of $y_0$ is $\tau(K)$.
\end{itemize}
We can extend $y_0$ to a Maslov-homogeneous cycle $y$ representing the
generator $\HFm(S^3)$ by adding only terms with non-zero $U$ powers
in them.

Write $y=y_0+U \cdot y_1$. Next, consider $y$ as a cycle in
$C^t$. Since $U y_1=v^2 y_1$ has algebraic filtration less than $2$,
and there is a uniform upper bound on the Alexander gradings of any
element, we conclude that for all $0\leq t$ sufficiently small,
$F^t(y)\leq 0$. Indeed, by making the upper bound smaller if needed,
for all $0\leq t$ sufficiently small,
$F^t(v^{t\cdot \tau(K)} \cdot y)\leq 0$, so we can view $v^{t\cdot \tau}\cdot y$ 
as an element of $E^t$. Since $y$ represents a
non-zero class in $\HFm(S^3)$, the class $v^{t\cdot \tau(K)} \cdot y$
represents a non-torsion homology class in $E^t$.

According to Lemma~\ref{lem:SecondConstruction}, in the model for
$E^t$ the Maslov grading of $v^{t\cdot \tau(K)}\cdot y$, which is
$-t\cdot \tau(K)$, corresponds to the grading in $C^t$.

We conclude that, for all sufficiently small $t\geq 0$, 
\[ \Upsilon_{K}(t)\geq -t\cdot \tau(K).\]
Combining this with Equation~\eqref{eq:UpsilonAbove},
we conclude that for all sufficiently small $\geq 0$,
\[ \Upsilon_K(t)=-t\cdot \tau(K), \]
from which Proposition~\ref{prop:SlopeUpsilon} follows immediately.
\end{prooff}

\bigskip

\begin{prooff} {\bf  of Propositions~\ref{prop:DeltaDerivative}.}
  Suppose that $t$ is a point where $\Delta \Upsilon'_K(t)\neq 0$.
  By Proposition~\ref{prop:Varying},
  there are two different generators $x$ and $y$ with
  $\gr_t(x)=\gr_t(y)$, but $A(x)\neq A(y)$ and
  \[\Delta\Upsilon'_K(t)=A(x)-A(y).\]
  The condition $\gr_t(x)=\gr_t(y)$ (that is,
  $M(x)-t A(x)=M(y)-t A(y)$) 
  ensures that
  \[ t \Delta\Upsilon'_K(t) = t (A(x)-A(y))=M(x)-M(y) \]
  is an even integer. (Recall that a non-torison element has even
  Maslov grading; cf. Proposition~\ref{prop:HFmRankOne}.)
\end{prooff}

\section{Computations}
\label{sec:Calcs}

Theorems~\ref{thm:AltKnots} and~\ref{thm:TorusKnots} are quick
consequences of the corresponding knot Floer homology computations.

\bigskip

\begin{prooff} {\bf  of Theorem~\ref{thm:AltKnots}.}
  Apply Theorem~\ref{thm:AltKnotsHFK}. In view of
  Lemma~\ref{prop:mirror}, we can assume without loss of generality
  that $\tau(K)=-\sigma(K)/2$ is non-negative. 

  In this case, there is a sequence of elements $x_0,\dots,x_n$, and
  $y_0,\dots,y_{n-1}$ in $\CFKinf(K)$, with 
  $\partial y_i = U x_i + x_{i+1}$, and  
  \begin{eqnarray*}
    A(x_i)&= \tau(K)-2i \\
    M(x_i)&= -2i. 
  \end{eqnarray*}
  Clearly, $x_0$ represents a
  non-torsion generator. In fact, any non-torsion class must
  contain at least one of the $x_i$.  Moreover, for $0\leq t\leq 1$,
  $\gr_t(x_0)$ is maximal among $\gr_t(x_i)$.  Thus,
    \[\Upsilon_K(t)=\gr_t(x_0)=-A(x_0)\cdot t =
    \frac{\sigma(K)}{2}\cdot t.\] The values for $1\leq t \leq 2$ now
    follow from Proposition~\ref{prop:SymmUps}.
\end{prooff}

\begin{remark}
  Theorem~\ref{thm:AltKnotsHFK} can be generalized immediately to {\em
    quasi-alternating knots} in the sense of~\cite{NoteLens}, using
  the appropriate generalization of Theorem~\ref{thm:AltKnots}
  from~\cite{QuasiAlternating}.
\end{remark}

\begin{theorem}
  \label{thm:LSpaceKnotsUpsilon}
  Let $K$ be an $L$-space knot, and let $\{\alpha_i\}_{i=0}^n$
  and $\{m_i\}_{i=0}^n$ be the 
  associated sequence of integers, as defined in Equation~\eqref{eq:DefAs}
  and~\eqref{eq:DefMs} respectively. Then, 
  \[
  \Upsilon _K (t)=\max _{\{i\big|0\leq 2i\leq n\}} \{ m_{2i}-t\alpha _{2i}\} .
  \]
\end{theorem}

\begin{proof}
  According to Theorem~\ref{thm:LSpaceKnots}, we can consider the
  model complex specified by Equation~\eqref{eq:LSpaceDifferential} in
  place of the knot Floer complex.  Glancing at the differential, it
  is clear that the non-torsion part is generated (over $\RRing$) by
  one of the even generators $x_{2k}$ with $0\leq 2k \leq n$
  (and $k$ is an integer). It follows that
  \[\Upsilon_K(t)=\max_{0\leq 2k\leq n} \{ \gr_t(x_{2k})\}
  = \max_{0\leq 2k\leq n} \{ m_{2k}-t\alpha_{2k}\} .\]
\end{proof}

The above result  immediately implies:

\bigskip

\begin{prooff} {\bf of Theorem~\ref{thm:TorusKnots}.}  The statement
  follows immediately from Theorem~\ref{thm:LSpaceKnotsUpsilon}, since lens
  spaces are $L$-spaces~\cite[Proposition~3.1]{HolDiskTwo}, and 
  $pq\pm 1$ surgery on the torus knot $T_{p,q}$ is a lens space.
\end{prooff}

Example~\ref{ex:T34} from the introduction follows immediately.
We generalize it to the family $T_{n,n+1}$, as follows:
\begin{prop}
  \label{prop:Tnnplus1}
  Consider the torus knot $T_{n,n+1}$. Then,
  $\Upsilon_{T_{n,n+1}}(t)$ is the piecewise linear function 
   whose values for $t\in [\frac{2i}{n},\frac{2i+2}{n}]$ (for $i=0,\dots n-1$) are given by
 \[ \Upsilon_{T_{n,n+1}}(t)= -i(i+1)-\frac{1}{2}n (n-2i-1)t.\]
  In particular,
  \[ \OneHalf t \cdot \Delta\Upsilon'_{T_{n,n+1}}(t)= \left\{\begin{array}{ll}
      1 & {\text{for $t = \frac{2i}{n}$}} \\
      0 &{\text{otherwise.}}
      \end{array}
      \right.\]
\end{prop}
\begin{proof}
The Alexander polynomial $\Delta_{n,n+1}(t)$ of $T_{n,n+1}$ is
$t^{-\frac{1}{2}n(n-1)}\frac{(t^{n(n+1)}-1)(t-1)}{(t^n-1)(t^{n+1}-1)}$. Let
\[ p(t)=\sum _{i=0}^{n-1}t^{ni}-t\sum _{i=0}^{n-2}t^{(n+1)i}.
\]
Since 
\[ (t^n-1)(t^{n+1}-1) p(t) = (t-1)(t^{n(n+1)}-1),\]
we conclude that
\[ p(t)=t^{\frac{1}{2}n(n-1)}\cdot \Delta _{T_{n,n+1}}(t).\]
Thus,
\[\alpha _{2i}=n(n-i)-\frac{1}{2}n(n-1)=\frac{1}{2}n(n-2i-1) ,\] 
and the formula for $\Upsilon_{T_{n,n+1}}(t)$ now follows from
Theorem~\ref{thm:TorusKnots}.
\end{proof}
 
The above examples show that for each rational number $t$, the homomorphism
\[ 
\OneHalf t \cdot\Delta\Upsilon '_K(t) \colon \Concord \to \Z
\]
is surjective. As the next lemma shows, the existence of this
map implies the existence of the stated direct summand in the
concordance group.

\begin{lemma}
  \label{lem:Splitting}
  Let $G$ be an Abelian group, and $H\subset G$ be a subgroup
  generated by the elements $(h_i )_{i=1}^{\infty}$.  Suppose that
  $(\lambda_n\colon G \to \Z)_{n=1}^{\infty}$ is a collection of
  homomorphisms with the property that $\lambda_n(h_n)=1$ and
  $\lambda_m(h_n)=0$ for $m>n$. Then, $H$ is a $\Z^{\infty}$ direct
  summand of $G$.
\end{lemma}

\begin{proof}
 Consider the map $\Lambda\colon G \to \Z^{\infty}$ given by
  \[ g \mapsto ( \lambda_n(g))_{n=1}^{\infty}.\] 
  Consider the linear transformation
  $\Z^{\infty}\to \Z^{\infty}$ given by 
  \[ (a_n)_{n=1}^{\infty}\mapsto \big( \lambda_n(\sum_{i=1}^{\infty} a_i h_i) \big)_{n=1}^{\infty} .\]
  The hypothesis ensures that this is the identity map plus a
  nilponent transformation. Such a map is necessarily invertible,
  concluding the proof.
\end{proof}

\bigskip

\begin{prooff} {\bf of Theorem~\ref{thm:Independence}.}
  Consider the homomorphisms
  $K\mapsto (\lambda_n(K)=\frac{1}{n}\Delta\Upsilon'_K(\frac{2}{n}))_{n=1}^{\infty}$
  and the elements $\{[T_{n,n+1}]\}_{n=1}^{\infty}\subset
  \Concord$. According to Proposition~\ref{prop:Tnnplus1},
  $\lambda_n(T_{n,n+1})=1$ and $\lambda_m(T_{n,n+1})=0$ for
  $m>n$. Thus, Lemma~\ref{lem:Splitting} applies and concludes the proof.
\end{prooff}

\newcommand\Ainfty{{\mathcal A}_{\infty}}
\newcommand\CFDa{\widehat{CFD}}
\newcommand\CFAa{\widehat{CFA}}
\newcommand\DT{\boxtimes}

\section{Generalities on bordered Floer homology (with torus boundary)}
\label{sec:Bordered}

The proof of Theorem~\ref{thm:IndependenceTopSlice} involves
computations of knot invariants for satellite knots. This problem is
well suited to bordered Floer homology~\cite{InvPair}.

Bordered Floer homology is an invariant for
  three-manifolds with parameterized (bordered) boundary. To a
  parameterized surface, this invariant associates a differential
  graded algebra, to a three-manifold with boundary it associates two
  kinds of modules over this algebra, called the {\em type $D$} and
  {\em type $A$} modules. Let $Y$ be a connected, closed, oriented
  three-manifold equipped with a parameterized separating surface $F$,
  expressing $Y=Y_1\cup_F Y_2$. A pairing theorem expresses the
  Heegaard Floer homology $\HFa(Y)$ as an algebraic pairing (the ``box
  tensor product'') of the type $D$ structure of $Y_1$ and of the type $A$
  structure of $Y_2$.

For the reader's convenience we collect here some useful facts about
bordered Floer homology (in the case of three-manifolds with torus
boundary). This material can be found in~\cite[Chapter~11]{InvPair};
see also~\cite{PNAS} for a general overview of the theory.

\subsection{The torus algebra}
\label{subsec:TorusAlg}
In this section we follow~\cite[Section~11.1]{InvPair}.
The algebra $\Alg(\Torus)$ associated to a torus
has two
minimal idempotents $\iota_0$ and $\iota_1$, and six other basic
generators:
\[
 \rho_1 \qquad  \rho_2 \qquad   \rho_3 \qquad
 \rho_{12} \qquad  \rho_{23} \qquad  \rho_{123}.
\]
The differential is zero, and the non-zero products are
\[
  \rho_1\rho_2 = \rho_{12} \qquad \rho_2\rho_3 = \rho_{23} \qquad
  \rho_1\rho_{23} = \rho_{123} \qquad \rho_{12}\rho_{3} = \rho_{123}.
\]
(All other products of two non-idempotent basic generators vanish
identically.)  There are also compatibility conditions with the
idempotents:
\begin{align*}
  \rho_1&=\iota_0\rho_1 \iota_1&
  \rho_2&=\iota_1\rho_2 \iota_0&
  \rho_3&=\iota_0\rho_3 \iota_1 \\
  \rho_{12}&=\iota_0\rho_{12} \iota_0&
  \rho_{23}&=\iota_1\rho_{23} \iota_1&
  \rho_{123}&=\iota_0\rho_{123} \iota_1.
\end{align*}

This algebra is graded by a non-commutative group $G$. One model for
$G$ is a group generated by triples $(j;p,q)$ where $j,p,q\in
\OneHalf \Z$ and $p+q \in \Z$.  The group law is
$$(j_1; p_1,q_1) \cdot (j_2; p_2,q_2)
= \biggl(j_1+j_2+
\begin{vmatrix}
p_1 & q_1 \\
p_2 & q_2 
\end{vmatrix}; p_1+p_2, q_1+q_2\biggr).$$
This group has a distinguished central element $\lambda=(1;0,0)$.
(The group $G(\{\mathbb T\})$, introduced in
\cite[Section~11.1]{InvPair} naturally grades the torus algebra;
$G(\{\mathbb T\})$ can be defined as a certain subgroup of $G$ we
discussed above.)

The gradings of the algebra elements are specified by the following formulae:
\begin{equation}\label{eq:grading-torus-alg}
  \gr(\rho_1) = \bigl(-\OneHalf;\OneHalf,-\OneHalf\bigr) \qquad 
  \gr(\rho_2) = \bigl(-\OneHalf;\OneHalf,\OneHalf\bigr) \qquad
  \gr(\rho_3) = \bigl(-\OneHalf;-\OneHalf,\OneHalf\bigr).
\end{equation}
This is extended to all other group elements by the rule
$\gr(ab)=\gr(a)\gr(b)$.

\subsection{Gradings on modules}

Let $Y_R$ be a torus-bordered three-manifold, and assume for
simplicity that $Y_R$ is a homology knot complement; that is,
$H_1(Y_R;\Z)\cong \Z$.

According to~\cite[Chapter~6]{InvPair},
  this bordered manifold has an associated
  {\em type D structure} in the sense of~\cite[Section~2.3]{InvPair},
  denoted $\CFDa(Y_R)$. As a vector space,
this is graded by a
homogeneous $G$-space.  In fact, there is an element $p$ which is homogeneous
with grading $\langle p\rangle$ and with the
property that $\CFDa(Y_R)$ is graded by the space of left cosets $G/\langle p\rangle$.
The type $D$ structure is equipped with a structure map
  \[ \delta^1\colon \CFDa(Y_R)\to\Alg\otimes\CFDa(Y_R),\]
  (where $\Alg$ is the torus algebra $\Alg(\Torus)$ recalled above)
which respects the grading, in the
sense that if $\x$ is some generator and $a\otimes \y$ appears with
non-zero multiplicity in $\delta^1(\x)$, then
\[ 
\lambda^{-1} \gr(\x)= \gr(a)\cdot \gr(\y),
\] 
where $\gr(a)$ is the grading in the algebra and $\gr(\x)$ and
$\gr(\y)$ denote gradings in the module.

Let $Y_L$ be a torus-bordered three-manifold, and assume again the
$Y_L$ is a homology knot complement. According to~\cite[Chapter~7]{InvPair}, there is a right $\Ainfty$ module associated to $Y_L$, the {\em type $A$ invariant of $Y_L$},
denoted $\CFDa(Y_L)$. Moreover, there is an element $q$
with the property that the type $A$ invariant $\CFAa(Y_L)$ is graded
by the right coset space $\langle q\rangle \backslash G$. 

The $\Ainfty$ operations  respect these gradings, in the sense that if $\x$
is some generator, and $\y$ appears with non-zero multiplicity in
$m_n(\x,a_1,\dots,a_{n-1})$, then
\begin{equation}
  \label{eq:GradingTypeA}
  \lambda^{n-2}\gr(\x)\gr(a_1)\cdots\gr(a_{n-1})=\gr(\y).
\end{equation}

The pairing theorem~\cite[Theorem~1.3]{InvPair} identifies the
quasi-isomorphism type of $\CFa(Y_{L}\cup Y_R)$ with that of the
tensor product $\CFAa(Y_L)\DT \CFDa(Y_R)$, as defined
in~\cite[Section~2.4]{InvPair}.  The grading set of $\CFa(Y_{L}\cup Y_R)$ is some cyclic group (given by the Maslov grading). The pairing
theorem also identifies this grading set with a subset of the double
coset space $\langle q\rangle \backslash G/\langle p\rangle$.  The
latter space has an action by $\Z$, induced from translation by the
central element $\lambda$.

\subsection{The pairing theorem and knots}

Suppose now that $Y_L$ contains a knot, in addition to a bordered
boundary.  In this case, the diagram for $Y_L$ contains yet another
basepoint ($w$), and $\CFAa(Y_L)$ can correspondingly be thought of as
a type $A$ structure (for example) over the torus algebra, where the
base ring is $\Field[U]$. The grading group can be correspondingly
enriched to $G\times \Z$, where the additional $\Z$-factor is called
the {\em Alexander factor}.

According to~\cite[Theorem~11.19]{InvPair} the pairing
$\CFAa(Y_L,z,w)\DT\CFDa(Y_R)$ now represents the knot Floer homology
$\HFKm$
of $Y_L\cup Y_R$, equipped with the knot (supported in $Y_L$).  The
tensor product is graded by a double coset space $\langle p\rangle
\backslash G\times \Z /\langle q\rangle$.  Translation on the
Alexander factor now corresponds to changing the Alexander grading for
the induced knot in $Y_L\cup Y_R$. 
We will use this (as
in~\cite[Chapter~11]{InvPair}) to study satellite operations (where
$Y_L$ is a solid torus).

\newcommand\AlexGr{\mathbf A}
\newcommand\MasGr{\mathbf M}
\section{Linear independence}
\label{sec:LinIndep}

Using bordered Floer homology computations, in this section we will determine
parts of the knot Floer chain complex of the knots of
Equation~\eqref{eq:TheKnots} from the introduction.  These computations will
enable us to give a proof of Theorem~\ref{thm:IndependenceTopSlice}.
In this proof we need to consider cables of the Whitehead double
$W_0^+(T_{2,3})$ of the trefoil knot $T_{2,3}$. We start with a simpler
computation of considering some cables of the trefoil, and then turn to 
cables of the Whitehead double.

See also ~\cite{Chen, HomCables, Petkova} for similar computations.

\subsection{A warm-up: cables of the trefoil knot}

Given a knot $K$ and relatively prime integers $(p,q)$, let $C_{p,q}(K)$
denote the $(p,q)$ cable of $K$. Let $T_{p,q}$ be the $(p,q)$ torus knot
($C_{p,q}$ of the unknot).  For integers $n\geq 2$ consider the family of
knots $C_{n,2n-1} (T_{2,3})$.  As a warm-up to our future calculations, we
prove the following:

\begin{lemma}
  \label{lem:PartUpsCableTorusKnot}
  The values of $\Upsilon_{C_{n,2n-1}(T_{2,3})}$ on the interval $[0,\frac{1}{n-1}]$ are determined by
  \[ \Upsilon_{C_{n,2n-1} (T_{2,3})}(t)=
  \left\{
    \begin{array}{ll}
      -(n^2-n+1) \cdot t & t\in [0,\frac{2}{2n-1}] \\ 2-(n^2-3n+2)
      \cdot t & t\in [\frac{2}{2n-1} , \frac{1}{n-1}].\\
   \end{array}
    \right.\]
\end{lemma}

In fact, it is not difficult to describe $\Upsilon_{C_{n,2n-1}
  (T_{2,3})}(t)$ completely; but the above partial computation will be
sufficient for our immediate needs.

We prove Lemma~\ref{lem:PartUpsCableTorusKnot} after a little
preparation.  The proof relies on a computation of knot Floer
homology, which can be done by a number of different techniques; see
for example~\cite{HeddenCables}.  In fact, according
to~\cite[Theorem~1.10]{HeddenCablesII}, $C_{n,2n-1} (T_{2,3})$ has an
$L$-space surgery, so one could apply Theorem~\ref{thm:LSpaceKnotsUpsilon};
see also~\cite{HomLSpace}.  We prefer instead to proceed using
bordered Floer homology (see~\cite[Chapter~11]{InvPair} for $n=2$
and~\cite{Petkova} for general $n$; see also~\cite{HomCables}), as the
computation will serve as a warm-up to a later computation given in
Lemma~\ref{lem:PartUpsCableWDTorusKnot}, where
Theorem~\ref{thm:LSpaceKnotsUpsilon} does not apply.

\begin{lemma}
  \label{lem:TypeDTref}
  The type $D$ module of the $+2$-framed right-handed trefoil knot
  complement has grading set given by
  $G/\lambda\gr(\rho_{12})\gr(\rho_{23})^{2}$.  It has
  five generators, $I$, $J$, $K$, $P$, and $Q$, with gradings
  specified by:
  \begin{equation}
    \begin{aligned}
      \gr(I)&=\lambda^{-2}\gr(\rho_{23})^{-1}  \qquad & \gr(J)=\lambda^{-1}  \qquad  &\gr(K)= \gr(\rho_{23}) \\
      \gr(P)&= \lambda^{-2}\gr(\rho_3)^{-1}  &\gr(Q)=\lambda^{-2}\gr(\rho_1)^{-1}&
    \end{aligned}
  \end{equation}
  The differential is specified by:
  \begin{equation}                                                                
  \label{eq:TrefoilD}
  \mathcenter{\begin{tikzpicture}[x=2.3cm,y=48pt]                                 
      \node at (0,2) (i) {$I$};
      \node at (1,2) (p) {$P$};
      \node at (2,2) (j) {$J$};
      \node at (2,1) (q) {$Q$};
      \node at (2,0) (k) {$K$};
      \draw[->](j) to node[above]{$\rho_3$} (p);
      \draw[->](p) to node[above]{$\rho_2$} (i);
      \draw[->](j) to node[right]{$\rho_{1}$} (q);
      \draw[->](k) to node[right]{$\rho_{123}$} (q);
      \draw[->](i) to node[below, sloped]{$\rho_{12}$} (k);
\end{tikzpicture}}
  \end{equation}
  where the arrows connecting generators represent terms in $\delta^1$,
  and the labels specify algebra elements; e.g.
  \[\delta^1J=\rho_3\otimes P + \rho_1\otimes Q.\]
\end{lemma}

\begin{proof}
  Recall that the knot Floer homology group $\HFKa(T_{2,3})$ of the
  right-handed trefoil has three generators, which we label ${\mathbf i}$,
  ${\mathbf j}$, and ${\mathbf k}$; with gradings $A({\mathbf i})=1$,
  $A({\mathbf j})=0$, $A({\mathbf k})=-1$, $M({\mathbf i})=0$, $M({\mathbf
    j})=-1$, $M({\mathbf k})=-2$; and a differential (in $\CFKinf$) with
  $\partial {\mathbf j} = U \cdot {\mathbf i} + {\mathbf k}$.  The type $D$
  module of the lemma follows from the $\HFK$-to-type $D$
  algorithm, given in~\cite[Theorem~11.27]{InvPair}.
\end{proof}

To compute the cable, we tensor with the type $A$ module for the
$(n,-1)$ cabling module; see~\cite{HomCables, Petkova}.
This graded module can be described as follows:
\begin{lemma}
  \label{lem:CablingPiece}
  The $(n,-1)$ cabling module has grading set 
  $\lambda \gr(U^n)\gr(\rho_3)\gr(\rho_2)\backslash G$.
  Its generators are $X$ and $\{A_i,B_i\}_{i=1}^n$,
  with gradings specified by 
  \begin{align}
    \gr(X)&=e     \label{eq:GradingX} \\
    \gr(A_i)&=\lambda^{i-n} \gr(\rho_2)^{-1} (\gr(\rho_2)\gr(\rho_1))^{i-n}
    \label{eq:GradingAi} \\
    \gr(B_i)=&=\lambda^{i-n}\gr(U)^{n-i} \gr(\rho_3) (\gr(\rho_2)\gr(\rho_1))^{i-n}.
    \label{eq:GradingBi}
  \end{align}
  The operations are specified by the following graph:
  \begin{equation}
    \label{eq:CablingTypeA}
  \mathcenter{\begin{tikzpicture}
      \node at (-8,1) (x) {$X$};
      \node at (-5,2) (an1) {$A_{n}$};
      \node at (-5,0) (bn1) {$B_{n}$};
      \node at (-2,2) (an2) {$A_{n-1}$};
      \node at (-2,0) (bn2) {$B_{n-1}$};
      \node at (0,2) (dotsup) {$\dots$};
      \node at (0,0) (dotsdown) {$\dots$};
      \node at (2,2) (ai) {$A_i$};
      \node at (2,0) (bi) {$B_i$};
      \node at (4,2) (dots2up) {$\dots$};
      \node at (4,0) (dots2down) {$\dots$};
      \node at (6,2) (a1) {$A_{1}$};
      \node at (6,0) (b1) {$B_{1}$};
      \draw[->,dashed](an1) to node[right]{$U^{n}$} (bn1);
      \draw[->,dashed](an2) to node[right]{$U^{n-1}$} (bn2);
      \draw[->,dashed](ai) to node[right]{$U^{i}$} (bi);
      \draw[->,dashed](a1) to node[right]{$U$} (b1);
      \draw[->](an1) to node[above,sloped]{$\rho_2$} (x);
      \draw[->](x) to node[below,sloped]{$\rho_3$} (bn1);
      \draw[->](bn2) to node[below]{$U \rho_2\otimes \rho_1$} (bn1);
      \draw[->](an2) to node[above]{$\rho_2\otimes \rho_1$} (an1);
      \draw[->](dotsup) to node[above]{$\rho_2\otimes \rho_1$} (an2);
      \draw[->](dotsdown) to node[below]{$U \rho_2\otimes \rho_1$} (bn2);
      \draw[->](ai) to node[above]{$\rho_2\otimes \rho_1$} (dotsup);
      \draw[->](bi) to node[below]{$U \rho_2\otimes \rho_1$} (dotsdown);
      \draw[->](dots2up) to node[above]{$\rho_2\otimes \rho_1$} (ai);
      \draw[->](dots2down) to node[above]{$\rho_2\otimes \rho_1$} (bi);
      \draw[->](a1) to node[above]{$\rho_2\otimes \rho_1$} (dots2up);
      \draw[->](b1) to node[below]{$U \rho_2\otimes \rho_1$} (dots2down);
    \end{tikzpicture}}
\end{equation}
\end{lemma}

\begin{proof}
  The cabling module is the type $A$ module
  associated to a doubly-pointed Heegaard diagram.
  Since that diagram has genus one, the
  holomorphic curve counting can be done combinatorially; see Figure~\ref{fig:CableDiagram} for a picture with $n=2$.
 \begin{figure}
 \begin{center}
 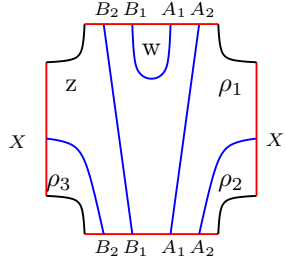
 \caption{\textbf{Heegaard diagram for the $n=2$ cabling piece.} This is 
taking place on  the punctured torus, with the usual opposite sides 
identifications. The Heegaard diagram
 represents a bordered diagram for an $(n,-1)$ cabling piece.
 \label{fig:CableDiagram}}
 \end{center}
 \end{figure}
 The computation was done in~\cite{Petkova} (see
 also~\cite{HomCables}); we recall here highlights for the reader's
 convenience.
  
 The diagram appearing in the statement of the lemma is a shorthand:
 dashed arrows represent $m_1$ actions (labelled by their outputs in
 $\Field[U]$), and all other
 operations are obtained by concatenating undashed paths
 (labelled by elements in $\Alg$ or $\Alg\otimes \Alg$).
   If there is a sequence of (undashed) arrows connecting some
 generator $P$ to some generator $Q$,
 there is a corresponding algebra operation from $P$ to $Q$
 whose sequence of input algebra elements is obtained from the sequences
 appearing on the edges by multiplying the last algebra element on some
 arrow with the first algebra  element on the next arrow.
 For instance, 
 concatenating the path from $A_{n-1}$ to $A_{n}$ (which is labelled $\rho_2\otimes \rho_1$) with the path from $A_{n}$  to $X$  (which is labelled by $\rho_2$)
we obtain an operation
  \[ m_3(A_{n-1}, \rho_2\otimes \rho_1\cdot \rho_2)=X.\]

  After verifying Equation~\eqref{eq:CablingTypeA},
  Equations~\eqref{eq:GradingX}, \eqref{eq:GradingAi},
  and~\eqref{eq:GradingBi} follow from Equation~\eqref{eq:GradingTypeA}.
  
  The verification of the grading set follows immediately from
  Equations~\eqref{eq:GradingAi} and~\eqref{eq:GradingBi} for $i=1$,
  and the fact that $\lambda^{-1} \gr(A_1)= \gr(U)\gr B_1$.
\end{proof}  

For simplicity, we have also reproduced the above answer in the
special case where $n=2$, see Equation~\eqref{eq:specCase}.  The
corresponding Heegaard diagram is pictured in
Figure~\ref{fig:CableDiagram}.  (Note that our numbering is slightly
different from the one from~\cite{Petkova}.)
\begin{equation} \label{eq:specCase}
  \mathcenter{\begin{tikzpicture}
      \node at (-9,1) (x) {$X$};
      \node at (-5,2) (a2) {$A_2$};
      \node at (-5,0) (b2) {$B_2$};
      \node at (-3,2) (a1) {$A_1$};
      \node at (-3,0) (b1) {$B_1$};
      \draw[->,dashed](a2) to node[right]{$U^{2}$} (b2);
      \draw[->,dashed](a1) to node[right]{$U^{1}$} (b1);
      \draw[->](a2) to node[above,sloped]{$\rho_2$} (x);
      \draw[->](x) to node[below,sloped]{$\rho_3$} (b2);
      \draw[->](b1) to node[below]{$U \rho_2\otimes \rho_1$} (b2);
      \draw[->](a1) to node[above]{$\rho_2\otimes \rho_1$} (a2);
    \end{tikzpicture}}
\end{equation}

The next lemma describes the chain homotopy type of the
bigraded chain complex 
$\CFKm(C_{n,2n-1}(T_{2,3}))$ over $\Field[U]$, whose homology is
the knot Floer homology $\HFKm(C_{n,2n-1}(T_{2,3}))$. This is the chain complex
obtained by taking the associated graded object for $\fCFKm(C_{n,2n-1}(T_{2,3}))$.

\begin{lemma}
  \label{lem:GradeComplex}
  The chain homotopy type of the complex $\CFKm(C_{n,2n-1}(T_{2,3}))$
  has a representative with generators
  $\{A_i\otimes P\}_{i=1}^n$,   $\{B_i\otimes P\}_{i=1}^n$
  $\{A_i\otimes Q\}_{i=1}^n$,   $\{B_i\otimes Q\}_{i=1}^n$
  and three more generators $\{X\otimes I, X\otimes J,X\otimes K\}$.
  The differential is specified by 
  \begin{align*}
    \partial (A_i\otimes Q) &= U^i B_i\otimes Q \\
    \partial (A_i\otimes P) &= \left\{\begin{array}{ll}
    U^i B_i \otimes P  & {\text{if $i<n-2$}} \\ 
    U^{n-2} B_{n-2} \otimes P + B_n\otimes Q & {\text{if $i=n-2$}}  \\
    U^{n-1} B_{n-1}\otimes P + X\otimes K & {\text{if $i=n-1$}} \\
    U^n B_n \otimes P + X\otimes I & {\text{if $i=n$}}
    \end{array}\right. \\
    \partial (X \otimes J) & = B_n \otimes P \\
    \partial (X\otimes I) &= 0 \\
    \partial (X\otimes K) &= 0 \\
    \partial (B_i\otimes Q) &= 0 \\
    \partial (B_i\otimes P) &= 0 .
  \end{align*}
  Relative bigradings are specified by 
  \begin{equation}
    \begin{aligned}
      \MasGr(B_1\otimes Q)-\MasGr(A_1\otimes Q)=1 \qquad 
      &\AlexGr(B_1\otimes Q)-\AlexGr(A_1\otimes Q)=1 \\
      \MasGr(A_1\otimes Q)-\MasGr(B_2\otimes Q)=1 \qquad
      &\AlexGr(A_1\otimes Q)-\AlexGr(B_2\otimes Q)=2n-2.
    \end{aligned}
  \end{equation}
  For $i=2,\dots,n-1$
  \begin{equation}
    \begin{aligned}
      \MasGr(B_i\otimes Q)-\MasGr(A_i\otimes Q) = 2i-1 
      \qquad &\AlexGr(B_i\otimes Q)-\AlexGr(A_i \otimes Q) = i \\
      \MasGr(A_i\otimes Q)-\MasGr(B_{i-1}\otimes P) = 1 \qquad &
      \AlexGr(A_i\otimes Q)-\AlexGr(B_{i-1}\otimes P) = n-i+1 \\
      \MasGr(B_{i-1}\otimes P)-\MasGr(A_{i-1}\otimes P)=2i-3 \qquad & 
      \AlexGr(B_{i-1}\otimes P)-\AlexGr(A_{i-1}\otimes P) = i-1 \\
      \MasGr(A_{i-1}\otimes P)-\MasGr(B_{i+1}\otimes Q) = 1  \qquad 
      &\AlexGr(A_{i-1}\otimes P)-\AlexGr(B_{i+1}\otimes Q) = n-i-1
    \end{aligned}
  \end{equation}
  and
  \begin{equation}
  \begin{aligned}
    \MasGr(B_n\otimes Q)-\MasGr(A_n\otimes Q) = 2n-1 \qquad 
    & \AlexGr(B_n\otimes Q)-\AlexGr(A_n\otimes Q)=n \\
    \MasGr(A_n\otimes Q)-\MasGr(B_{n-1}\otimes P) = 1 \qquad 
    & \AlexGr(A_n\otimes Q)-\AlexGr(B_{n-1}\otimes P) =1 \\
    \MasGr(B_{n-1}\otimes P)-\MasGr(A_{n-1}\otimes P)= 2n-3 \qquad
    & \AlexGr(B_{n-1}\otimes P)-\AlexGr(A_{n-1}\otimes P)= n-1 \\
    \MasGr(A_{n-1}\otimes P)-\MasGr(X\otimes K)=1 \qquad &
    \AlexGr(A_{n-1}\otimes P)-\AlexGr(X\otimes K)= 0 \\
    \MasGr(X\otimes J)-\MasGr(B_n\otimes P) = 1 \qquad &
    \AlexGr(B_{n}\otimes P)-\AlexGr(X\otimes J)=0 \\
    \MasGr(A_n\otimes P)-\MasGr(X\otimes I) = 1 \qquad &
    \AlexGr(A_n\otimes P)-\AlexGr(X\otimes I) = 0 \\
  \end{aligned}
  \end{equation}
  and finally
  \begin{equation}
    \begin{aligned}
      \MasGr(X\otimes J)-\MasGr(X\otimes K) = 2n-1 \qquad &
      \AlexGr(X\otimes J)-\AlexGr(X\otimes K) = n \\
      \MasGr(B_n\otimes P)-\MasGr(A_n\otimes P)= 2n-1 \qquad &
      \AlexGr(B_n\otimes P)-\AlexGr(A_n\otimes P) = n. \\
    \end{aligned}
  \end{equation}
  These are calibrated by 
  \begin{equation}
    \MasGr(B_1\otimes Q) = 0 
    \qquad
    \AlexGr(B_1\otimes Q) = n^2-n+1.
  \end{equation}
\end{lemma}

\begin{remark}
The equations are stated in the above order in order to draw attention to the
ordering of the generators by Alexander grading; e.g. the following sequence of generators
have decreasing Alexander grading:
\[
  B_1\otimes Q, \qquad A_1\otimes Q, \qquad B_2\otimes Q,
\]
then (for $i=2, \dots, n-1$)
\[ B_i\otimes Q, \qquad A_i\otimes Q, \qquad B_{i-1}\otimes P, \qquad A_{i-1}\otimes P, \qquad B_{i+1}\otimes Q, \]
and finally
\[ A_{n}\otimes Q, \qquad B_{n-1}\otimes P, \qquad A_{n-1}\otimes P.\]
(The gradings of other generators will be irrelevant, as they do not represent homology classes in $\HFKa$.)
\end{remark}

For future reference, notice that all other generators of
$\CFKm(C_{n,2n-1}(W^+_0(T_{2,3})))$ have Alexander grading $<g-2n+1$ and
Maslov grading $<-2$.  See Figure~\ref{fig:BigDiagramChase} for an
illustration of the portion of $\fCFKm$ (with Alexander
grading $\geq \AlexGr(B_2\otimes Q^r)$).

\begin{remark}
  See Figure~\ref{fig:CableTrefoilAns} for an illustration of the
  chain complex with $n=3$ (the general $n>2$ case looks similar); see
  Figure~\ref{fig:CableTrefoilAns2} for the degenerate case where
  $n=2$.
\end{remark}

 \begin{figure}
 \begin{center}
 \input{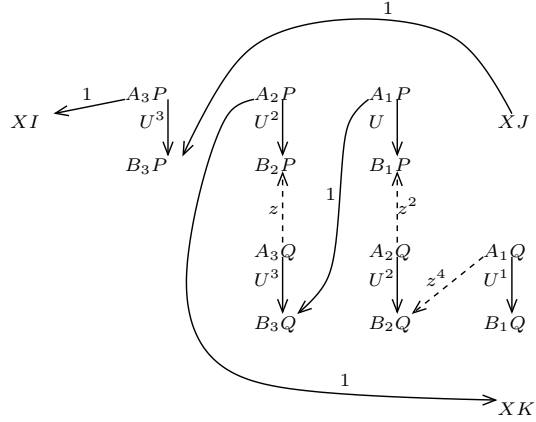}
 \caption{\textbf{Knot Floer complex $\CFKm (C_{3,5}(T_{2,3}))$
of $C_{3,5} (T_{2,3})$.}
   \label{fig:CableTrefoilAns}
  Solid arrows (which are all labelled with $U$-powers) indicate
  differentials (and the labels indicate the coefficients); dashed
  arrows are not differentials, but they connect pairs of generators
  of Maslov grading difference $1$ and Alexander grading difference
  recorded in the $z$ exponent of the labels. The complex for
  $C_{n,2n-1}(T_{2,3})$ with $n>3$ has very similar structure; the
  case $n=2$ is slightly degenerate; see
  Figure~\ref{fig:CableTrefoilAns2}.}
 \end{center}
 \end{figure}

 \begin{figure}
 \begin{center}
 \input{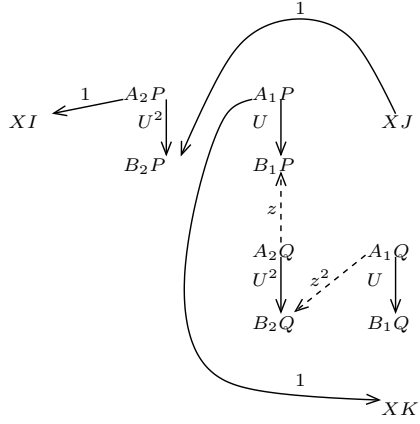}
 \caption{\textbf{Knot Floer complex of $C_{2,3} (T_{2,3})$.}  
   \label{fig:CableTrefoilAns2}
   This is the complex for $\CFKm(C_{2,3}(T_{2,3}))$, with the
   notational conventions from Figure~\ref{fig:CableTrefoilAns}.}
 \end{center}
 \end{figure}
\bigskip

\begin{prooff} {\bf  of Lemma~\ref{lem:GradeComplex}.}
  The lemma is a straightforward pairing of the module from
  Lemma~\ref{lem:CablingPiece} with the one from
  Lemma~\ref{lem:TypeDTref} (in view of the pairing
  theorem,~\cite[Theorem~11.19]{InvPair}).

  To illustrate this, we verify
    that
    \[ \partial(A_{n-2}\otimes P)= U^{n-2} B_{n-2}\otimes P + B_n\otimes Q.\]
    Lemma~\ref{lem:CablingPiece}
    states that $m_1(A_{n-2})= U^{n-2}B_{n-2}$.  This gives rise to
    the first term in the above boundary map. For the second term, we pair
    the sequence
    \[ \begin{array}{llll}\delta^1 P=\rho_2\otimes I, &\delta^1 I=
      \rho_{12}\otimes K, &\delta^1 K = \rho_{123} \otimes Q.
    \end{array}\]
    with the action
    \[ m_{4}(A_{n-2},\rho_2\otimes \rho_{12}\otimes \rho_{123})=B_n\]
    coming from the concatenation of four arrows in
    Equation~\eqref{eq:CablingTypeA}.  It is easy to see
  that there are no other terms in the differential.
  The other differentials are verified similarly.

  The pairing theorem can also be used to compute the stated bigradings.
  We illustrate this by computing
  the grading of
  $A_i\otimes P$, as follows:
  \begin{align*}
    \gr(A_i\otimes P) &= \gr(A_i)\cdot \gr(P) \\
    &= \lambda^{i-n}\gr(\rho_2)^{-1} (\gr(\rho_2)\gr(\rho_1))^{i-n}\cdot \lambda^{-2}\gr(\rho_3)^{-1} \\
    &\sim \left(\lambda \gr(U^n)\gr(\rho_3)\gr(\rho_2)\right)^{2i+1-2n} \\
    & \qquad 
    \left(\lambda^{i-n}\gr(\rho_2)^{-1} (\gr(\rho_2)\gr(\rho_1))^{i-n}\cdot \lambda^{-2}\gr(\rho_3)^{-1}\right)  \\
    & \qquad 
    \left(\lambda \gr(\rho_{12})\gr(\rho_{23})^2\right)^{n-i} \\
    & =\lambda^{-1-4i-2i^2+4n+4in-2n^2}u^{n(-1-2i+2n)}.
  \end{align*}
  In the above, $\sim$ denotes the equivalence relation of double cosets; the
  exponent $n-i$ of $\lambda\gr(\rho_{12})\gr(\rho_{23})$ is chosen to cancel
  all factors of $\gr(\rho_{12})$ (up to overall factors of $\lambda$); and
  the exponent $(2i+1-2n)$ of $\lambda \gr(U^n)\gr(\rho_3)\gr(\rho_2)$ is chosen to
  cancel the factors of $\gr(\rho_{23})$ (up to factors of $u$ and $\lambda$).
  The final step is a straightforward computation
    in the grading group, using the formulas recalled in Section~\ref{subsec:TorusAlg}.
  As in~\cite[Section~11.9]{InvPair},
  the pairing theorem interprets this double coset element as computing the Maslov/Alexander bigrading of
  generators. Specifically, the exponent of $\lambda$ computes $\MasGr-2\AlexGr$, while the exponent of $u$
  computes $-\AlexGr$. Thus the above computation shows that,  the
  Maslov/Alexander bigrading of $A_i\otimes P$ (up to overall shifts) is
  \[ 
    \MasGr(A_i\otimes P) = -1 - 4i - 2i^2 + 2n + 2n^2
    \qquad \AlexGr(A_i\otimes P)=-(n(-1 - 2i + 2n)). \]

    Proceeding in a similar manner, we find:
  \begin{align*}
    \MasGr(B_i\otimes P) = -2(1 + i + i^2 - n - n^2) \qquad 
    &\AlexGr(B_i\otimes P) = (i - n)(-1 + 2n)  \\
    \MasGr(A_i\otimes Q) = -2 i^2+2 i+2 n^2+2 n-1
    \qquad
    &\AlexGr(A_i\otimes Q)=
    -2 n (-i+n+1)\\
    \MasGr(B_i\otimes Q) = -2 \left(i^2+i-n^2-n+1\right)
    \qquad
    &\AlexGr(B_i\otimes Q)=(2 n-1) (i-n)\\
    \MasGr(X\otimes I) = -2n-2 \qquad & \AlexGr(X\otimes I)= n \\
    \MasGr(X\otimes J) = -1 \qquad & \AlexGr(X\otimes J)= 0 \\
    \MasGr(X\otimes K) = 2n \qquad & \AlexGr(X\otimes K)= -n.
  \end{align*}
  The relative bigrading statements in the statement of the lemma are a direct consequence of these computations.
  
  The non-trivial homology class in $\HFKa$ with minimal Alexander
  grading is represented by $B_{n-1}\otimes P$; in fact, that class
  descends to a non-torsion class in $\HFKm$.  By symmetry, it follows
  that the cycle $B_1\otimes Q$ with maximal Alexander grading 
  represents the $\tau$-invariant of $C_{n,2n-1}(T_{2,3})$  (in the sense that it descends to a generator for
  $\CFa(S^3)$); in particular $\MasGr(B_1\otimes Q)=0$. Its Alexander
  grading can be read off from the Alexander polynomial.
\end{prooff}

By Lemma~\ref{lem:GradeComplex}, the tensor product of these two modules has generating set
\[ \{ A_i\otimes P, B_i\otimes P, A_i\otimes Q, B_i\otimes Q, X\otimes I,
X\otimes J, X\otimes K\}_{i=1}^{n}.\] 
When $n>2$, 
there are four differentials (not
decorated by $U$): from $X\otimes J$ to $B_n\otimes P$; from
$A_n\otimes P$ to $X\otimes I$; and from $A_{n-2}\otimes P$ to $X\otimes
K$; and from $A_{n-1}\otimes P$ to $B_n \otimes Q$.

\bigskip

\begin{lemma}
  \label{lem:MonotoneGenerators}
  Let $K$ be a knot so that $\CFKa(K)$ has three generators $a$, $b$,
  and $c$ with the property that there are integers $1\leq k$ and
  $0\leq \ell$ with
  \begin{align*}
    \MasGr(a)-\MasGr(b)=2k-1 & \qquad \AlexGr(a)-\AlexGr(b)=k \\
    \MasGr(b)-\MasGr(c)=1 & \qquad \AlexGr(b)-\AlexGr(c)= \ell.
  \end{align*}
  Then, for any integer $n>\max(2,\ell)$, if $t<\frac{1}{n-1}$, then
  \[ gr_t(a)> \gr_t(b)> \gr_t(c).\]
\end{lemma}

\begin{proof}
  This is straightforward arithmetic.
\end{proof}

\begin{prooff} {\bf of Lemma~\ref{lem:PartUpsCableTorusKnot}.}
  Let $L_n=C_{n,2n-1}(T_{2,3})$.

  By the computation of the differentials in Lemma~\ref{lem:GradeComplex}, it follows that the set
  \[ \{ A_i\otimes P\}_{i=1}^{n-3}, \{ B_i\otimes P\}_{i=1}^{n-1},
  \{A_i\otimes Q\}_{i=1}^{n}, \{B_i\otimes Q\}_{i=1}^{n-1},
  \{A_n\otimes Q+ U^{n}A_{n-1}\otimes P\} \] of cycles in $\CFKm(L_n)$
  generate the homology $\HFKm(L_n)$.  Note the ranges of the indices:
  the computation of the differential allows us to remove the chain
  complex generators $X\otimes I$, $X\otimes J$, $X\otimes K$,
  $A_{n-1}\otimes P$, $B_{n-1}\otimes P$, and $A_{n-2}\otimes P$.  The
  final generator takes into account the differential which eliminates
  $B_n\otimes Q$.  Similarly, the corresponding subset of generators in $\CFKt(L_n)$ (where the last
  generator is replaced by $v^{c-\gr_t(A_n\otimes Q)} A_n\otimes Q + v^{c-\gr_t(A_{n-1}\otimes P)} A_{n-1}\otimes P$
  with $c=\max(\gr_t(A_{n-1}\otimes Q),\gr_t(A_n\otimes P))$) span a quasi-isomorphic subcomplex 
  of $\CFKt(L_n)$.

  In view of Lemma~\ref{lem:GradeComplex}, we can apply Lemma~\ref{lem:MonotoneGenerators} repeatedly
  to conclude that for  $t\leq \frac{1}{n-1}$,
  we have that for  all $i=2\dots n-1$, 
  \[ \gr_t(B_i\otimes Q)>\gr_t(A_i\otimes Q)>\gr(B_{i-1}\otimes P)>\gr_t(A_{i-1}\otimes P)>\gr_t(B_{i+1}\otimes Q);
  \]
  and also
  \[ \gr_t(B_{n}\otimes Q)>\gr_t(A_n\otimes Q)>gr_t(B_{n-1}\otimes P)>\gr_t(A_{n-1}\otimes P).\]
  Using the grading computations from Lemma~\ref{lem:GradeComplex}, we also see that for $t\leq \frac{1}{n-1}$,
  \[ \gr_t(B_1\otimes Q)>\gr_t(A_2\otimes Q) \qquad\text{and}\qquad \gr_t(A_1\otimes Q)>\gr_t(A_2\otimes Q).\]
  Thus, the generators $B_1\otimes Q$,
  $A_1\otimes Q$, and $B_2\otimes Q$ are the three homology generators with
  maximal $\gr_t$.
  
  We verify next that $B_1\otimes Q$ is a cycle, representing a
  non-torsion homology class in $\HFKt(L_n)$.  
  From
  the gradings computed in Lemma~\ref{lem:GradeComplex}, it follows
  that the Maslov grading of $B_1\otimes Q$ is greater by at least $2$ than the Maslov
  gradings of all other elements, except for $A_1\otimes Q$. 
  But $\partial(B_1\otimes Q)$ cannot equal $A_1\otimes Q$, because that would violate $\partial^2=0$.
  It follows that $B_1\otimes Q$
  represents a cycle in $\fCFKm(L_n)$. Moreover, since
  its Alexander grading is greater than the Alexander grading of all
  other generators, it follows that $B_1\otimes Q$ represents a
  non-trivial homology class in $H(\fCFKm(L_n))\cong
  \Field[U]$. We conclude that $B_1\otimes Q$, now thought of as an
  element of $\HFKt(L_n)$, has non-trivial image in
  $H(\CFKt(L_n\otimes \Ring^*))\cong \Ring^*$; i.e. it
  is a non-torsion homology class.

  Next, we claim that 
  \begin{equation} \partial(A_1\otimes Q)=U\cdot B_1\otimes Q + B_2\otimes Q 
    \label{eq:DA1}
  \end{equation}
  in
  $\fCFKm(C_{n,2n-1}(K))$.  Observe first that the Maslov gradings of
  $U\cdot B_1\otimes Q$ and $B_2\otimes Q$ are strictly greater than
  the Maslov gradings of all generators of $\fCFKm$, other than
  $A_1\otimes Q$. It follows that $\partial(A_1\otimes Q)$ can contain
  no other terms.  
  By the computation of $\CFKm(C_{n,2n-1}(K))$, it
  follows that $U\cdot B_1\otimes Q$ appears with non-zero coefficient
  in $\partial (A_1\otimes Q)$. We wish to verify that $\partial(A_1\otimes Q)$ also contains
  $B_2\otimes Q$, an element whose Alexander filtration level is $2n-2$ less than that of $A_1\otimes Q$.
  Since the Alexander filtration levels are different, the existence of this term in the differential is not 
  visible directly from the differential in the associated graded graded object
  $\CFKm(L_n)$; rather, its existence is verified by
  the following indirect argument.

  When $n>2$, we argue as follows. By Lemma~\ref{lem:GradeComplex},
  the three homology classes in
  $\HFKa(L_n)$ with minimal Alexander grading are
  represented by $B_{n-1}\otimes P$, $A_n \otimes Q + U^n
  A_{n-2}\otimes P$, and $B_{n-2}\otimes P$ (noting that $B_{n}\otimes
  Q$ is homologous to $U^{n-2} B_{n-2}\otimes P$); and that lemma
  gives a differential in
  $\CFKm(L_n)$ from $A_n \otimes Q + U^n
  A_{n-2}\otimes P$ to $U^{2n-2}\cdot B_{n-2}\otimes P$. 

  Consider now the complex ${\mathcal C}'=\fCFKm(L_n)'$
  appearing in Proposition~\ref{prop:HFKsymmetry}, which is obtained
  from $\fCFKm(L_n)$ by reversing the roles of the algebraic
  and Alexander filtrations.  Generators for ${\mathcal C}'$ are of the form
  $\x'=U^{A(\x)}\cdot \x$, where $\x$ is a generator for
  $\fCFKm(L_n)$. The fact that $U^{2n-2}\cdot B_{n-2}\otimes P$ appears
  in $\partial(A_n\otimes Q + U^{n}\cdot A_{n-2}\otimes P)$ ensures that
  the element $(B_{n-2}\otimes P)'$, whose Alexander grading is $2n-2$ smaller than that of 
  $(A_n\otimes Q + U^{n}\cdot A_{n-2}\otimes P)'$,
  appears in $\partial (A_n\otimes Q + U^{n}\cdot A_{n-2}\otimes P)'$.
  The filtered chain homotopy equivalence from ${\mathcal C}'$ to  $\fCFKm(L_n)$
  sends $(B_{n-1}\otimes P)'$, $(A_n\otimes Q+U^n \cdot A_{n-2}\otimes P)'$, and $(B_{n-2}\otimes P)'$ to
  $B_1\otimes Q$, $A_1\otimes Q$, and $B_2\otimes Q$ respectively, as those are the only generators
  in the corresponding bigradings. It follows that there is a non-zero term in the  differential from
  $A_1\otimes Q$ to $B_2\otimes Q$, which drops Alexander filtration level by $2n-2$, i.e.
  establishing Equation~\eqref{eq:DA1} when $n>2$.
  We say that this differential from $A_1\otimes Q$ to $B_2\otimes Q$ is symmetric to the differential
  from $A_n\otimes Q + U^n A_{n-2}\otimes P$ to $U^{2n-2}\cdot B_{n-2}\otimes P$.

  When $n=2$, there is no generator $A_{n-2}\otimes P$. Instead, the
  three generators in $\HFKa(L_n)$ with minimal
  Alexander grading are $B_1\otimes P$, $A_2\otimes Q$, and
  $B_2\otimes Q$. In this case, the differential from $A_2\otimes Q$
  to $U^2 \cdot B_2\otimes Q$ is symmetric to the differential from
  $A_1\otimes Q$ to $B_2\otimes Q$ which drops Alexander grading by
  $2n-2=2$, completing the verification of Equation~\eqref{eq:DA1}.

  Since $B_1\otimes Q$ represents a non-torsion class in $\HFKt(L_n)$
  and Equation~\eqref{eq:DA1} holds, we conclude that $B_2\otimes Q$ also represents
  a non-torsion class in $\HFKt(L_n)$. Since
  \begin{align*}
    \gr_t (B_1\otimes Q)&= -(n^2-n+1) t, \\
    \gr_t(B_2\otimes Q)&= -2-(n^2-3n+2) t,
  \end{align*}
  the computation of $\Upsilon_{L_n}(t)$ for $t\in [0,\frac{1}{n-1}]$ now follows.
\end{prooff}

\subsection{Cables of the Whitehead double}
Now we turn to the (partial) computation of the knot Floer complex of
the knot $C_{n,2n-1}(W_0^+(T_{2,3}))$. Our goal is to determine
$\Upsilon$ of this knot on the interval $[0,\frac{1}{n-1}]$.

\begin{lemma}
  \label{lem:PartUpsCableWDTorusKnot}
  The values of $\Upsilon_{C_{n,2n-1}(W_0^+(T_{2,3}))}(t)$ on the interval $[0,\frac{1}{n-1}]$
  are determined by
  \[ \Upsilon_{C_{n,2n-1}( W_0^+( T_{2,3}))}(t)=
  \left\{
    \begin{array}{ll}
      -(n^2-n+1) \cdot t & t\in [0,\frac{2}{2n-1}] \\ 2-(n^2-3n+2)
      \cdot t & t\in [\frac{2}{2n-1} , \frac{1}{n-1}].\\
   \end{array}
    \right.\]
\end{lemma}

Recall that for a knot $K$, its $0$-twisted Whitehead double (with a
positive clasp) is denoted by $W_0^+ (K)$.  The knot Floer homology
for this knot was computed (in terms of the knot Floer complex for
$K$) in~\cite{HeddenWhitehead}; see also~\cite{Eftekhary}. In the
special case where $K$ is the right-handed trefoil knot $T_{2,3}$, his
result specializes to the following:

\begin{theorem}(Hedden, \cite{HeddenWhitehead}) 
  \label{thm:WD}
  For the $0$-twisted Whitehead double of the right-handed trefoil
  (with its positive clasp), the knot Floer homology has 15
  generators, which we denote ${\mathbf i}^r$, ${\mathbf j}^r$,
  ${\mathbf k}^r$ for $r=0,1,2,3$ and ${\mathbf l}^s$ for $s=1,2,3$.
  The Alexander gradings of these elements are given (for $r=0,1,2,3$
  and $s=1,2,3$) by
  \[
  A({\mathbf i}^r)=A({\mathbf l}^s)=0 \qquad A({\mathbf j}^r)=1 \qquad
  A({\mathbf k}^r)=-1.
  \]
 The  Maslov gradings are given by
  \[ 
  \begin{aligned}
    M({\mathbf i}^0)=-1, \qquad &M({\mathbf j}^0)=0, 
\qquad &M({\mathbf k}^0)=-2, \\
    M({\mathbf i}^1)=-1=M({\mathbf l}^1), \qquad &M({\mathbf j}^1)=0, \qquad & M({\mathbf k}^1)=-2, \\
    M({\mathbf i}^r)=-2=M({\mathbf l}^s), \qquad &M({\mathbf j}^s)=-1, \qquad
    & M({\mathbf k}^s)=-2,
  \end{aligned}
  \]
  for $s=2,3$.
  Moreover, for $r=0,1,2,3$ and $s=1,2,3$
  \[
  \partial {\mathbf i}^r=U {\mathbf j}^r, \qquad \partial {\mathbf k}^s =
  {\mathbf l}^s;
  \]
  similarly, if we let $\partial^1_z$ denote the component of the differential
  which crosses the $z$ basepoint exactly once, but not the $w$ basepoint,
  then
\[
    \partial^1_z {\mathbf i}^r={\mathbf k}^r, \qquad
    \partial^1_z {\mathbf j}^s = {\mathbf l}^s.
\]
\qed
 \end{theorem}

 Informally, Theorem~\ref{thm:WD} says that the knot Floer complex splits as a
 sum of a component which looks like the knot Floer complex for the
 right-handed trefoil, and three further ``boxes'': four generators connected
 with four arrows, two vertical and two horizontal. This direct sum
 description is a little misleading: there might in principle be further
 horizontal arrows which cross both $w$ and $z$ basepoints. However, these are
 not relevant in the algorithm for reconstructing the corresponding type $D$
 structure.

\begin{prop} 
  \label{prop:TypeDWD}
  The type $D$ structure of the  complement of the 
0-framed positive Whitehead double of
  the right-handed trefoil knot, with framing $+2$, splits as a direct
  sum of four summands; one of these is the type $D$ structure of the
  right-handed trefoil, spelled out in Lemma~\ref{lem:TypeDTref}
  (though we will now keep the superscript $0$ in the notation for the
  five generators, $I^0$, $P^0$, $J^0$, $Q^0$, and $K^0$).

  There are three further summands, with eight generators apiece
  $\{I^t,J^t, K^t, P^t, Q^t, R^t,S^t\}$ with $t=1,2,3$ and
  differential
  \begin{equation}                                                       
\label{eq:SquareD}
    \mathcenter{\begin{tikzpicture}[x=2.3cm,y=48pt] \node at (0,0) (j)
        {$I^t$} ; \node at (1,0) (p) {$P^t$}; \node at (2,0) (i)
        {$J^t$}; \node at (2,-1)(q){$Q^t$}; \node at (2,-2)(k){$K^t$};
        \node at (0,-1)(r){$R^t$}; \node at (1,-2)(s){$S^t$}; \node at
        (0,-2)(l){$L^t$}; \draw[->](i) to node[above] {$\rho_{3}$} (p)
        ; \draw[->](p) to node[above]{$\rho_{2}$}(j) ; \draw[->](i) to
        node[right]{$\rho_{1}$}(q) ; \draw[->](k) to
        node[right]{$\rho_{123}$}(q) ; \draw[->](j) to
        node[left]{$\rho_{1}$}(r); \draw[->](l) to
        node[left]{$\rho_{123}$}(r); \draw[->](s) to
        node[above]{$\rho_{2}$}(l); \draw[->](k) to
        node[above]{$\rho_{3}$}(s);
    \end{tikzpicture}}
  \end{equation}
  Gradings for these generators, thought of as elements of
  $G/\lambda\gr(\rho_{12})\gr(\rho_{23})^{2}$, are given by:
  \begin{equation}
    \begin{aligned}
      \gr(I^1)=\lambda^{-2}\gr(\rho_{23})^{-1}  \qquad & \gr(J^1)=\lambda^{-1}  \qquad  &\gr(K)= \gr(\rho_{23}) \qquad & \gr(L^1)=\lambda^{-1}\\
      \gr(P^1)= \lambda^{-2}\gr(\rho_3)^{-1} \qquad
      &\gr(Q^1)=\lambda^{-2}\gr(\rho_1)^{-1}& \gr(S^1)= \lambda^{-1}
      \gr(\rho_3) \gr(\rho_{23}) \qquad & \gr(R^1)=
      \lambda^{-2}\gr(\rho_1)^{-1}.
    \end{aligned}
  \end{equation}
  For $s=2,3$, corresponding eight generators have grading $\lambda^{-1}$ times
  their $s=1$ counterparts; e.g. $\gr(I^s)=\lambda^{-3}\gr(\rho_{23})^{-1}$.
  (For $s=0$, the gradings are as specified in Lemma~\ref{lem:TypeDTref}; note
  that for those five generators, the gradings are the same as the gradings of
  the corresponding $s=1$ generators.)
\end{prop}

\begin{proof}
  This is a straightforward combination of Theorem~\ref{thm:WD} with the HFK-to-type $D$ module
  result~\cite[Theorem~11.27]{InvPair}.
\end{proof}

Thus, to compute the knot Floer homology of $C_{n,2n-1}(W_0^+( T_{2,3}))$, it
remains to compute the pairing of the cabling type $A$ module with a
``square'' (on the eight generators $I^t$, $J^t$, $K^t$, $L^t$, $P^t$,
$Q^t$, $S^t$, $R^t$). This computation was done by
Petkova~\cite{Petkova}. Those results can be summarized as follows:

\begin{lemma}
  \label{lem:TensorWithSquare}
  (See~\cite{Petkova})
  Consider the square type $D$ module with eight generators and differentials according to the following
  diagram:
  \begin{equation}                                                                
    \label{eq:SquareDTest}
    \mathcenter{\begin{tikzpicture}[x=2.3cm,y=48pt]                                 
        \node at (0,0) (j) {$I$} ;
        \node at (1,0) (p) {$P$};
        \node at (2,0) (i) {$J$};
        \node at (2,-1)(q){$Q$};
        \node at (2,-2)(k){$K$};
        \node at (0,-1)(r){$R$};
        \node at (1,-2)(s){$S$};
        \node at (0,-2)(l){$L$};
        \draw[->](i) to node[above] {$\rho_{3}$} (p) ;
        \draw[->](p) to node[above]{$\rho_{2}$}(j) ;
        \draw[->](i) to node[right]{$\rho_{1}$}(q) ;
        \draw[->](k) to node[right]{$\rho_{123}$}(q) ;
        \draw[->](j) to node[left]{$\rho_{1}$}(r);
        \draw[->](l) to node[left]{$\rho_{123}$}(r);
        \draw[->](s) to node[above]{$\rho_{2}$}(l);
        \draw[->](k) to node[above]{$\rho_{3}$}(s);
    \end{tikzpicture}}
  \end{equation}
  Gradings for these generators, thought of as elements of $G/\lambda\gr(\rho_{12})\gr(\rho_{23})^{2}$, are given by:
  \begin{equation}
    \begin{aligned}
      \gr(I)=\lambda^{-2}\gr(\rho_{23})^{-1}  \qquad & \gr(J)=\lambda^{-1}  \qquad  &\gr(K)= \gr(\rho_{23}) \qquad & \gr(L)=\lambda^{-1}\\
      \gr(P)= \lambda^{-2}\gr(\rho_3)^{-1} \qquad
      &\gr(Q)=\lambda^{-2}\gr(\rho_1)^{-1}& \gr(S)= \lambda^{-1}
      \gr(\rho_3)^{-1} \gr(\rho_{23}) \qquad & \gr(R)=
      \lambda^{-2}\gr(\rho_1)^{-1}.
    \end{aligned}
  \end{equation}
  The pairing of this type $D$ module with the cabling type $A$ module
  from Lemma~\ref{lem:CablingPiece} gives a chain complex with
  generators
  \[ \{A_i\otimes P, A_i\otimes Q, A_i\otimes R, A_i\otimes S, 
  B_i\otimes P, B_i\otimes Q, B_i\otimes R, B_i\otimes S, 
  X\otimes I, X\otimes J, X\otimes K, X\otimes L\}_{i=1}^n.\]
  Let
  \begin{itemize}
    \item $i$ denote any integer between $1,\dots,n$,
    \item $j$ denote any integer between $1,\dots,n-1$,
    \item $k$ any integer between $1,\dots,n-2$;
  \end{itemize}
  then
  the differential is specified by:
  \begin{align*}
    \partial (A_j\otimes P) &= A_{j+1}\otimes R + U^j\cdot B_j\otimes P \\
    \partial (A_n\otimes P) &= U^n \cdot B_n \otimes P + J\otimes X \\
    \partial (A_i\otimes Q) &= U^i \cdot B_i\otimes Q \\
    \partial (A_i\otimes R) &= U^i \cdot B_i\otimes R \\
    \partial (A_k\otimes S) &= U^k \cdot B_k\otimes S \\
    \partial (A_{n-1}\otimes S) &= B_n\otimes R + U^{n-1} \cdot B_{n-1}\otimes S \\
    \partial (A_n\otimes S) &= U^n B_n \cdot \otimes S + L\otimes X \\
    \partial (B_j\otimes P) &= U \cdot B_{j+1}\otimes R \\
    \partial (X\otimes J) &= B_n \otimes P\\
    \partial (X\otimes I) &= 0\\
    \partial (X\otimes K) &= B_n \otimes S. \\
  \end{align*}
  Relative gradings are specified as follows. The relative bigradings of the generators
  $A_i\otimes P$, $B_i\otimes P$, $A_i\otimes Q$, and $B_i\otimes Q$,
  $X\otimes I$, $X\otimes J$, and $X\otimes K$ are as in
  Lemma~\ref{lem:GradeComplex}. For $j=1,\dots,n-1$
    \begin{align}
      \MasGr(B_{j+1}\otimes Q)-\MasGr(A_{j}\otimes Q) = 1
      \qquad &
      \AlexGr(B_{j+1}\otimes Q)-\MasGr(A_{j}\otimes Q) = 1 \nonumber \\
      \MasGr(B_{j+1}\otimes P)-\MasGr(B_j\otimes R) = 1 \qquad &
      \AlexGr(B_{j+1}\otimes P)-\AlexGr(B_j\otimes R) = 1 \label{eq:GradeNewGenerators} \\
      \MasGr(B_i\otimes S) - \MasGr(A_i\otimes S)=2i-1 
      \qquad &
      \AlexGr(B_i\otimes S) - \AlexGr(A_i\otimes S)=i. \nonumber
    \end{align}
\end{lemma}

\begin{remark}
  The results of the above lemma can be paraphrased as follows.  If
  $a$ and $b$ are two generators with $\MasGr(a)-\MasGr(b)=1$ and
  $\AlexGr (a) - \AlexGr (b)=i$, we write a dashed line from $a$ to
  $b$ labelled by $z^i$. (Note that this does not indicate a
  differential, hence the dashing on the line.)  With this notation,
  the tensor product of the box with the cabling bimodule produces a
  cyclic summand and further $n$ distinct summands; one of these has
  the form
  \begin{equation}
    \mathcenter
        {\begin{tikzpicture}
            \node at (3,0)(aq)  {$A_{1}\otimes Q$};
            \node at (0,0)(bq)  {$B_{1}\otimes Q$};
            \node at (3,-3)(ar)  {$A_{1}\otimes R$};
            \node at (0,-3)(br)  {$B_{1}\otimes R$};
            \draw[->](aq) to node[above]{$U$} (bq);
            \draw[->](ar) to node[above]{$U$}(br);
            \draw[->,dashed](aq) to node[right]{$z^{n}$} (ar);
            \draw[->,dashed](bq) to node[right]{$z^{n}$} (br);
        \end{tikzpicture}}
  \end{equation}
  For $j$ between $1,\dots,n-1$, the additional summands are
  \begin{equation}
    \mathcenter
        {\begin{tikzpicture}
            \node at (0,0) (bq) {$B_{j+1}\otimes Q$};
            \node at (6,0) (aq) {$A_{j+1}\otimes Q$};
            \node at (0,-2) (bs) {$B_{j}\otimes S$};
            \node at (3,-2) (as) {$A_{j}\otimes S$};
            \node at (3,-4) (br) {$B_{j+1}\otimes R$};
            \node at (6,-4) (bp) {$B_{j} \otimes P$};
            \node at (9,-4) (ap) {$A_{j} \otimes P$};
            \node at (9,-5) (ar) {$A_{j+1}\otimes R$};
            \draw[->](aq) to node[above]{$U^{j+1}$} (bq);
            \draw[->,dashed](bq) to node[left]{$z$} (bs);
            \draw[->](as) to node[above]{$U^{j}$} (bs);
            \draw[->,dashed](as) to node[left]{$z^{n-j-1}$} (br);
            \draw[->](bp) to node[above]{$U$} (br);
            \draw[->](ap) to node[above]{$U^{j}$}(bp);
            \draw[->](ap) to (ar);
            \draw[->] (ar) to node[below]{$U^{j+1}$}(br);
            \draw[->,dashed](aq) to node[right]{$z^{n-j}$} (bp);
        \end{tikzpicture}}
\end{equation}
\end{remark}

\bigskip

\begin{prooff} {\bf  of Lemma~\ref{lem:TensorWithSquare}.}
  Again, this is a straightforward computation in the spirit of
  Lemma~\ref{lem:GradeComplex}.
\end{prooff}

Now we are in the position to give a partial computation of the
$\Upsilon$-invariant of $C_{n,2n-1}(W_0^+(T_{2,3}))$.

\bigskip

\begin{prooff} {\bf of Lemma~\ref{lem:PartUpsCableWDTorusKnot}.}
  Let $K_n'=C_{n,2n-1}(W^+_0(T_{2,3}))$.
  Consider the eighteen generators of $\CFKm(K_n')$
  indexed by $r=0,1,2,3$, $s=1,2,3$:
  \begin{equation}
    \label{eq:Gens}
  \begin{array}{llllll}
    B_1\otimes Q^r, & A_1\otimes Q^r,&  B_2\otimes Q^r, &
    B_1\otimes R^s, & A_1\otimes R^s,
  \end{array}
  \end{equation}
  obtained by tensoring the generators $A_1$ and $B_2$ of the cabling piece
    with type $D$ generators $Q^r$ and $R^s$ of the type $D$ structure of
    the complement of $W^+_0(T_{2,3})$ from Proposition~\ref{prop:TypeDWD}
    (and the superscripts are as in the statement of that proposition).
  
  It will be useful to identify the bigradings of these generators.
  To this end, let
  \[
  g=n^2-n+1, \qquad {\text{and}}\qquad
  \epsilon_i = \left\{\begin{array}{ll}
      0 & {\text{$i=0,1$}} \\
      -1 &{\text{$i=2,3$.}}
      \end{array}
      \right.
      \]
  For $r=0,1$ and  $s=1,2,3$, we have that
  \begin{align*}
   \MasGr(B_1\otimes Q^r) = 0 +\epsilon_r \qquad & \AlexGr (B_1\otimes
   Q^r)=g\\ \MasGr(A_1\otimes Q^r) = -1 + \epsilon_r \qquad &
   \AlexGr(A_1\otimes Q^r)=g-1 \\ \MasGr(B_1\otimes R^s) = -1
   +\epsilon_s \qquad & \AlexGr(B_1\otimes R^s)=g-n
   \\ \MasGr(A_1\otimes R^s) = -2 +\epsilon_s \qquad &
   \AlexGr(A_1\otimes R^s)=g-n-1 \\ \MasGr(B_2\otimes Q^r) = -2
   +\epsilon_r\qquad & \AlexGr(B_2\otimes Q^r) = g-2n+1.\\
  \end{align*}

  The computation of $\Upsilon$ will follow from partial information about $\fCFKm$ that 
  can be extracted from the above computations. This partial information is broken into 
  a sequence of successively verified claims.

  {\bf Claim 1}.  
For $t<\frac{1}{n-1}$, the generators of $K_n'$
  enumerated in Equation~\eqref{eq:Gens} have $\gr_t$ strictly greater than all of its other generators.
  This follows from the argument of
  Lemma~\ref{lem:PartUpsCableTorusKnot} and
  Equation~\eqref{eq:GradeNewGenerators}.

 \begin{figure}
 \begin{center}
 \input{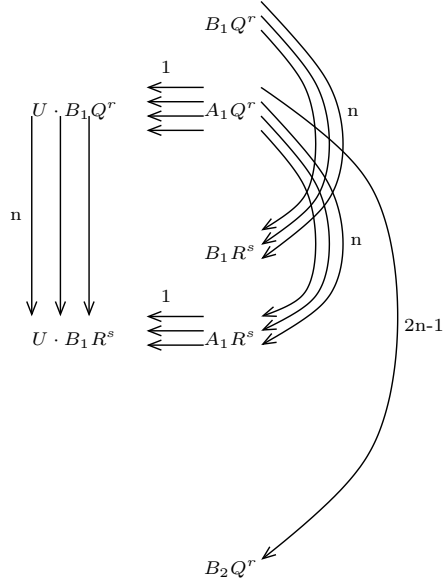}
 \caption{\textbf{Portion of the knot Floer complex of 
$C_{n,2n-1}(W_0^+(T_{2,3}))$.}
   \label{fig:BigDiagramChase}
   We have illustrated generators in Alexander gradings
   $\geq \AlexGr(B_2\otimes Q^r)$, and appearing with $U$ multiplies
   with exponent $\leq 1$. The convention here is that $r=0,1,2,3$ and
   $s=1,2,3$.  The horizontal coordinate represents the number of $U$
   powers, and the vertical coordinate indicates the Alexander
   grading. We have also illustrated all vertical and horizontal
   differentials connecting these elements; more precisely, a
   collection of parallel arrows indicates a linear map connecting
   spans of generators, and the number of parallel arrows indicates
   the dimension of its image.  The integers indicate the lengths of
   these arrows.}
 \end{center}
 \end{figure}

   {\bf Claim 2}. 
    There are  non-zero elements $x_1$, $y$, and $x_2$
  in $\fCFKm$, which are in the span (over $\Field$) of the eighteen generators 
  from Equation~\eqref{eq:Gens}, 
  satisfying the following further properties:
 \begin{enumerate}[label=(X-\arabic*),ref=(X-\arabic*)]
   \item\label{prop1}  $x_1$ is in the same bigrading as $B_1\otimes Q^0$
   \item $y$ is in the same bigrading is $A_1\otimes Q^0$
   \item $x_2$ is in the same bigrading as $B_2 \otimes Q^2$
   \item $\partial_v x_1=0$
   \item $\partial_h y=x_1$
   \item\label{propn} $\partial_v y = x_2$.
 \end{enumerate}
 We find these elements as follows. First we find some element $y$ in the same
 bigrading as $A_1\otimes Q$ with the property that $\partial_v y=x_2$, 
 where $x_2$ is a non-zero element in the same bigrading as $B_2\otimes Q$. The element $y$ corresponds,
 under the conjugation
 symmetry of knot Floer homology, 
 to  the element of $\fCFKm$ represented by 
 \[\xi = A_n\otimes Q^0+U^n A_{n-2}\otimes P^0 \]
 (in case $n>2$; when $n=2$, take $\xi=A_2$).
 
   Since $\partial \xi = U^{2n-2} B_{n-2}\otimes P^0$, 
   there is a symmetric differential $\partial _h y = U x_1$,
  for some non-zero element $x_1$ in the same bigrading as $B_1\otimes Q^0$.
In fact, the
 bordered computation shows that all non-zero elements in the bigrading of
 $y$ have non-zero $\partial_h$ in the bigrading of $Ux_1$. Thus, the fact
 that $\partial_v x_1=0$ is a consequence of $\partial^2=0$.
 This completes the construction of $x_1$, $x_2$, and $y$ satisfying Properties~\ref{prop1}-\ref{propn}.

{\bf Claim 3.} The element  $x_1$ constructed above is a cycle in $\fCFKm(K_n')$, which represents a  non-trivial
homology class.
The fact that 
 $x_1$ is a cycle
 follows from the fact that 
 $\partial_v(x_1)=0$, so $\partial x_1$ contains terms with non-zero $U$ power,
 and, according to the above grading computations,
 all such elements have Maslov grading $<-2$;
 but $\MasGr(x_1)=0$.
 Moreover, since $x_1$ has maximal Alexander grading among all generators, 
 it follows that $x_1$ represents a homologically non-trivial class in $\CFa(S^3)=\fCFKm/(U=0)$
 and hence also in $\fCFKm$, whose homology is $\Field[U]$.
 It follows that $x_1$ represents a non-torsion homology class in $\HFKt(K_n')$.

{\bf Claim 4.}
The following equation holds in $\fCFKm(K_n')$,
\begin{equation}
  \label{eq:ComputeBoundaryMinus}
  \partial y = U x_1 + x_2.
\end{equation}
To see this, observe that the definitions of $\partial_v$ and $\partial_h$ ensure that
\[ \partial y = U x_1 + x_2 + U z, \]
where $z$ is some element in Alexander filtration level is less than or equal to that of $y$.
But the above grading computations  show that such an element $z$
has Maslov grading $<0$, and so $\MasGr(U z)<-2$. Since $\MasGr(y)=-1$, we conclude that $z=0$,
verifying Equation~\eqref{eq:ComputeBoundaryMinus}.

{\bf Claim 5.} The elements $x_1$ and $x_2$ represent non-torsion homology classes in 
$\HFKt(K_n')$
The statement for $x_1$ follows immediately from Claim 3, and the statement for $x_2$ follows
from that, together with Claim 4.

{\bf Claim 6.} For $t<\frac{1}{n-1}$, the elements $x_1$ and $x_2$ are the two non-torsion elements
of $\HFKt(K_n')$ with maximal $\gr_t$.
Elements with the same bigrading as 
$A_1\otimes Q^r$ have $\gr_t(A_1\otimes Q^r)>\gr_t(B_1\otimes Q^r)$ for all $t$. However,
the differential $\partial_h$ is injective on the span of $A_1\otimes Q^r$, which implies also that
the cycles in $\CFKt(K_n')$ cannot contain components among the $A_1\otimes Q^r$.
Similarly, the differential $\partial_v$ is injective on the span of $B_1\otimes Q^r$ so 
cycles in $\CFKt(K_n')$ cannot contain components among the $B_1\otimes Q^r$. 
Finally, if a cycle in $\CFKt(K_n')$
contains a component among $B_1\otimes R^s$, then that cycle is homologous to another one,
obtained by adding $\partial(B_1\otimes Q^s)$.
It follows from Claim 1 now that $x_1$ and $x_2$ are two non-torsion elements with maximal $\gr_t$ for $t\in [0,\frac{1}{n-1}]$.

In view of Claim 6, the result follows from the fact that
\[ \gr_t(x_1)=-(n^2-n+1)t\qquad{\text{and}}\qquad
\gr_t(x_2)=-2-(n^2-3n+2).\]
\end{prooff}

In the proof of Theorem~\ref{thm:IndependenceTopSlice} we need to compare the
above result with $\Upsilon_{T_{n,2n-1}}$.

\begin{lemma}
  \label{lem:PartUpsTorusKnot}
  \[ \Upsilon_{T_{n,2n-1}}(t)=
  - (n-1)^2\cdot t
  \]
  for $t\leq \frac{2}{n}$.
\end{lemma}

That latter function can be computed explicitly from the Alexander
polynomials, as in Theorem~\ref{thm:TorusKnots}. 
Thus, we could obtain the theorem by playing around with coefficients
of the Alexander polynomial; we prefer instead
to obtain these bounds via bordered Floer homology, in the
spirit of the previous computations.

\bigskip

\begin{prooff} {\bf of Lemma~\ref{lem:PartUpsTorusKnot}.}
  Recall~\cite[Theorem~A.11]{InvPair} that the $2$-framed unknot complement
  has type~D module with three generators which we write as $P$, $Q$, and $I$,
  and coefficient maps
\begin{equation}                                                                
  \label{eq:UnknotD}
  \mathcenter{\begin{tikzpicture}[x=2.3cm,y=48pt]
      \node at (1,1) (p) {$P$} ;
  \node at (1,-1) (q) {$Q$};
  \node at (0,0) (i) {$I$};
 \draw[->](p) to node[right] {$\rho_{23}$} (q) ;
  \draw[->](i) to node[left]{$\rho_{123}$}(p) ;
  \draw[->](q) to node[below]{$\rho_{2}$}(i) ;
\end{tikzpicture}}
\end{equation}

By the pairing theorem \cite[Theorem~11.19]{InvPair}, the
tensor product of this with the cabling type~$A$ module computes
$\CFKm(T_{n,2n-1})$.  In the tensor product, we obtain a sequence
starting with
\[ B_{i}\otimes P\overset{U^i}{\leftarrow} 
A_{i}\otimes P \overset{z^{n-i}}{\dashrightarrow} 
B_{i}\otimes Q\overset{U^{i}}{\leftarrow} 
A_{i}\otimes Q \overset{z^{n-i-1}}\dashrightarrow 
B_{i+1}\otimes P,\]
for $i=1,\dots,n-2$,
terminating at
\[ B_{n-1}\otimes P\overset{U^{n-1}}{\leftarrow} A_{n-1}\otimes
P\overset{z}{\dashrightarrow} B_{n-1}\otimes P.\] Note that the remaining
generators cancel in homology. See Figure~\ref{fig:T35Ans} for an example.
 \begin{figure}
 \begin{center}
 \input{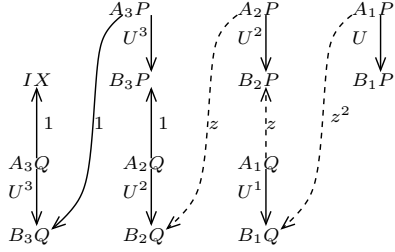}
 \caption{\textbf{Knot Floer complex of $T_{3,5}$.}
   \label{fig:T35Ans}
   The complex for $T_{3,5}$, as computed by the pairing theorem.}
 \end{center}
 \end{figure}

The generators $B_i\otimes P$ and $B_i\otimes Q$ are the ones that can 
represent torsion homology classes, since they are the ones with  even Maslov gradings.
Moreover, it follows from the above computations that for all $i=1,\dots,2n-2$, 
if $t<\frac{2}{n}$
\begin{align*}
  \gr_t(B_i\otimes P)-\gr_t(B_i\otimes Q)=2i-n t>0 \\
  \gr_t(B_i\otimes Q)-\gr_t(B_{i+1}\otimes Q)=2i-(n-1)t>0,
\end{align*}
so $B_1\otimes Q$ is a non-torsion class with maximal $\gr_t$ for $t<\frac{2}{n}$;
and $M(B_1\otimes Q)=0$ and $A(B_1\otimes Q)=(n-1)^2$.
It follows that 
\[\Upsilon_{T_{n,2n-1}}(t)=-(n-1)^2\cdot t \]
for $t\leq \frac{2}{n}$.
\end{prooff}

Putting the (partial) computations of $\Upsilon
_{C_{n,2n-1}(W_0^+(T_{2,3}))}$ and of $\Upsilon _{T_{n,2n-1}}$
together, we get

\bigskip

\begin{prooff} {\bf  of Theorem~\ref{thm:IndependenceTopSlice}.}
  Observe that $K\mapsto \Delta\Upsilon'_{K}(t)$ is a concordance
  homomorphism.  For $t\leq \frac{2}{2n-1}$, this homomorphism
  vanishes for $T_{n,2n-1}$ (by Lemma~\ref{lem:PartUpsTorusKnot});
  thus,
  \[ 
  \Delta\Upsilon'_{K_n}(t) = \left\{\begin{array}{ll}
      0 & {\text{for $t< \frac{2}{2n-1}$}} \\
      2n-1 & {\text{for $t=\frac{2}{2n-1}$.}}
    \end{array}\right.
  \]
  We can now apply Lemma~\ref{lem:Splitting} to the homomorphisms
  $\{ \frac{1}{2n-1} \Delta\Upsilon_K'(\frac{2}{2n-1})\}_{n=2}^{\infty}$
  and the knots $K_n$.
\end{prooff}

\begin{remark}
  When $n=1$, the knot $K_n$ is simply the Whitehead double of the
  trefoil.  Using Theorem~\ref{thm:WD} directly, we can see that the
  family of knots $K_n$ for all $n\geq 1$ is linearly independent. But
  for this linear independence result, we use the homomorphism
  $\frac{1}{2} \Delta\Upsilon'_{K}(1)$, as well as the
  $\frac{1}{2n-1}\Delta\Upsilon'_K(\frac{2}{2n-1})$ for $n\geq 2$.
\end{remark}

The above linear independence result can be stated in terms of the concordance genus.

\begin{corollary}
  \label{cor:ConcordanceGenusFamily}
  Let $\{a_n\}_{n=2}^{\infty}$ be a sequence of integers with finitely
  many non-zero terms. Consider the knot $K= \#_{n=2}^{\infty} a_n K_n$.
  Let 
  \[ c=\sum a_n \tau(K_n)
  = \sum a_n 
  \left(\frac{n^2-3n+2}{2}\right);\] and let $m=\max \{ n \mid a_n\neq 0\}$.
  Then,
  the concordance genus of $K$ is bounded below by
\[  
\max\{ |c|,|c+ a_m\cdot 2m+1|\}.
\]
\end{corollary}

\begin{proof}
  This is a direct application of
  Theorem~\ref{thm:BoundConcordanceGenus}, combined with computations
  in the proof of Theorem~\ref{thm:IndependenceTopSlice}.
\end{proof}

\section{Comparison with Hom's homomorphisms}
\label{sec:furtherformal}

It is interesting to compare the concordance homomorphisms constructed
here with those defined by Hom~\cite{JenHomInfinitelyGenerated}.  By a
recent result of Hom~\cite{HomEpsilonUpsilon} there are knots for which
our invariant $\Upsilon_K(t)\equiv 0$, but for which her invariant
$\epsilon$ (which she uses to construct concordance homomorphisms) is
non-zero.  We expect conversely that there are also knots $K$ with
$\Upsilon_K(t)\not\equiv 0$ but $\epsilon=0$.  In this section, we
give a formal construction which shows that there is no algebraic
obstruction to the existence of such knots.

Just like $\Upsilon_K$, Hom's homomorphisms are constructed from
invariants of (suitable) Maslov graded, Alexander filtered chain
complexes over $\Field[U]$. By construction, her homomorphisms vanish
on a particular subset of such complexes. Let us recall this set.

\begin{definition}
  Let $C$ be a Maslov graded, Alexander filtered chain complex which
  is free over $\Field[U]$. Suppose moreover that $H_*(C)\cong
  \Field[U]$, with generator in Maslov grading $0$. Let ${\mathcal
    A}(C)$ be the subcomplex of $C$ generated by all elements with
  Alexander filtration $A\leq 0$.  Let ${\mathcal A}'(C)$ be
  subcomplex of $C\otimes _{\Field [U]}\Field[U,U^{-1}]$ which is
  generated by $C\subset C\otimes _{\Field [U]}\Field [U, U^{-1}]$ and
  all elements of $C\otimes _{\Field [U]} \Field[U,U^{-1}]$ with
  $A\leq 0$.

  We say that $C$ is \defin{strongly trivial} if the map ${\mathcal
    A}(C)\to {\mathcal A'}(C)$ 
induces an isomorphism 
  \[H_*({\mathcal
    A}(C))/\Tors\to H_*({\mathcal A'}(C))/\Tors.\]  (Here $\Tors$ denotes the
  torsion part of an $\Field [U]$-module.) In the case where the
  rank of $H_*(C)$ is equal to one, this is equivalent to the condition that
  $\delta({\mathcal A}(C))=\delta(C)=\delta({\mathcal A}'(C))$, in the
  notation of Definition~\ref{def:Delta}.

  We say that $C$ is \defin{$\epsilon$-trivial} if 
  the map
  \[ {\mathcal A}(C)/U \rightarrow {\mathcal A}'(C)/U \]
on ${\mathcal A}(C)/U ={\mathcal A}(C)/(U \cdot {\mathcal A}(C))$
given by the embedding induces a non-zero map in homology.
\end{definition}

Note that the map ${\mathcal A}(C)\to {\mathcal A}'(C)$
naturally factors through $C$ itself; similarly 
${\mathcal A}(C)/U\to {\mathcal A}'(C)/U$ factors through $C/U$.
In~\cite[Definition~3.1]{HomHFKandtheSmoothConcordance} 
Hom defines 
a complex $C$ to have $\epsilon(C)=0$ if both maps
\[ H_*({\mathcal A}(C)/U)\to H_*(C/U)~\text{and}~ H_*(C/U)\to H_*({\mathcal
  A}'(C)/U) \] are non-trivial. Since, $H_*(C/U)=\Field$, this condition
is equivalent to the condition that $C$ is $\epsilon$-trivial, in the
above sense.
The relevance of strong triviality is the following:

\begin{prop}\label{prop:stronghenceUpsilonzero}
  If $C$ is strongly trivial, then  $\Upsilon_{C}(t)\equiv 0$.
\end{prop}
\begin{proof}
There are inclusions ${\mathcal A}(C_{\RRing})\subset E^t
  \subset {\mathcal A}'(C_{\RRing})$ for all $t$, hence by
  Lemma~\ref{lem:InclusionInequality} we get that $\delta ( {\mathcal
    A}(C))\leq \delta ( E^t) \leq \delta( {\mathcal A}'(C))$. 
  (Note that $\delta({\mathcal A}(C_\RRing))=\delta({\mathcal A}(C))$ by 
  Proposition~\ref{prop:SameUpsilon}.)
  Since
  $C$ is strongly trivial, it follows that $\delta ( {\mathcal
    A}(C))=\delta ( {\mathcal A}'(C))$, implying that $\delta (E^t)$
  is constant. Since for $t=0$ we have that $\delta (E^t)=0$, the
  claim of the lemma follows.
\end{proof}

\begin{prop}
  If $C$ is strongly trivial, then it is also $\epsilon$-trivial.
\end{prop}

\begin{proof}
  Consider a generator of $H_*({\mathcal A}(C))/\Tors$.  This can be
  lifted to an element $\xi$ of $H_*({\mathcal A}(C))$ which is in the
  cokernel of $U$, i.e. which injects into $H_*({\mathcal A}(C)/U)$. Call
  the image ${\widehat \xi}$. Moreover, since its image in
  $H_*({\mathcal A}'(C))$ induces a generator of $H_*({\mathcal A}'(C))/\Tors$,
  it follows that its image also injects in $H_*({\mathcal A}'(C)/U)$.
  By commutativity of the diagram
  \[ \begin{CD}
    {\mathcal A}(C)@>>> C @>>> {\mathcal A'}(C) \\
    @VVV @VVV @VVV\\
    {\mathcal A}(C)/U @>>> C/U @>>> {\mathcal A'}(C)/U,
  \end{CD}
  \]
  we conclude that ${\widehat \xi}$ is mapped non-trivially into
  $H_*({\mathcal A}'(C)/U)$, as desired.
\end{proof}

The converse of the above proposition is not true. An
$\epsilon$-trivial complex with $\Upsilon _C$ not identically zero
(hence $C$ not strongly trivial) can be given as follows.

Consider the $\Z\oplus \Z$-filtered complex $C^{\infty}$ over
$\Field[U,U^{-1}]$ 
with five generators
$a_{0,0}$, $b_{3,0}$, $c_{0,3}$, $d_{3,3}$, and $e_{1,1}$, satisfying
the grading conditions
\begin{align*}
  M(a_{0,0})&=1\\
  M(b_{3,0})&=-4\\
  M(c_{0,3})&=2\\
  M(d_{3,3})&=-3\\  
  M(e_{1,1})&=0
\end{align*}
and 
\begin{align*}
A(a_{0,0})=A(d_{3,3})= A(e_{1,1})&=0 \\
  A(c_{0,3})= -A(b_{3,0})&=3;
\end{align*}
and equipped with the 
differential
\begin{align*}
  \partial a_{0,0} & =0\\
  \partial b_{3,0} &= U^3\cdot a_{0,0} \\
  \partial c_{0,3} & =  a_{0,0} \\
  \partial d_{3,3} &= b_{3,0} +  U^3 \cdot c_{0,3}\\  
  \partial e_{1,1} &= U\cdot a_{0,0},
  \end{align*}
pictured in Figure~\ref{fig:FormalComplex}.

\begin{figure}
  \begin{center}
    \input{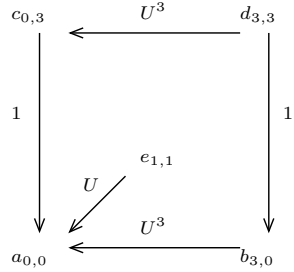}
 \caption{\textbf{An $\epsilon$-trivial chain complex with non-trivial
     $\Upsilon$}.
   Arrows represent terms in the differentials, and the labels
       represent the algebra element appearing in the corresponding term.}
 \label{fig:FormalComplex}
 \end{center}
 \end{figure}

\begin{prop}
  \label{prop:NotHom}
  The above complex is $\epsilon$-trivial, but
  \[ \Upsilon_C(t) =\left\{\begin{array}{ll}
0 &{\text{for $0\leq t\leq \frac{2}{3}$}} \\
2-3t &{\text{for $\frac{2}{3}\leq t\leq 1$}} \\
-4+3t &{\text{for $1\leq t \leq \frac{4}{3}$}} \\
0 &{\text{for $\frac{4}{3}\leq t\leq 2$.}}
\end{array}\right.\]
\end{prop}

\begin{proof}
It is easy to see that $H_*({\mathcal {A}}(C)/U)\cong \Field ^3$,
 generated by the elements $a_{0,0}, e_{1,1}$ and $b_{3,0}$.
A similar
  computation shows that $H_*({\mathcal {A}}'(C)/U)\cong \Field ^3$,
  generated by
  $d_{3,3}$, $e_{1,1}$ and
  $c_{0,3}+U^{-3}b_{3,0}$. 
The map on homology induced by the map ${\mathcal A}(C)/U
  \rightarrow {\mathcal A}'(C)/U$ maps $e_{1,1}$ to $e_{1,1}$, hence
  it is non-zero, showing that $C$ is $\epsilon$-trivial.

Since 
\begin{align*}
  \partial _ta_{0,0} & =0\\
  \partial _tb_{3,0} & =  v^{3(2-t)}\cdot a_{0,0} \\
  \partial c_{0,3} &= v^{3t}\cdot a_{0,0} \\
  \partial d_{3,3} &= v^{3(2-t)} \cdot b_{3,0} + v^{3t}\cdot c_{0,3}\\  
  \partial e_{1,1} &= v^2\cdot a_{0,0},
  \end{align*}
 we can easily see that for $t\leq \frac{2}{3}$ the elements
\[
z_1=b_{3,0}+v^{4-3t}\cdot e_{1,1} \quad \text{and} \quad 
z_2=v^{2-3t}\cdot c_{0,3}+e_{1,1}
\]
generate the torsion-free quotient, while for $t\in [\frac{2}{3},1]$ the
elements
\[
z_1=b_{3,0}+v^{4-3t}\cdot e_{1,1} \quad \text{and} \quad 
z_2'=c_{0,3}+v^{3t-2}\cdot e_{1,1}
\]
play the same role.
Since $\gr _t (z_1)=-4+3t$,  $\gr _t (z_2)=0$, and 
$\gr _t (z_2')=2-3t$, we get the desired shape of 
$\Upsilon _C(t)$ on $[0,1]$. A similar computation 
(or the symmetry $\Upsilon _C(t)=\Upsilon _C (2-t)$) then computes
$\Upsilon _C$ on $[0,2]$, concluding the proof of the proposition.
\end{proof}

\newcommand\cCFLm{\mathrm{cCFL}^-}
\newcommand\CFLt{\mathrm{tCFL}}
\newcommand\HFLt{\mathrm{tHFL}}

\section{The case of links}
\label{sec:Links}

Knot Floer homology can be generalized to links in several ways; see
for instance~\cite{OSLinks}.  There are analogous generalizations of the
$t$-modified theory to links.  We describe here one such
generalization, which will be useful in~\cite{Unorient}.

An $\ell$-component oriented link
$L=(L_1,\ldots,L_{\ell})$ can be represented by a Heegaard diagram
${\mathcal {H}}=(\Sigma,\alphas,\betas,\w,\z )$, where:
\begin{itemize}
\item 
  $\Sigma$ is a surface of genus $g$,
\item $\alphas$ and $\betas$ are $g+\ell-1$-tuples of pairwise
  disjoint, simple closed curves,
\item and the pair $(\w,\z)=\{(w_i,z_i)\}_{i=1}^{\ell}$
is an $\ell$-tuple of pairs of basepoints;
\end{itemize}
see~\cite[Section~3]{OSLinks}.
The diagram equips each
component of $L$ with an orientation; we assume that this orientation
matches with the given orientation of $L$.

The generating set of the free $\RRing$-module 
$\CFLt({\mathcal {H}})$ is  given by the intersection
points $\Gen =\Ta \cap \Tb \in \Sym ^{g+\ell -1}(\Sigma )$. 
For $\phi\in\pi_2(\x,\y)$, let
\[ n_{\ws}(\phi)=\sum_{i=1}^{\ell} n_{w_i}(\phi )\qquad\text{and}\qquad
n_{\zs}(\phi)=\sum_{i=1}^{\ell} n_{z_i}(\phi ).\] The Maslov and
Alexander functions are once again characterized up to an overall
additive shift by the equations
\begin{align}
    M(\x)-M(\y)&=\Mas(\phi)-2n_{\ws}(\phi), \label{eq:MaslovFormulaLinks}\\
    A(\x)-A(\y)&=n_{\zs}(\phi)-n_{\ws}(\phi), \label{eq:AlexanderFormulaLinks} 
  \end{align}
for any $\phi\in\pi_2(\x,\y)$.
Before pinning down the additive indeterminacy on these functions, 
we  consider the differential on $\CFLt({\mathcal {H}})$:
\begin{equation}
  \label{eq:tModDiffLinks}
  \partial_t \x = \sum_{\y\in\Gen}\sum_{\{\phi\in\pi_2(\x,\y)\big|\Mas(\phi)=1\}} \# \left(\frac{\ModFlow(\phi)}{\mathbb R}\right)
  v^{t n_{\zs}(\phi)+ (2-t) n_{\ws}(\phi)} \y.
\end{equation}

\begin{lemma}
  \label{lem:TZeroLink}
  The homology of the $t=0$ 
  specialization of the above complex is a free $\Ring$-module
  of rank $2^{\ell-1}$. In fact, up to an overall shift in gradings,
  there is a graded isomorphism
  \begin{equation}
    \label{eq:TZero}
    H_*({\mathrm{tCFL}}\vert _{t=0}({\mathcal {H}}))\cong
    (\Ring_{-\OneHalf}\oplus\Ring_{\OneHalf})^{\ell-1}.
  \end{equation}
  The same holds when $t=2$:
  \begin{equation}
    \label{eq:TTwo}
    H_*({\mathrm{tCFL}}\vert _{t=2}({\mathcal {H}}))\cong
    (\Ring_{-\OneHalf}\oplus\Ring_{\OneHalf})^{\ell-1}.
  \end{equation}
\end{lemma}

\begin{proof}
  The $t=0$ specialization is independent of the placement of $\z$.
(This specialization is equipped with the Maslov grading.)
Consider the chain complex
  over $\Field[U_1,\dots,U_\ell]$ with the same generators as before,
  but with differential specified by
  \[
    \partial \x = \sum_{\y\in\Gen}\sum_{\{\phi\in\pi_2(\x,\y)\big|\Mas(\phi)=1\}} \# \left(\frac{\ModFlow(\phi)}{\mathbb R}\right)
    U_1^{n_{w_1}(\phi)}\cdots U_\ell^{n_{w_{\ell}}(\phi)} \y.
    \]
  According to~\cite[Theorem~4.4]{OSLinks}, this chain complex
  computes $\HFm (S^3)\cong \Field[U]$, where all the $U_i$ act as
  translations by $U$. Thus, if we set them equal to one another, the
  resulting complex is $\Field[U]\otimes V^{\ell-1}$.  The $t=0$
  complex is gotten by changing the base ring to $\RRing$ with
  variable $v$ (and with the understanding of $U=v^2$), equipped 
with the Maslov grading.

  The $t=2$ specialization is independent of the placement of  $\w$
  (even in the defintion of $\gr_{2}=M-2A$). 
\end{proof}

\begin{definition}
  Let $L$ be an oriented link.  Eliminate the additive indeterminacy in $M$ by
  the requirement that Equation~\eqref{eq:TZero} holds without shifting the
  grading.  Next, eliminate the additive indeterminacy in $A$ by the
  requirement that Equation~\eqref{eq:TTwo} holds without shifting the
  grading.  Using these normalizations, we define the grading $\gr _t$ on the
  generator $\x$ of $\CFLt({\mathcal{H}})$ by the usual formula
\[
\gr _t (\x )=M(\x )-tA(\x ),
\]
and extend it to the $\RRing$-module by $\gr _t (v^{\alpha }\x )=\gr
_t (\x )-\alpha$. The homology  
$\HFLt({\mathcal{H}})$ of the resulting graded chain complex
is a graded $\RRing$-moduli, called  the \defin{$t$-modified link
  homology} of $L$.
\end{definition}

We have the following analogue of Theorem~\ref{thm:Invariance}:

\begin{theorem}
  \label{thm:InvarianceForLinks}
  The $t$-modified link homology $\HFLt ({\mathcal {H}})$ of the Heegaard
  diagram ${\mathcal{H}}$ is an invariant of the underlying oriented link $L$,
and is denoted by $\HFLt (L)$.
\end{theorem}

\begin{proof}
  In~\cite{OSLinks}, the link complex $\CFLm({\mathcal{H}})$ 
  is a $\Z^{\ell}$-filtered chain
  complex which is also equipped with a Maslov grading.  The Alexander
  multi-grading is specified by the vector
  \[{\mathbf A}=(A_1,\dots,A_\ell),\]
  and the underlying algebra is $\Field[U_1,\dots,U_\ell]$.
  
  We can specialize the link complex by setting $U_1=\dots=U_\ell$ to
  get a variant which is defined over $\Field[U]$, endowed with the
  $\Z$-filtration $A=\sum_{i=1}^{\ell} A_i$ (specified again 
 up to an overall shift). We call the resulting complex the
  {\em{algebraically collapsed link complex}} $\cCFLm({\mathcal{H}})$.
  It follows that
  $\CFLt({\mathcal{H}})$ is simply the $t$-modification (in the sense
  of Section~\ref{sec:Formal}) of the algebraically collapsed link complex.

  It is easy to see that $\Z^{\ell}$-filtered homotopy equivalences
  between $\CFLm({\mathcal{H}})$'s for different Heegaard
    diagrams representing $L$ induce $\Z$-filtered homotopy
  equivalences of the corresponding collapsed complex. Thus, since the
  filtered homotopy type of $\CFLm({\mathcal{H}})$ is a link
  invariant~\cite[Theorem~4.7]{OSLinks}, functoriality of the
  $t$-modification (Proposition~\ref{prop:tModifyFunctor}) implies the
  result.
\end{proof}
  
The definition of the knot invariant $\Upsilon _K (t)$ extends to links
as follows.

\begin{definition}
  For $t\in[0,2]$ choose a homogeneous basis
  $\{e_i(t)\}_{i=1}^n$ for the free $\Ring$-module
  $\HFLt(L)/\Tors$. The \defin{$\Upsilon$-set of the oriented link $L$
    at $t$} is the set $\{\gr_t(e_i(t))\}_{i=1}^n$ (a set with
    possible repetitions).
\end{definition}

\begin{theorem}
  The $\Upsilon$-set of $L$ at any $t\in[0,2]$ is a set with
  $2^{\ell-1}$ elements (counted with repetitions). It is an invariant of the
  oriented link $L$.
\end{theorem}

\begin{proof}
  The proof consists of two parts. First, we must show that
  $\HFLt(L)/\Tors$ is a free module of rank $2^{\ell-1}$. Second, we
  must show that the set is a link invariant.
  
  To see that $\HFLt(L)/\Tors$ is a free module of rank $2^{\ell-1}$,
  we use the fact that $\HFKinf(L)\cong \Field[U,U^{-1}]\otimes (\Field \oplus \Field )^{\otimes (\ell -1)}$.
  This follows from an application of~\cite[Theorem~4.7]{OSLinks}
  (exactly as in the proof of Lemma~\ref{lem:TZeroLink}).

  Invariance follows immediately from
  Theorem~\ref{thm:InvarianceForLinks}.
\end{proof}

\bibliographystyle{plain}
\bibliography{biblio}

\begin{thebibliography}{10}

\bibitem{brandal}
W.~Brandal.
\newblock {\em Commutative rings whose finitely generated modules decompose},
  volume 723 of {\em Lecture Notes in Mathematics}.
\newblock Springer, Berlin, 1979.

\bibitem{CassonGordon}
A.~Casson and C.~Gordon.
\newblock On slice knots in dimension three.
\newblock In {\em Algebraic and geometric topology ({P}roc. {S}ympos. {P}ure
  {M}ath., {S}tanford {U}niv., {S}tanford, {C}alif., 1976), {P}art 2}, Proc.
  Sympos. Pure Math., XXXII, pages 39--53. Amer. Math. Soc., Providence, R.I.,
  1978.

\bibitem{Chen}
W.~Chen.
\newblock On the upsilon invariant of cable knots.
\newblock arXiv:1604.04760.

\bibitem{Eftekhary}
E.~Eftekhary.
\newblock Longitude {F}loer homology and the {W}hitehead double.
\newblock {\em Algebr. Geom. Topol.}, 5:1389--1418 (electronic), 2005.

\bibitem{FellKrat}
P.~Feller and D.~Krcatovich.
\newblock On cobordisms between knots, braid index, and the upsilon-invariant.
\newblock arXiv:1602.02637.

\bibitem{Freedman}
M.~Freedman.
\newblock A surgery sequence in dimension four;\ the relations with knot
  concordance.
\newblock {\em Invent. Math.}, 68(2):195--226, 1982.

\bibitem{FuchsSalce}
L.~Fuchs and L.~Salce.
\newblock {\em Modules over valuation domains}, volume~97 of {\em Lecture Notes
  in Pure and Applied Mathematics}.
\newblock Marcel Dekker Inc., New York, 1985.

\bibitem{HeddenCables}
M.~Hedden.
\newblock On knot {F}loer homology and cabling.
\newblock {\em Alg. Geom. Topol.}, 5:1197--1222, 2005.

\bibitem{HeddenWhitehead}
M.~Hedden.
\newblock Knot {F}loer homology of {W}hitehead doubles.
\newblock {\em Geom. Topol.}, 11:2277--2338, 2007.

\bibitem{HeddenCablesII}
M.~Hedden.
\newblock On knot {F}loer homology and cabling. {II}.
\newblock {\em Int. Math. Res. Not. IMRN}, (12):2248--2274, 2009.

\bibitem{JenHomInfinitelyGenerated}
J.~Hom.
\newblock An infinite rank summand of topologically slice knots.
\newblock arXiv:1310.4476.

\bibitem{HomConcordanceGenus}
J.~Hom.
\newblock On the concordance genus of topologically slice knots.
\newblock arXiv:1203.459.

\bibitem{HomLSpace}
J.~Hom.
\newblock A note on cabling and {$L$}-space surgeries.
\newblock {\em Algebr. Geom. Topol.}, 11(1):219--223, 2011.

\bibitem{HomCables}
J.~Hom.
\newblock Bordered {H}eegaard {F}loer homology and the tau-invariant of cable
  knots.
\newblock {\em J. Topol.}, 7(2):287--326, 2014.

\bibitem{HomHFKandtheSmoothConcordance}
J.~Hom.
\newblock The knot {F}loer complex and the smooth concordance group.
\newblock {\em Comment. Math. Helv.}, 89(3):537--570, 2014.

\bibitem{HomEpsilonUpsilon}
J.~Hom.
\newblock A note on the concordance invariants epsilon and upsilon.
\newblock {\em Proc. Amer. Math. Soc.}, 144(2):897--902, 2016.

\bibitem{HomWu}
J.~Hom and Z.~Wu.
\newblock Four-ball genus bounds and a refinement of the {O}zsv\'ath-{S}zab\'o
  tau invariant.
\newblock {\em J. Symplectic Geom.}, 14(1):305--323, 2016.

\bibitem{InvPair}
R.~Lipshitz, P.~Ozsv\'ath, and D.~Thurston.
\newblock Bordered {F}loer homology: Invariance and pairing.
\newblock arXiv:0810.0687.

\bibitem{PNAS}
R.~Lipshitz, P.~S. Ozsv{\'a}th, and D.~P. Thurston.
\newblock Tour of bordered {F}loer theory.
\newblock {\em Proc. Natl. Acad. Sci. USA}, 108(20):8085--8092, 2011.

\bibitem{Litherland}
R.~Litherland.
\newblock Signatures of iterated torus knots.
\newblock In {\em Topology of low-dimensional manifolds ({P}roc. {S}econd
  {S}ussex {C}onf., {C}helwood {G}ate, 1977)}, volume 722 of {\em Lecture Notes
  in Math.}, pages 71--84. Springer, Berlin, 1979.

\bibitem{Living}
C.~Livingston.
\newblock Notes on the knot concordance invariant upsilon.
\newblock {\em Algebr. Geom. Topol.}, 17(1):111--130, 2017.

\bibitem{QuasiAlternating}
C.~Manolescu and P.~Ozsv{\'a}th.
\newblock On the {K}hovanov and knot {F}loer homologies of quasi-alternating
  links.
\newblock In {\em Proceedings of {G}\"okova {G}eometry-{T}opology {C}onference
  2007}, pages 60--81. G\"okova Geometry/Topology Conference (GGT), G\"okova,
  2008.

\bibitem{Murasugi}
K.~Murasugi.
\newblock On the {A}lexander polynomial of alternating algebraic knots.
\newblock {\em J. Austral. Math. Soc. Ser. A}, 39(3):317--333, 1985.

\bibitem{Unorient}
P.~Ozsv\'ath, A.~Stipsicz, and Z.~Szab\'o.
\newblock Unoriented knot {F}loer homology and the unoriented four-ball genus.
\newblock Int. Math. Res. Notices, 2016. to appear.

\bibitem{Gridbook}
P.~Ozsv\'ath, A.~Stipsicz, and Z.~Szab\'o.
\newblock {\em Grid homology for knots and links}, volume 208 of {\em
  Mathematical Surveys and Monographs}.
\newblock American Mathematical Society, Providence, RI, 2015.

\bibitem{AbsGraded}
P.~Ozsv{\'a}th and Z.~Szab{\'o}.
\newblock Absolutely graded {F}loer homologies and intersection forms for
  four-manifolds with boundary.
\newblock {\em Advances in Mathematics}, 173(2):179--261, 2003.

\bibitem{AltKnots}
P.~Ozsv{\'a}th and Z.~Szab{\'o}.
\newblock Heegaard {F}loer homology and alternating knots.
\newblock {\em Geom. Topol.}, 7:225--254, 2003.

\bibitem{FourBall}
P.~Ozsv{\'a}th and Z.~Szab{\'o}.
\newblock Knot {F}loer homology and the four-ball genus.
\newblock {\em Geom. Topol.}, 7:615--639, 2003.

\bibitem{GenusBounds}
P.~Ozsv{\'a}th and Z.~Szab{\'o}.
\newblock Holomorphic disks and genus bounds.
\newblock {\em Geom. Topol.}, 8:311--334, 2004.

\bibitem{OSKnots}
P.~Ozsv{\'a}th and Z.~Szab{\'o}.
\newblock Holomorphic disks and knot invariants.
\newblock {\em Adv. Math.}, 186(1):58--116, 2004.

\bibitem{HolDiskTwo}
P.~Ozsv{\'a}th and Z.~Szab{\'o}.
\newblock Holomorphic disks and three-manifold invariants: properties and
  applications.
\newblock {\em Ann. of Math. (2)}, 159(3):1159--1245, 2004.

\bibitem{HolDisk}
P.~Ozsv{\'a}th and Z.~Szab{\'o}.
\newblock Holomorphic disks and topological invariants for closed
  three-manifolds.
\newblock {\em Ann. of Math. (2)}, 159(3):1027--1158, 2004.

\bibitem{NoteLens}
P.~Ozsv{\'a}th and Z.~Szab{\'o}.
\newblock On knot {F}loer homology and lens space surgeries.
\newblock {\em Topology}, 44(6):1281--1300, 2005.

\bibitem{OSLinks}
P.~Ozsv{\'a}th and Z.~Szab{\'o}.
\newblock Holomorphic disks, link invariants and the multi-variable {A}lexander
  polynomial.
\newblock {\em Algebr. Geom. Topol.}, 8(2):615--692, 2008.

\bibitem{Trans}
P.~Ozsv{\'a}th, Z.~Szab{\'o}, and D.~Thurston.
\newblock Legendrian knots, transverse knots and combinatorial {F}loer
  homology.
\newblock {\em Geom. Topol.}, 12(2):941--980, 2008.

\bibitem{Petkova}
I.~Petkova.
\newblock Cables of thin knots and bordered {H}eegaard {F}loer homology.
\newblock {\em Quantum Topol.}, 4(4):377--409, 2013.

\bibitem{RasmussenThesis}
J.~Rasmussen.
\newblock {\em Floer homology and knot complements}.
\newblock PhD thesis, Harvard University, 2003.

\bibitem{JakeOtherPaper}
J.~Rasmussen.
\newblock Lens space surgeries and a conjecture of {G}oda and {T}eragaito.
\newblock {\em Geom. Topol.}, 8:1013--1031, 2004.

\bibitem{Jake}
J.~Rasmussen.
\newblock Khovanov homology and the slice genus.
\newblock {\em Invent. Math.}, 182(2):419--447, 2010.

\bibitem{Tristram}
A.~Tristram.
\newblock Some cobordism invariants for links.
\newblock {\em Proc. Cambridge Philos. Soc.}, 66:251--264, 1969.

\end{thebibliography}

\end{document}